\documentclass[12pt,reqno]{amsart}
\usepackage[headings]{fullpage}
\usepackage{amssymb,amsmath,amscd,bbm,tikz}
\usepackage{graphicx}
\usepackage{texdraw}
\usepackage{pb-diagram}
\usepackage[all,cmtip]{xy}
\usepackage{url}
\usepackage[bookmarks=true,%
    colorlinks=true,%
    linkcolor=blue,%
    citecolor=blue,%
    filecolor=blue,%
    menucolor=blue,%
    urlcolor=blue,%
    breaklinks=true]{hyperref}
\usepackage{slashed}    
\usepackage{tikz-cd}
\usepackage{stmaryrd}

\usepackage{todonotes}
\usepackage{enumerate}
\usepackage{mathtools}

\usepackage{arydshln}

\usepackage{verbatim} 

\newtheorem{theorem}{Theorem}[section]
\theoremstyle{definition}
\newtheorem{proposition}[theorem]{Proposition}
\newtheorem{lemma}[theorem]{Lemma}

\newtheorem{corollary}[theorem]{Corollary}

\def\BZ{\mathbbm Z}
\def\BQ{\mathbbm Q}
\def\BR{\mathbbm R}
\def\BC{\mathbbm C}

\def\calA{\mathcal A}

\def\calT{\mathcal T}

\def\calB{\mathcal B}

\def\SL{\mathrm{SL}}

\def\ve{\varepsilon}

\def\PSL{\mathrm{PSL}}

\def\Re{\mathrm{Re}}

\def\be{\begin{equation}}
\def\ee{\end{equation}}

\newcommand{\fad}{\operatorname{\Phi}_{\mathsf{b}}}

\def\bb{\mathsf{b}}

\def\z{\zeta}

\def\GL{\mathrm{GL}}

\def\diag{\mathrm{diag}}

\usepackage{accents}

\def\hb{\hbar}
\def\vphi{\varphi}
\def\geom{\mathrm{geom}}
\def\mbfQ{\mathbf Q}
\def\tz{\hat z}
\def\tdelta{\hat \delta}
\def\llangle{\left\langle\!\left\langle}
\def\rrangle{\right\rangle\!\right\rangle}

\newcommand{\bea}{\begin{equation}\begin{aligned}}
\newcommand{\eea}{\end{aligned}\end{equation}}
\renewcommand{\=}{\;=\;}

\newcommand{\tq}{\widetilde q}

\newcommand{\Z}{\mathbb{Z}}
\newcommand{\Li}{\operatorname{Li}}

\newcommand{\hroot}{\hbar^{\frac{1}{2}}}
\newcommand{\hsqrt}{\hbar^{\frac{1}{2}}}
\newcommand{\ezh}{e^{z\hbar^{\frac 12}}}

\makeatletter

\renewcommand\thepart{\@Roman\c@part}%
\renewcommand\part{%
   \if@noskipsec \leavevmode \fi
   \par
   \addvspace{6.7ex}%
   \@afterindentfalse
   \secdef\@part\@spart}
\def\@part[#1]#2{%
    \ifnum \c@secnumdepth >\m@ne
      \refstepcounter{part}%
      \addcontentsline{toc}{part}{Part~\thepart.\ #1}%
    \else
      \addcontentsline{toc}{part}{#1}%
    \fi
    {\parindent \z@ \raggedright
     \interlinepenalty \@M
     \normalfont
     \ifnum \c@secnumdepth >\m@ne
       \centering\large\scshape \partname~\thepart.%
       \hspace{1ex}%
     \fi%
     \large\scshape #2%
     \markboth{}{}\par}%
    \nobreak
    \vskip 4.7ex
    \@afterheading}
  \def\@spart#1{
  \refstepcounter{part}%
  \addcontentsline{toc}{part}{#1}%
    {\parindent \z@ \raggedright
     \interlinepenalty \@M
     \normalfont
     \centering\large\scshape #1\par}%
     \nobreak
     \vskip 4.7ex
     \@afterheading}
\renewcommand*\l@part[2]{%
  \ifnum \c@tocdepth >-2\relax
    \addpenalty\@secpenalty
    \addvspace{0.75em \@plus\p@}%
    \begingroup
      \parindent \z@ \rightskip \@pnumwidth
      \parfillskip -\@pnumwidth
      {\leavevmode
       \normalsize \bfseries #1\hfil \hb@xt@\@pnumwidth{\hss #2}}\par
       \nobreak
       \if@compatibility
         \global\@nobreaktrue
         \everypar{\global\@nobreakfalse\everypar{}}%
      \fi
    \endgroup
  \fi}

\def\l@subsection{\@tocline{2}{0pt}{2pc}{6pc}{}}
\makeatother

\begin{document}
% \title{Topological invariance of Chern-Simons perturbation theory}
% \title{Topological invariance of complex Chern-Simons and
% Teichm\"uller TQFT perturbation theory}
\title[Perturbative invariants of cusped hyperbolic 3-manifolds]{
  Perturbative invariants of cusped hyperbolic 3-manifolds}

\author{Stavros Garoufalidis}
\address{% Max Planck Institute for Mathematics \\
         % Vivatsgasse 7, 53111 Bonn, GERMANY \newline
  International Center for Mathematics, Department of Mathematics \\
  Southern University of Science and Technology \\
  Shenzhen, China \newline
  {\tt \url{http://people.mpim-bonn.mpg.de/stavros}}}
\email{stavros@mpim-bonn.mpg.de}
\author{Matthias Storzer}
\address{Max Planck Institute for Mathematics \\
         Vivatsgasse 7, 53111 Bonn, Germany \newline
         {\tt \url{http://guests.mpim-bonn.mpg.de/storzer}}}
\email{storzer@mpim-bonn.mpg.de}
\author{Campbell Wheeler}
\address{Max Planck Institute for Mathematics \\
         Vivatsgasse 7, 53111 Bonn, Germany \newline
         {\tt \url{http://guests.mpim-bonn.mpg.de/cjwh}}}
\email{cjwh@mpim-bonn.mpg.de}

\date{11 July 2023}%23 May 2023}%\today }
%\dedicatory{}

\begin{abstract}
  We prove that a formal power series associated to an ideally triangulated cusped
  hyperbolic 3-manifold (together with some further choices) is a topological
  invariant. This formal power series is conjectured to agree to
  all orders in perturbation theory with two important topological invariants of
  hyperbolic knots, namely the Kashaev invariant and the Andersen--Kashaev
  invariant (also known as the state-integral) of Teichm\"uller TQFT.
\end{abstract}

\maketitle

{\footnotesize
\tableofcontents
}

%%%%%%%%%%%%%%%%%%%%%%%%%%%%%%%%%%%%%%%%%%%%%%%%%%%%%%%%%%%%%%%%%%%%%%%%%%%% 
%%%%%%%%%%%%%%%%%%%%%%%%%%%%%%%%%%%%%%%%%%%%%%%%%%%%%%%%%%%%%%%%%%%%%%%%%%%%

\section{Introduction}
\label{sec.intro}

This paper concerns the topological invariance of a formal power series associated
to an ideally triangulated 3-manifold $M$ with a torus boundary component~\cite{DG}. 
The series is defined by formal Gaussian integration of a finite dimensional integral
(a so-called ``state-integral'')
and it is expected to coincide to all orders with the asymptotic expansion of
three important topological invariants of 3-manifolds.

The first is the Kashaev invariant of a hyperbolic
knot~\cite{K95}, where Kashaev's famous volume conjecture is refined to an
asymptotic statement to all orders in perturbation theory using the above formal
power series. This was studied in detail in~\cite{GZ:kashaev} based on extensive
numerical computations, but it is only proven for a handful of hyperbolic knots.

The second invariant is the Andersen--Kashaev state integral~\cite{AK}, whose
asymptotic expansion for the simplest hyperbolic $4_1$ knot was shown to agree with
the above series in~\cite[Sec.12]{AK}, and also observed numerically for the
first three simplest hyperbolic knots in~\cite{GZ:qseries}.
The state-integrals of~\cite{AK} are finite-dimensional integrals whose integrand
is a product of Faddeev quantum dilogarithms times an exponential of a quadratic
form, assembled from an ideal triangulation of a 3-manifold with torus boundary
components. Andersen--Kashaev proved that these are topological invariants that
are the partition function of the Teichm\"uller TQFT~\cite{AK,AK:teichmuller},
which is a 3-dimensional version of a quantization of Teichm\"uller
space~\cite{Kashaev:quantization, Fock-Chekhov}. 

A third invariant is the Chern--Simons theory with complex gauge group $\SL_2(\BC)$ .
This theory was introduced by Witten~\cite{Witten:complexCS} and studied
extensively by Gukov~\cite{Gukov:Apoly}. Although Chern--Simons theory with compact
gauge group $\mathrm{SU}(2)$ has an exact nonperturbative definition given by the
so-called Witten--Reshetikhin--Turaev invariant~\cite{Witten:Jones,RT} and a
well--defined perturbation theory involving Feynman diagrams with uni-trivalent
vertices~\cite{AS,AarhusII}, the same is not known for Chern--Simons theory with
complex gauge group $\SL_2(\BC)$. For reasons that are not entirely
understood, the partition function of complex Chern--Simons theory for 3-manifolds
with torus boundary reduces to a finite-dimensional state-integral, as if some
unknown localization principle holds. The corresponding state-integrals were
introduced and studied by Hikami~\cite{Hikami}, Dimofte~\cite{Dimofte:complexCS}
and others~\cite{DGLZ}. Unfortunately, in those works the integration contour
was not pinned down, and this problem was finally dealt with in~\cite{AK}
and, among other things, implied topological invariance of the state-integrals,
which were coined to be the partition function of Teichm\"uller TQFT.
But ignoring the contour of integration, and focusing on a critical point of the
action, which is a solution to a system of polynomial equations, allowed~\cite{DG}
to give a definition of the formal power series that are the main focus of our paper.  
Note, however, that the Feynman diagrams of~\cite{DG} involve stable graphs of
arbitrary valency, a perturbative definition of Chern--Simons theory with complex
gauge groups would involve uni-trivalent graphs. 

The above discussion points out several aspects of these formal power series
and conjectural relations to perturbation theory of complex Chern--Simons
and of Teich\"uller TQFT. Aside from conjectures, this paper concerns
a theorem, the topological invariance of the above formal power series.
  
Let us briefly recall the key ingredients that go into the definition of the
series, and discuss those in detail in later sections. The first ingredient is  
an ideal triangulation $\calT$ with $N$ ideal tetrahedra of a 3-manifold $M$ with
torus boundary. Each tetrahedron has shapes $z \in \BC\setminus\{0,1\}$, $z'=1/(1-z)$
and $z''=1-1/z$ attached to its three pairs of opposite edges, and the shapes
satisfy a system of polynomial equations (so-called ``gluing equations''~\cite{Th} or
Neumann-Zagier equations~\cite{NZ}) determined by the combinatorics
of the triangulation, one for each edge and peripheral curve. After some choices
are made (such as an ordering of the tetrahedra and their edges, a choice of shape
for each tetrahedron, a choice of an edge to remove from the gluing equations, and
a choice of peripheral curve to include), one obtains two matrices $A$ and $B$
with integer entries such that $(A\,|\,B)$ is the upper-half of a symplectic $2N \times
2N$ matrix, as well as a vector $\nu \in \BZ^N$. In addition, we choose a solution
$z=(z_1,\dots,z_N)$ of the gluing equations as well as a flattening $(f,f'')$, i.e.,
an integer solution to the equation $\nu = A f + B f''$. The power series
$\Phi^\Xi(\hbar)$ depends on the tuple $\Xi=(A,B,\nu,z,f,f'')$, which we collectively
call a NZ-datum. 

The next ingredient that goes in the definition of $\Phi^\Xi(\hbar)$ is an
auxiliary formal power series
% h;w;z;
% Psi(x,h,N)=-sum(k=0,N,sum(ell=0,N,if(k+ell/2>1,bernfrac(k)
% *x^ell*h^(2*k+ell-2)/k!/ell!*polylog(2-k-ell,z),0)))
% PsiDG(x,h,N)=sum(k=0,N,sum(ell=0,N,if(k+ell/2>1,(-1)^k*bernfrac(k)
% *(-x)^ell*h^(2*k+ell-2)/k!/ell!*polylog(2-k-ell,1/z),0)))
% Psi(w,h+O(h^20),20)-PsiDG(w,h+O(h^20),20)
% /* = -1/2*w*h + 1/12*h^2 + O(h^20) */
\be
\label{psih}
  \psi_{\hbar}(x,z)
  \;:=\;
  \exp\bigg(\!-\!\!\sum_{\underset{k+\frac{\ell}{2}>1}{k,\ell\in\BZ_{\geq0}}}
  \frac{B_{k}\,x^{\ell}\,\hbar^{k+\frac{\ell}{2}-1}}{\ell!\,k!}\Li_{2-k-\ell}(z)\bigg)
  \in1+\hbar^{\frac{1}{2}}\BQ[z,(1-z)^{-1}][x]\llbracket\hbar^{\frac{1}{2}}
  \rrbracket\,.
\ee
This series, which differs from the one given in~\cite[Eq. 1.9]{DG} by a factor of
$\exp\big(\frac{1}{12}\hbar-\frac{1}{2}x\hbar^{\frac{1}{2}}\big)$, comes from
the asymptotics of the infinite Pochhammer symbol (also known as the
quantum dilogarithm)
\be
\label{poch}
(x;q)_\infty = \prod_{k=0}^\infty (1-q^k x) = \exp\Big(
-\sum_{n=1}^\infty \frac{x^n}{n(1-q^n)} \Big) \,.
\ee
Explicitly, for complex numbers $x$ and $z$ with $|z|<1$ and
$\Re(\hbar)<0$ we have (see~\cite{Zagier:dilog} and
also~\cite[Lem.2.1]{GZ:asymptotics})
\be
\label{eq.qpoch.asy.psi}
  (ze^{x\hbar^{1/2}};e^{\hbar})_{\infty}^{-1}
  \;\sim\;
  \frac{1}{\sqrt{1-z}}\exp\Big(-\frac{\Li_{2}(z)}{\hbar}
  -\frac{\Li_{1}(z)x}{\hbar^{\frac{1}{2}}}
  -\frac{\Li_{0}(z)x^2}{2}\Big)
  \psi_{\hbar}(x,z)\,, \qquad (\hbar\rightarrow 0) \,.
\ee
Assuming that the matrix $B$ is invertible, (i.e., that $\det(B) \neq 0$),
we introduce the function
\be
\label{integrand}
f^\Xi_{\hb}(x,z) \=
\exp \bigg(\frac{\hbar^{\frac12}}{2} x^t1-\frac{\hbar^{\frac12}}{2} x^t  B^{-1} \nu
+\frac\hbar8 f^t B^{-1} A f \bigg) 
\prod_{i=1}^N \psi_{\hb}(x_i,z_i) \in \BQ(z)[x]\llbracket \hbar^{\frac12} \rrbracket\,,
\ee
which is the product of one quantum dilogarithm per tetrahedron (each
with its own integration variable), 
with some additional terms coming from the NZ-datum
where $x=(x_1,...,x_N)^t$ and $z=(z_1,...,z_N)$.

The last ingredient is the formal Gaussian integration~\cite{BIZ}  
\begin{small}
\bea
\label{FGI}
\left\langle
f_\hbar(x,z)
\right\rangle_{x,\Lambda}
\;:=\; \exp\left(
{\frac 1 2 \sum_{i,j=1}^N (\Lambda^{-1})_{i,j}
\frac{\partial}{\partial x_i}
\frac{\partial}{\partial x_j}}\right)
f_\hbar(x,z) \Big|_{x=0} =
\frac{\int e^{-\frac12x^t \Lambda \,x}f_{\hb,z}(x) \, dx}{
  \int e^{-\frac12x^t \Lambda \,x} \, dx}
\in \BQ(z)\llbracket \hbar \rrbracket
\eea
\end{small}
of a power series $f_\hbar(x,z) \in \BQ(z)[x]\llbracket\hroot \rrbracket$, with
respect to an invertible matrix $\Lambda$. Assuming further that the symmetric
matrix
\be
\label{lambda}
\Lambda \= - B^{-1} A + \diag(1/(1-z_j))
\ee
is invertible, we define 

\be
\label{Phidef}
\Phi^\Xi(\hbar) \= \langle f^\Xi_{\hbar}(x,z) \rangle_{x, \Lambda}
\in \BQ(z)\llbracket \hbar \rrbracket\,.
\ee

Using a solution $z$ of the Neumann--Zagier equations that corresponds to
the discrete faithful representation $\rho^\geom: \pi_1(M) \to \PSL_2(\BC)$
of a cusped hyperbolic 3-manifold $M$, one obtains a series $\Phi^\Xi(\hbar)$
with coefficients in the invariant trace field of $M$.
Our main theorem is the following.

\begin{theorem}
\label{thm.1}
$\Phi^\Xi(\hbar)$ is a topological invariant of a cusped hyperbolic 3-manifold.
\end{theorem}

A corollary of the above theorem is that every coefficient of the above
series is a topological invariant of a cusped hyperbolic 3-manifold. These
topological invariants have a geometric origin, since they take values in
the invariant trace field of the manifold, and hence they are algebraic, but
not (in general) rational, numbers. This series seems to come from 
hyperbolic geometry (though a definition of these invariants in terms of the
hyperbolic geometry of the 3-manifold is not known), and perhaps not from enumerative
quantum field theory (such as the Gromov--Witten or any of the 4-dimensional
known theories). The behavior of these series under finite cyclic coverings of
cusped hyperbolic 3-manifolds is given in~\cite{GY:1loop.torsion}, and the
formulas presented there (e.g., for the coefficient of $\hbar^2$) point to
unknown connections with the spectral theory of hyperbolic 3-manifolds.

The proof of our theorem combines the ideas of the topological invariance of the
Andersen--Kashaev state-integrals~\cite{AK} with those of the Aarhus
integral~\cite{AarhusII}. We briefly recall that the building block for the
state-integral is the Faddeev quantum
dilogarithm, the state-integral is a finite-dimensional integral of a product of
Faddeev quantum dilogarithm~\cite{Faddeev}, one for each tetrahedron of an (ordered)
ideal triangulation. Aside from elementary choices, two important parts
in the proof of topological invariance of the state-integral is
invariance under (a) the choice of a nondegenerate quad, and
(b) 2--3 ordered Pachner moves, the latter connecting one ideal triangulation
with another. In~\cite{AK}, (a) and (b) are dealt with a Fourier transform and
a pentagon identity for the Faddeev quantum dilogarithm.

The power series $\Phi^\Xi(\hbar)$ is given instead by a formal Gaussian integral
(as opposed to an integral over Euclidean space), where the building block is
the formal power series $\psi_\hbar$ instead of the Faddeev
quantum dilogarithm. The topological invariance of the $\Phi^\Xi(\hbar)$ under
the choice of quad and under the 2--3 Pachner moves follows from a Fourier tranform
identity and a pentagon identity for $\psi_\hbar$; see Theorems~\ref{thm.fourier}
and~\ref{thm.pentagon} below. Although these identities are, in a sense, limits of
the corresponding identities for the Faddeev quantum dilogarithm (just as
$\psi_\hbar$ is a limit of the Faddeev quantum dilogarithm), it would requires
additional analytic work to do so, and instead we give algebraic proofs of
theorems~\ref{thm.fourier} and~\ref{thm.pentagon} using properties of the
formal Gaussian integration, together with holonomic properties of the involved
formal power series.

Having discussed the similarities between the proof of Theorem~\ref{thm.1} and
the corresponding theorem for the state-integral of~\cite{AK}, we now point out
some differences. In~\cite{AK}, Andersen--Kashaev use ordered triangulations,
and the state-integral is
obtained by the push-forward of a distribution with variables at the faces and
the tetrahedra of the ordered triangulation. Part of the distribution is a product
of delta functions in linear forms of the face-variables. In our Theorem~\ref{thm.1},
and in the formal Gaussian integration, we carefully avoided the need to use
delta functions, although such a reformulation of our results are possible, with
additional effort.

We end this section by pointing out a wider context for the asymptotic series
$\Phi^\Xi(\hbar)$ and Theorem~\ref{thm.1}. It was clear from~\cite{DG} that
a NZ-datum $\Xi=(A,B,\nu,z,f,f'')$ depends on two square matrices $A$ and $B$
such that $AB^t$ is symmetric that may or may not come from topology, and that
the series $\Phi^\Xi(\hbar)$ is defined under the assumption that $\det(B) \neq 0$
and $\det(\Lambda) \neq 0$. Doing so, the proof of Theorem~\ref{thm.1} shows that
the series $\Phi^\Xi(\hbar)$ is invariant under the moves that appear
in~\cite[Sec.6]{GZ:kashaev}; see also Section~\ref{sec.elementary}. 

%%%%%%%%%%%%%%%%%%%%%%%%%%%%%%%%%%%%%%%%%%%%%%%%%%%%%%%%%%%%%%%%%%%%%%%%%%%% 
%%%%%%%%%%%%%%%%%%%%%%%%%%%%%%%%%%%%%%%%%%%%%%%%%%%%%%%%%%%%%%%%%%%%%%%%%%%%

\section{Preliminaries}
\label{sec.preliminaries}

\subsection{The Faddeev quantum dilogarithm}
\label{sub.faddeev}

In this subsection, we recall some basic properties of the Faddeev quantum
dilogarithm, which are motivations for Theorems~\ref{thm.fourier}
and~\ref{thm.pentagon} below. At the same time, we will ignore additional
properties of the Faddeev quantum dilogarithm that play no role in our paper,
such as the fact that it is a meromorphic function with precise
zeros, poles and residues.
  
%% taken from p-adic quantum dilogarithm, sec.2

The Faddeev quantum dilogarithm~\cite{Faddeev} $\vphi:=\fad{}$ satisfies a key
integral pentagon identity~\cite{FK:QDL} 
\be
\label{phipentagon}
  e^{2 \pi i x y} \widetilde\vphi(x)\widetilde\vphi(y) 
  =\int_\BR e^{\pi i z^2} \widetilde\vphi(x-z)\widetilde\vphi(z)\widetilde\vphi(y-z)
  \operatorname{d}\!z 
\ee
where both sides are tempered distributions on $\BR$ and $\widetilde\vphi$
denotes the distributional (inverse) Fourier transformation
\begin{equation}\label{phifourier}
\widetilde\vphi(x):=\int_\BR e^{-2 \pi i x y} \vphi(y) \operatorname{d}\!y \,.
\end{equation}
It turns out that the inverse Fourier transform of $\vphi^{\pm 1}$ is expressed
in terms of $\vphi$ as given in~\cite[Sec.13.2]{AK}
\be
\label{fff}
\begin{aligned}
\int_\BR e^{-2 \pi i x y} \vphi(y) \operatorname{d}\!y &=
\zeta_8 (q/\tq)^{\frac{1}{24}} e^{-\pi i x^2} \vphi(-x+c_\bb) \\
\int_\BR e^{-2 \pi i x y} \vphi(y)^{-1} \operatorname{d}\!y &=
\zeta_8^{-1} (\tq/q)^{\frac{1}{24}} e^{\pi i x^2} \vphi(x-c_\bb) 
\end{aligned}
\ee
where $q=e^{2 \pi i \bb^2}$, $\tq=e^{-2\pi i/\bb^2}$, $c_\bb=\frac{i}{2}(\bb+\bb^{-1})$
and $\zeta_8=e^{2 \pi i/8}$. 
%% zeta_0 = \e(1/8 + (b^2+1/b^2)/24), zeta_inv = \e((b^2+1/b^2)/24)
%% \e(x)=e^{2 \pi i x}
The Faddeev quantum dilogarithm satisfies the inversion formula 
\be
\label{doublephi}
\vphi(x) \vphi(-x) = (\tq/q)^{\frac{1}{24}} e^{\pi i x^2}
\ee
see for example~\cite[App.A]{AK}.
In a certain domain, the Faddeev quantum dilogarithm is given as a ratio
of two Pochhammer symbols 

\be
\fad(x) \= \frac{(-q^{\frac{1}{2}} e^{2 \pi i \bb x};q)_\infty}{
  (-\tq^{\frac{1}{2}} e^{2 \pi i \bb^{-1} x};\tq)_\infty} \,,
\ee

\noindent
from which it follows that its asymptotic expansion as $\bb \to 0$
is given by replacing the denominator by $1$ and the numerator by the
asymptotic expansion of the Pochhammer symbol.

\subsection{Neumann--Zagier data}
\label{sub.NZ}

In this section, we discuss in detail Neumann--Zagier data following~\cite{DG}.
We start with a 3-manifold $M$ with torus boundary component (all manifolds and their
triangulations will be oriented throughout the paper) equipped with
a concrete oriented ideal triangulation, that is with a triangulation such that
each tetrahedron $\Delta$ of $\calT$ comes with a bijection of its vertices with
those of the standard 3-simplex. (All triangulations that are used in the computer
programs \texttt{SnapPy}~\cite{snappy} and \texttt{Regina}~\cite{regina} are concrete).

Every concrete tetrahedron $\Delta$ has shape parameters $(z,z',z'')$ assigned to
pairs of opposite edges as in Figure~\ref{fig.tetshapes}, where the complex numbers
$z'=1/(1-z)$ and $z''=1-1/z$ satisfy the equations
\be
\label{3z}
z''+z^{-1}=1, \qquad z z' z'' =-1 \,.
\ee

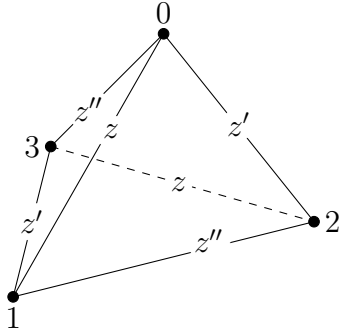
\begin{figure}[htpb!]
\begin{center}  
\begin{tikzpicture}[baseline=-3] 
    
\filldraw (2,3.5) circle (2pt) node[above] {$0$};
\filldraw (0,0) circle (2pt) node[below] {$1$};
\filldraw (4,1) circle (2pt) node[right] {$2$};
\filldraw (0.5,2) circle (2pt) node[left] {$3$};;
\draw (0,0) -- (2,3.5);
\draw (0,0) -- (4,1);
\draw (0,0) -- (0.5,2);
\draw (4,1) -- (2,3.5);
\draw (0.5,2) -- (2,3.5);
\draw[dashed] (0.5,2) -- (4,1);

\filldraw[white] (1.3,2.2) circle (6pt);
\filldraw[white] (2.2,1.5) circle (5pt);
\filldraw[white] (3,2.3) circle (6pt);
\filldraw[white] (0.25,1) circle (7pt);
\filldraw[white] (2.6,0.75) circle (8pt);
\filldraw[white] (1,2.5) circle (8pt);

\node at (1.3,2.2) {$z$};
\node at (2.2,1.5) {$z$};
\node at (3,2.3) {$z'$};
\node at (0.25,1) {$z'$};
\node at (2.6,0.75) {$z''$};
\node at (1,2.5) {$z''$};

\end{tikzpicture}
  \caption{The shapes of an ideal tetrahedron.}
\label{fig.tetshapes}
\end{center}
\end{figure}

If $\calT$ is a triangulation as above, an Euler characteristic argument shows
that the number of tetrahedra equals to the number of edges. Fix an ordering
of the tetrahedra $\Delta_j$ for $j=1,\dots,N$ and of the edges $e_1,\dots,e_N$
of $\calT$, and assign a shape $z_j$ to the tetrahedron $\Delta_j$ for $j=1,\dots,N$.
To describe the complete hyperbolic structure of $M$ (when it exists) and its
deformations, Thurston~\cite{Th} introduced the gluing equations
for the variables $z=(z_1,\dots,z_N)$ 
around each edge $e_i$, for $i=1,\dots, N$. In logarithmic form, these equations
have the form
\be
\label{3G}
\sum_{j=1}^N G_{i,j} \log z_j + G'_{i,j} \log z'_j + G''_{i,j} \log z''_j \= 2 \pi i,
\qquad (i=1, \dots, N)
\ee
where $G_{i,j}$ (and likewise for $G'_{i,j}$ and $G''_{i,j}$) denote the number
of times that an edge of $\Delta_j$ labelled $z_j$ winds around the edge $e_i$.
Every peripheral curve $\gamma$ in the boundary also gives rise to a gluing
equation of the same form as~\eqref{3G}, except the right hand side is $0$
instead of $2 \pi i$. Choosing a symplectic basis for $H_1(\partial M,\BZ)$, we
can enhance the equations~\eqref{3G} by adding the peripheral equations
\be
\label{3Gb}
\sum_{j=1}^N G_{N+c,j} \log z_j + G'_{N+c,j} \log z'_j + G''_{N+c,j} \log z''_j
\= 0, \qquad (c=1,2) \,.
\ee
It turns out that the $(N+2) \times N$ matrices $G$, $G'$ and $G''$ have both
symmetry and redundancy. Any one of the edge equations is implies by the others,
and we make a choice to remove one of them, and replace it by one peripheral
equation for a fixed peripheral curve, resulting into $N \times N$ matrices
$\mathbf G$, $\mathbf G'$ and $\mathbf G''$. Using the second
Equation~\eqref{3z} in logarithmic form $\log z_j + \log z'_j + \log z''_j = \pi i$,
we can eliminate one of the three shapes of each tetrahedron (this is a choice of
quad). For example, eliminating the shape $z'_j$ for $j=1,\dots,N$ then results
into a system of equations
\be
\label{3Gc}
\sum_{j=1}^N A_{i,j} \log z_j + B_{i,j} \log z''_j \= \pi i \nu
\qquad (i=1, \dots, N)
\ee
where
\be
\label{AB}
A \= \mathbf G - \mathbf G', \qquad B \= \mathbf G'' - \mathbf G'
\ee
are the Neumann--Zagier matrices~\cite{NZ} and
$\nu = (2,\dots,2,0)^t -\mathbf G' (1,\dots,1,1)^t \in \BZ^N$.
The Neumann--Zagier matrices $(A\,| \,B)$ have an important symplectic property,
they are the upper part of a symplectic matrix over $\BZ[1/2]$ (and even a symplectic
matrix over $\BZ$ if one divides the peripheral gluing equation by 2 while keeping
the integrality of its coefficients)~\cite{NZ}. It follows that $A B^t$ is
symmetric, that $(A \,|\, B)$ has full rank, and that one can always choose a
quad such that $B$ is invertible--for this see~\cite[Lem.A.3]{DG}.

A further ingredient of a Neumann--Zagier datum is a flattening, that is a triple
$(f,f',f'')$ of vectors in $\BZ^N$ that satisfy the conditions
\be
\label{3f}
f + f' + f''=(1,\dots,1)^t, \qquad 
\mathbf G f + \mathbf G' f' + \mathbf G'' f'' = (2,\dots,2,0)^t \,.
\ee
Using our choice of quad, we can eliminate $f'$ and thus obtain a pair
$(f,f'')$ of vectors in $\BZ^N$ that satisfy the condition
\be
\label{2f}
A f + B f'' = \nu \,.
\ee
This defines all the terms that appear in a NZ-datum $\Xi=(A,B,\nu,z,f,f'')$. 
The definition of the series $\Phi^\Xi(\hbar)$ requires a nondegenerate NZ datum
$\Xi$, that is one that satisfies the condition $\det(B) \neq 0$, which as
we discussed above, can always be achieved, as well as the condition
$\det(\Lambda) \neq 0$. We discuss this choice next, and connect it to
the geometric representation of a hyperbolic 3-manifold $M$. 

We end this section with a comment regarding Neumann--Zagier data of 3-manifolds
with several (as opposed to one) torus boundary components. Their ideal
triangulations with equal number $N$ of tetrahedra as edges and the edge gluing
equations have the same shape~\eqref{3G} as above. If $r$ denotes the number of
boundary components, then there are $2r$ peripheral equations~\eqref{3Gb} and after
a choice of one peripheral curve per boundary component, this leads to $(N+r)\times N$
matrices $G$, $G'$ and $G''$. The edge gluing equations have redundancy, and although
it is not true that we can remove any $r$ of them, it is shown
in~\cite[Sec.4.6]{GHRS} that one can remove $r$ of them and and replace them by
$r$ peripheral equations so as to obtain $N \times N$ matrices 
$\mathbf G$, $\mathbf G'$ and $\mathbf G''$ such that the corresponding matrices
$(A \,|\,B)$ defined in~\eqref{AB} have the same symplectic properties as in the
case of $r=1$. Moreover, any two choices of removal of the $r$ edge equations are
related to each other by an invertible matrix in $\GL_r(\BZ)$. Finally, flattenings
satisfy Equations~\eqref{3f} where the right hand side of the second equation
is the vector $(2,\dots,2,0,\dots,0)^t \in \BZ^{N+r}$ with the first $N$ coordinates
equal to $2$ and the remaining $r$ coordinates equal to zero.

For simplicity in the presentation (related to the choice of peripheral curves and
the flattenings), we will give the proof of Theorem~\ref{thm.1} assuming that the
3-manifold $M$ has one cusp. The proof applies verbatim to the general case of
arbitrary number of cusps.

\subsection{Geometric aspects}
\label{sub.geometric}

In this section, we discuss some geometric aspects of Theorem~\ref{thm.1} related
to the choice of the shapes $z$ in a Neumann--Zagier datum. Let us fix an ideal
triangulation $\calT$. A solution $z \in (\BC\setminus\{0,1\})^N$ to the
Neumann--Zagier equations gives rise, via a developing map, to a representation
$\rho_z: \pi_1(M) \to \PSL_2(\BC)$. For a detailed discussion, see the appendix
of~\cite{BoydII}. However, not every representation $\rho: \pi_1(M) \to \PSL_2(\BC)$ is
``seen'' by $\calT$, that is, is in the image of the map $z \to \rho_z$.
What's more, if $\rho$ is in the image of the above map and we do a 2--3 Pachner
move on the triangulation, it may no longer be in the image of the map corresponding
to the new triangulation. The reason is that the shapes of the two triangulations
are related by a rational map, which may send a complex number different from $0$
and $1$ to $0$, $1$ or $\infty$. To phrase the problem differently, every two ideal
triangulations (each with at least two tetrahedra, as we will always assume) of
a 3-manifold with non-empty boundary are related by a sequence of 2--3 Pachner
moves~\cite{Matveev,Piergallini}. However, it is not known that the set of
ideal triangulations that see the discrete faithful representation
$\rho^\geom : \pi_1(M) \to \PSL_2(\BC)$ is connected under 2--3 Pachner moves,
nor is it known whether the set of nondegenerate NZ data is connected under
2--3 Pachner moves. 

A solution to these issues was found in~\cite{DG}, and this was used to
prove the topological invariance of the 1-loop function, and was also used
in~\cite{GHRS} to prove the topological invariance of the 3D-index. Let us recall
the geometric details here. Every cusped hyperbolic 3-manifold $M$ (complete,
finite volume) has a canonical ideal cell decomposition whose cells are
3-dimensional convex ideal polytopes given by Epstein--Penner~\cite{EP}. It is easy to
see that every convex ideal polyhedron can be triangulated into ideal
tetrahedra by connecting an ideal vertex to all other ideal vertices (thus
reducing the problem to the ideal cone of an ideal polygon), and then triangulating
further the ideal polygon into ideal triangles. Doing so, the triangulation of
the common faces of the 3-dimensional convex ideal polytopes may not match, in
which case one can pass from one triangulation of a polygonal face to another by
adding flat tetrahedra. 

The question is whether every two such triangulations are related by a sequence of
2--3 moves. This is a combinatorial problem of convex geometry, which we summarize
below. For a detailed discussion, the reader may consult the book~\cite{LRS}
and references therein.

Fix a convex polytope $P$ in $\BR^d$. One can consider the set of
triangulations of $P$. When $d=2$, $P$ is a polygon and it is known that 
every two triangulations are related by a sequence of flips. For general
$d$, flips are replaced by {\em geometric bistellar moves}. When $d \geq 5$,
it is known that the graph of triangulations (with edges given
by geometric bistellar flips) is not connected, and has isolated vertices.
For $d=3$, it is not known whether the graph is connected.

The situation is much better when one considers {\em regular triangulations}
of $P$. In that case, the corresponding graph of regular triangulations
is connected and, in fact, it is the edge set of the {\em secondary polytope} of
$P$. When $d=3$ and $P$ is convex and in general position, then the only
geometric bistellar move is the 2--3 move where the added edge that appears in
the move is an edge that connects two vertices of $P$. When $d=3$ and $P$ is
not in general position, the same conclusion holds as long as one allows for
tetrahedra that are flat, i.e., lie on a 2-dimensional plane.

Returning to the Epstein-Penner ideal cell decomposition, let $\calT^{\mathrm{EP}}$
denote a regular ideal triangulations of the canonical ideal cell decomposition
of $M$. (For a detailed discussion of this set, see also~\cite[Sec.6]{GHRS}).
The set $\{\calT^{\mathrm{EP}}\}$ of regular ideal triangulations
is connected by 2--3 Pachner moves. Moreover, such triangulations see $\rho^\geom$
since by the geometric construction, the shapes are always nondegenerate, i.e.,
not equal to $0$ or $1$ and in fact always have nonnegative
(though sometimes zero) imaginary part.

Finally, we need to show that $\det(\Lambda)$ is nonzero. This follows from the
fact that\\ $\det(B)\det(\Lambda)=\det(-A+B\diag(1/(1-z_j))$ equals (up to
multiplication by a monomial in $z$ and $z''$) to the 1-loop invariant
of~\cite[Sec.1.3]{DG}. The nonvanishing of the latter follows from
Thurston's hyperbolic Dehn surgery theorem~\cite{Th}, which implies that
$\rho^\geom \in X_M^\geom \cup P_M$ is an isolated smooth point of the geometric
component $X^\geom_M$ of the $\PSL_2(\BC)$-character variety of $M$, intersected
with the locus $P_M$ of boundary-parabolic $\PSL_2(\BC)$-representations, i.e.,
representations $\rho$ satisfying $\mathrm{tr}(\rho(\gamma))^2=4$ for all peripheral
elements $\gamma \in \pi_1(M)$. 
Since the 1-loop invariant is the determinant of the Hessian of the defining
NZ equations of $\rho^\geom$, it follows that the 1-loop invariant is nonzero. 

%%%%%%%%%%%%%%%%%%%%%%%%%%%%%%%%%%%%%%%%%%%%%%%%%%%%%%%%%%%%%%%%%%%%%%%%%%%% 
%%%%%%%%%%%%%%%%%%%%%%%%%%%%%%%%%%%%%%%%%%%%%%%%%%%%%%%%%%%%%%%%%%%%%%%%%%%%

\section{Formal Gaussian integration}
\label{sec.fgi}

\subsection{Basics on formal Gaussian integration}
\label{sub.prelim}

In this section, we review the basic properties of formal Gaussian integration,
which is a combinatorial analogue of integration of analytic functions. This theory
was introduced and studied in detail in~\cite{AarhusII}, where it was used to
define a universal perturbative finite type invariant of rational homology 3-spheres,
and to identify it with the trivial connection contribution of Chern--Simons
perturbation theory. 

As a warm-up, the formal Gaussian integral of a monomial $x^n$ in one variable
with respect to the quadratic form $\lambda \neq 0$ is defined by
\bea
\label{formalG1}
\left\langle x^n \right\rangle_{x,\lambda}
\= \begin{cases}
  \lambda^{-n/2} (n-1)!! & n \,\, \text{even} \\
  0 & n \,\, \text{odd} \,.
\end{cases}
\eea
When $\lambda>0$, then the above bracket equals to a normalized Gaussian integral
\be
\left\langle x^n \right\rangle_{x,\lambda} \=
\frac{\int_\BR e^{-\frac{1}{2}\lambda x^2} x^n \, dx}{
  \int_\BR e^{-\frac{1}{2}\lambda x^2} \, dx} \,,
\ee
explaining the naming of formal Gaussian integration. The formal Gaussian integration
can be extended by linearity to $\langle f(x) \rangle_{x,\lambda}$ for any polynomial
$f(x)$, or even further to a power series $f(x) = \sum_{n \geq 0} a_n x^n$ whose
coefficients tend to zero in a local ring (such as the ring
$\BQ\llbracket \hroot \rrbracket$).

The formal Gaussian integral~\eqref{formalG1} has a multivariable extension given
in Equation~\eqref{FGI} where $x$ is a vector of variables and $\Lambda$ is
an invertible matrix over a field matching the size of $x$.
Throughout this paper, the entries of $\Lambda$ are elements of the field $\BQ(z)$
where $z$ is a vector of variables, the integrable functions $f_\hbar(x,z)$ that
we apply the formal Gaussian integration with respect to $x$ with be elements of
$\BQ(z)[x]\llbracket \hroot \rrbracket$, and the result of formal Gaussian integration
will be an element of the ring $\BQ(z)\llbracket \hbar \rrbracket$ or
$\BQ(z)[x]\llbracket \hbar^{\frac{1}{2}} \rrbracket$. When we specialize
to the case of a vector of algebraic numbers $z$, then $\BQ(z)$ defines the
corresponding number field.    

Just like integration of sufficiently smooth functions satisfies certain invariance
properties (such as change of variables, iterated integration, and even integration
by parts~\cite{AarhusII}), so does formal Gaussian integration. The corresponding
identities of formal Gaussian integration are combinatorial statements about
polynomials or rational functions and often follow from the corresponding statements
of integration of functions. 

We now give some elementary properties of formal Gaussian integration that we will
use in our paper.

\begin{lemma}
\label{lem.ff1}
%We have the following identities in $\BQ(z)\llbracket\hbar\rrbracket$.
\rm{(a)} For all invertible matrices $\Lambda$ and $P$ over $\BQ(z)$, we have
\bea
\left\langle
f_{\hbar}(Px,z)
\right\rangle_{x, P^t \Lambda P}
\= \left\langle f_{\hbar}(x,z)\right\rangle_{x,\Lambda}\,.
%\left\langle f_{\hbar}(x',x'')\right\rangle_{x,\Lambda}
%\= \left\langle
%f_{\hbar}(x'+E^t x'',x'')
%\right\rangle_{x,\widetilde\Lambda} \,.
\eea
\rm{(b)} For all invertible matrices $\Lambda$ and vectors $G$ over $\BQ(z)$,
we have
\be
\label{fga1}
\langle\exp(-G^{t}\Lambda x \hroot) f_{\hbar}(x+G \hroot,z)
\rangle_{x,\Lambda} \= \exp\left(\frac{G^{t}\Lambda G}{2}\hbar\right)
\langle f_{\hbar}(x,z)\rangle_{x,\Lambda}
\ee
\rm{(c)} If
$\Lambda=\left(\begin{smallmatrix} A & B \\ B^t & C \end{smallmatrix}\right)$,
and $\Lambda$ and $A$ invertible, then for any $F$, we have 
\bea
\label{ff1}
\left\langle
\exp(F x' \hroot)f_{\hbar}(x'',z)
\right\rangle_{x,\Lambda}
\= \exp\left(\frac{FA^{-1}F^t}{2}\hbar \right)
\left\langle
\exp(-FA^{-1}B x''\hroot) f_{\hbar}(x'',z)
\right\rangle_{x'',C-B^t A^{-1} B} \!\!.
\eea
\end{lemma}

\begin{proof}
Part (a) follows from the fact the integration is unchanged
under a linear change of variables.
%% $\int\exp(x^t\Lambda x/2)f(x)dx=\int\exp(x^tP^t\Lambda Px/2)f(Px)dx\det(P)$

Part (b) follows from the fact the integration is unchanged
under an affine change of variables.

Part (c) follows from Fubini's theorem~\cite[Prop.2.13]{AarhusII}, combined
with Equation~\eqref{fga1}.
\end{proof}

The next lemma, which will be important in the application of $q$-holonomic
methods in Section~\ref{sub.qhol} and in the proofs of Theorems~\ref{thm.fourier}
and~\ref{thm.pentagon} below, concerns the behavior of formal Gaussian
integration when $z=(z_1,\dots,z_r)$ is shifted to
$e^{\ve \hbar}z:=(e^{\ve_1 \hbar}z_1,\dots,e^{\ve_r \hbar}z_r)$ for integers $\ve_j$.

\begin{lemma}
\label{lem.ff2}
For all invertible matrices $\Lambda(z)$ over $\BQ(z)$ and all integrable functions
$f_\hbar$, we have
\be
\label{eq.ff2}
\langle f_{\hbar}(x,z)\rangle_{x,\Lambda(z)} |_{z \mapsto e^{\ve \hbar}z}
\= \sqrt{\frac{\det\Lambda(e^{\ve\hbar}z)}{\det\Lambda(z)}} 
\Big\langle \exp\bigg(
\frac{x^{t}\big(\Lambda(z)-\Lambda(e^{\ve\hbar}z)\big)x}{2}\bigg)
f_{\hbar}(x,e^{\ve\hbar}z) \Big\rangle_{x,\Lambda(z)} \,.
\ee
\end{lemma}

\begin{proof}
The lemma follows from recentering the Gaussian after multiplying
$z$ by $e^{\ve\hbar}$.
\end{proof}

\subsection{\texorpdfstring{$q$--holonomic aspects of $\psi_\hbar$}{q-holonomic aspects of psi}}
\label{sub.qhol}
It is well-known that identities of holonomic functions can be proven 
algorithmically~\cite{WZ,PWZ}. We will use these ideas, adapted to our needs,
to prove fundamental identities between certain Gaussian integrals involving
the building block $\exp(\psi_{\hbar})$ of our series $\Phi^\Xi(\hbar)$.
Since $\psi_{\hbar}$ is related to the infinite Pochhammer symbol given in
Equation~\eqref{poch}, its functional equations will be of fundamental importance.
From its definition, it is clear that the infinite Pochhammer symbol
satisfies a simple first order linear $q$--difference equation
\be
\label{eq.poch}
(x;q)_{\infty} \= (1-x)(qx;q)_{\infty} \,.
\ee
To convert this into an identity of formal $\hbar$-series where $q=e^{\hbar}$,
we use the fact that there is an action of the quantum plane
(also known as
$q$--Weyl~\cite[Ex. 1.7]{Etingof} algebra)
on the space $\BQ(z)[x]\llbracket\hroot\rrbracket$
of integrable functions by an action on the $z$-variable given by
\be
\label{ML}
(L f_\hbar)(x,z) \= f_\hbar(x,e^{\hbar} z), \qquad
(M f_\hbar)(x,z) \= e^{\hbar} f_\hbar(x,e^{\hbar} z)
\ee
where $LM=qML$. The next lemma asserts that the completion
$\widehat{\psi}_{\hbar}(x,z)$ 
\be
\label{eq.psicomp}
\begin{aligned}
\widehat{\psi}_{\hbar}(x,z)
\;:&=\;
\exp\Big(\!
-\frac{\Li_{2}(z)}{\hbar}
-\frac{\Li_{1}(z)x}{\hbar^{\frac{1}{2}}}
+\frac{1}{2}\Li_{1}(z)
-\frac{\Li_{0}(z)x^2}{2}
\Big)
\psi_{\hbar}(x,z) \\
&=\;
\exp\Big(\!
-\sum_{k,\ell\in\BZ_{\geq0}}
\frac{B_{k}\,x^{\ell}\,\hbar^{k+\frac{\ell}{2}-1}}{\ell!\,k!}\Li_{2-k-\ell}(z)\Big)
\\
&\quad\in
\exp\Big(\!-\frac{\Li_{2}(z)}{\hbar}
-\frac{\Li_{1}(z)x}{\hbar^{\frac{1}{2}}}
+\frac{1}{2}\Li_{1}(z)
-\frac{\Li_{0}(z)x^2}{2}
\Big)(1+\hbar^{\frac{1}{2}}\BQ(z)[x]\llbracket\hbar^{\frac{1}{2}}
\rrbracket)
\end{aligned}
\ee
of $\psi_{\hbar}(x,z)$, as well as $\psi_{\hbar}(x,z)$ itself, satisfy a
corresponding first order linear $q$--difference equation, albeit with
complicated coefficients.

\begin{lemma}
\label{lem.psihdiff}
\rm{(a)} We have:
\be
\label{eq.qdiff.of.psih}
\widehat{\psi}_{\hbar}(x,e^{\hbar}z)
\=
(1-ze^{x\hbar^{1/2}})\,\widehat{\psi}_{\hbar}(x,z)\,.
\ee
\rm{(b)} We have:
\be
\label{eq.qdiff.of.psi}
\begin{aligned}
&\psi_{\hbar}(x,e^{\hbar}z)
\=
(1-ze^{x\hbar^{1/2}})
\,\psi_{\hbar}(x,z)
\sqrt{\frac{1-e^{\hbar}z}{1-z}}\\
&\times
\exp\Big(\frac{\Li_{2}(e^{\hbar}z)}{\hbar}-\frac{\Li_{2}(z)}{\hbar}
+\frac{\Li_{1}(e^{\hbar}z)x}{\hbar^{\frac{1}{2}}}
-\frac{\Li_{1}(z)x}{\hbar^{\frac{1}{2}}}
+\frac{1}{2}\Li_{0}(e^{\hbar}z)x^2
-\frac{1}{2}\Li_{0}(z)x^2\Big)\,.
\end{aligned}
\ee
\end{lemma}

Note that the identity~\eqref{eq.qdiff.of.psih} takes place in a larger ring,
which includes the symbols
$\exp(\frac{\Li_{2}(z)}{\hbar})$, $\exp(\frac{\Li_{1}(z)x}{\hroot})$,
$\exp(\frac{\Li_{0}(z)x^2}{2})$ and $\exp(\frac{\Li_{0}(z)x^2}{2})$ which one can
adjoin in the differential field $\BQ(z)((\hbar))$ as is standard in differential
Galois theory of linear differential equations~\cite{Singer:galois-differential}.
The symbols $\Li_k(z)$ for $k=0,1,2$ can be interpreted as normalized solutions
to the linear differential equations $z \frac{d}{dz} \Li_k(z) = \Li_{k-1}(z)$
with $\Li_0(z)=z/(1-z)$ and satisfy the usual properties
$\Li_k(e^\hbar z) = \sum_{r=0}^\infty \frac{1}{r!} \Li_{k-r}(z) \hbar^r$.
On the other hand, the coefficients in identity~\eqref{eq.qdiff.of.psi}
are elements of the ring $\BQ(z)[x]\llbracket \hroot \rrbracket$.

\begin{proof}
Part (a) can be proved directly using the facts that for
the Bernoulli polynomials $B_{r}(x)$ we have
$B_{r}(1)=B_{r}+\delta_{r,1}$ and
$B_{r}(x)=\sum_{k=0}^{r}\binom{r}{k}B_{n-k}x^{k}$.
Applying these identities we find that
\be
\begin{aligned}
&\exp\Big(-\sum_{k,\ell\in\BZ_{\geq0}}
\frac{B_{k}x^{\ell}\hbar^{k+\frac{\ell}{2}-1}}{\ell!k!}\Li_{2-k-\ell}(ze^{\hbar})\Big)
\\
&\=
\exp\Big(-\sum_{r,\ell\in\BZ_{\geq0}}\sum_{k=0}^{r}B_{k}\binom{r}{k}
\frac{x^{\ell}\hbar^{r+\frac{\ell}{2}-1}}{\ell!r!}\Li_{2-r-\ell}(z)\Big)\\
&\=
\exp\Big(-\sum_{r,\ell\in\BZ_{\geq0}}
\frac{B_{r}(1)x^{\ell}\hbar^{r+\frac{\ell}{2}-1}}{\ell!r!}\Li_{2-r-\ell}(z)\Big)\\
&\=
\exp\Big(-\Li_{1}(ze^{x\hbar^{1/2}})-\sum_{r,\ell\in\BZ_{\geq0}}
\frac{B_{r}x^{\ell}\hbar^{r+\frac{\ell}{2}-1}}{\ell!r!}\Li_{2-r-\ell}(z)\Big)\\
&\=
(1-ze^{x\hbar^{1/2}})\exp\Big(-\sum_{r,\ell\in\BZ_{\geq0}}
\frac{B_{r}x^{\ell}\hbar^{r+\frac{\ell}{2}-1}}{\ell!r!}\Li_{2-r-\ell}(z)\Big)\,.
\end{aligned}
\ee
Part (b) follows from (a), using Equation~\eqref{eq.psicomp} and expanding
in $\hbar$ to find that
\be
\begin{tiny}
\begin{aligned}
\sqrt{\frac{1-e^{\hbar}z}{1-z}}
\exp\Big(\frac{\Li_{2}(e^{\hbar}z)}{\hbar}-\frac{\Li_{2}(z)}{\hbar}
+\frac{\Li_{1}(e^{\hbar}z)x}{\hbar^{\frac{1}{2}}}
-\frac{\Li_{1}(z)x}{\hbar^{\frac{1}{2}}}
+\frac{1}{2}\Li_{0}(e^{\hbar}z)x^2
-\frac{1}{2}\Li_{0}(z)x^2\Big)
\in\BQ(z)[x]\llbracket\hbar^{\frac{1}{2}}\rrbracket\,.
\end{aligned}
\end{tiny}
\ee
\end{proof}

We end this section by discussing a useful relation between $\psi_\hbar(x,z)$
and its specialization at $x=0$. It is easy to see that the completion
$\widehat{\psi}_{\hbar}(x,z)$ of $\psi_\hbar(x,z)$ is regular at $x=0$ and satisfies
\be
\widehat{\psi}_{\hbar}(x,z)
\=
\widehat{\psi}_{\hbar}(0,ze^{x\hbar^{1/2}})\,.
\ee
This implies a corresponding statement
\bea
\label{eq.psid}
\psi_\hbar(x,z)
\= \psi_\hbar\big(0,ze^{x\hbar^{1/2}}\big)C_\hbar(x,z),
\eea
for $\psi_{\hbar}$ 
where
\bea
\label{Cdef}
C_\hbar(x,z) &\= \exp\Big(\!-\sum_{\ell\geq 3}
\frac{\hbar^{\frac \ell 2 -1}}{\ell!}\Li_{2-\ell}(z)x^\ell 
+\frac 1 2 \sum_{\ell\geq 1} \frac{h^{\frac \ell2}}{\ell!}\Li_{1-\ell}(z)x^\ell\Big)\\
&\=\exp\Big(\!-\frac{\Li_2(ze^{x\hroot})}\hbar+\frac{\Li_2(z)}\hbar
-\frac{\log(1-z)}\hroot x + \frac{z}{1-z}\frac {x^2}2\\
&\qquad\qquad-\frac 1 2 (\log(1-z\ezh) - \log(1-z))\Big)\,.
\eea

\subsection{\texorpdfstring{Fourier transform of $\psi_\hbar$}{Fourier transform of psi}}
\label{sub.fourier}

In this section, we prove two functional identities for $\psi_{\hbar}(\bullet,z)$,
$\psi_{\hbar}(\bullet,z')$ and $\psi_{\hbar}(\bullet,z'')$
where $z':=1/(1-z)$ and $z'':=1-1/z$ are related to the $\BZ/3\BZ$-symmetry of
the shapes of a tetrahedron, and reflect the $\BZ/3\BZ$-symmetry of the dilogarithm
function.

The next theorem is a formal Gaussian integration version of
the Fourier transform of the Faddeev quantum dilogarithm given in
Equation~\eqref{phifourier}. The proof, however, does not follow from the
distributional identity~\eqref{phifourier}, and instead uses $q$--holonomic ideas
and the basics of formal Gaussian integration. 

This theorem and its Corollary~\ref{cor.fourier} are used in Section~\ref{sec.quad}
to show that the power series $\Phi^\Xi(\hbar)$ are independent of the choice of
a nondegenerate quad.

\begin{theorem}
\label{thm.fourier}
We have:
\bea
\label{eq.fourier}
\psi_\hbar (x,z)
\=e^{-\frac{\hbar}{24}}\Big\langle
\exp\Big(
\Big(y+\frac{xz}{1-z}\Big) \frac{\hbar^{\frac 1 2}}{2}\Big)\,
\psi_\hbar\Big(y+\frac{xz}{1-z}, \frac 1 {1-z}
\Big)
\Big\rangle_{y,1-z^{-1}} 
\eea
in the ring $\BQ(z)[x]\llbracket\hroot\rrbracket$.
\end{theorem}

In fact, both sides of of Equation~\eqref{eq.fourier} lie in the subring
$\BQ[z^{\pm},(1-z)^{-1}][x]\llbracket\hroot\rrbracket$. 

\begin{proof}
It it clear from the definition that both sides of Equation~\eqref{eq.fourier}
are elements of the ring $\BQ(x)[z]\llbracket\hbar\rrbracket$. 
First, we will prove the specialization of Equation~\eqref{eq.fourier}
when $x=0$, i.e., we will show that
\bea
\label{eq.fourier0}
\psi_\hbar(0,z)
\= e^{-\frac{\hbar}{24}}
\Big\langle
\exp\Big(
\frac y2\hroot\Big)\,
\psi_\hbar\Big(y,\frac 1{1-z}\Big)
\Big\rangle_{y,1-z^{-1}} \in \BQ(z)[[\hbar]] \,.
\eea
To prove this, we will combine $q$-holonomic ideas with formal Gaussian
integration. From Equation~\eqref{eq.qdiff.of.psih}, we see that the left hand
side of Equation~\eqref{eq.fourier0} multiplied by
$\frac{\exp(-\Li_{2}(z)\hbar^{-1})}{\sqrt{1-z}}$ satisfies a simple first order
$q$--difference equation. We want to show the same for the right hand side.
To do this, consider the function
\be
\begin{aligned}
I_{m,\hbar}(w,z)
\=
w^me^{-\frac{\hbar}{24}+\frac{\pi^2}{6\hbar}}
\exp\Big(
\log(w)\Big(
\!-\frac{\log(w)}{2\hbar}
+\frac{\pi i}{\hbar}
-\frac{\log(z)}{\hbar}
+\frac{1}{2}\Big)\Big)\,
\widehat{\psi}_\hbar(0,w)
\end{aligned}
\ee
which is an element of a larger ring discussed in relation to
$\widehat{\psi}_{\hbar}$ of Equation~\eqref{eq.psicomp}. Using
Equation~\eqref{eq.psid} and the symmetry
\be
    \Li_2(z)
    \=
    \Li_2\Big(\frac 1 {1-z}\Big)
    -\frac 1 2\log({1-z})^2 +\pi i \log({1-z})-\log(z)\log({1-z})
    -\frac {\pi^2}6
\ee
of the dilogarithm~\cite{Zagier:dilog}, it is easy to see that
the right hand side of Equation~\eqref{eq.fourier0} is given by
\be
\label{eq.fourier.I.to.bracket}
\Big\langle
\sqrt{-z}
\exp\Big(\frac{\Li_{2}(z)}{\hbar}
+(1-z^{-1})\frac{w^2}{2}\Big)
I_{0,\hbar}\Big(\frac{1}{1-z}e^{w\hbar^{1/2}},z\Big)
\Big\rangle_{w,1-z^{-1}}\,.
\ee
The function $I_{m,\hbar}$ satisfies the linear $q$--difference equations
\be
\begin{aligned}
I_{m,\hbar}(e^{\hbar}w,z)
&\=
-e^{m\hbar}w^{-1}z^{-1}(1-w)I_{m,\hbar}(w,z)\,,\\
I_{m,\hbar}(w,e^{\hbar}z)
&\=
w^{-1}I_{m,\hbar}(w,z)\,,\\
I_{m,\hbar}(w,z)
&\=
wI_{m,\hbar}(w,z)\,,
\end{aligned}
\ee
which imply that
\be
\label{eq.If}
(1-e^{-m\hbar}z)I_{m,\hbar}\Big(\frac{1}{1-z}e^{w\hbar^{1/2}+\hbar},z\Big)
\=
I_{m}\Big(\frac{1}{1-z}e^{w\hbar^{1/2}},e^{\hbar}z\Big).
\ee
If we multiply both sides of~\eqref{eq.If} with the factors from
Equation~\eqref{eq.fourier.I.to.bracket} and take the bracket of both sides
and apply Lemma~\ref{lem.ff1} and Lemma~\ref{lem.ff2} to change
coordinates and the Gaussian, we find that when $m=0$ both sides of
Equation~\eqref{eq.fourier0} satisfy the same $q$--difference
equation~\eqref{eq.qdiff.of.psih}.
% \be
% \begin{aligned}
% &e^{-\frac{\hbar}{24}}
% \left\langle
% \exp\left(
% \psi_\hbar\left(w,\frac 1{1-e^{\hbar}z}\right)
% +\frac w2\hroot
% \right)
% \right\rangle_{w,1-z^{-1}}\\
% &\=\Big\langle
% \sqrt{1-e^{\hbar}z}
% \exp\Big(\frac{\Li_{2}(e^{\hbar}z)}{\hbar}
% +(1-e^{-\hbar}z^{-1})\frac{w^2}{2}\Big)
% I_{m,\hbar}\Big(\frac{1}{1-z}e^{w\hbar^{1/2}},e^{\hbar}z\Big)
% \Big\rangle_{w,1-e^{-\hbar}z^{-1}}\\
% &\=\Big\langle
% \sqrt{1-e^{\hbar}z}
% \exp\Big(\frac{\Li_{2}(e^{\hbar}z)}{\hbar}
% +(1-e^{-\hbar}z^{-1})\frac{w^2}{2}\Big)
% (1-e^{-m\hbar}z)I_{m,\hbar}\Big(\frac{1}{1-z}e^{w\hbar^{1/2}+\hbar},z\Big)
% \Big\rangle_{w,1-e^{-\hbar}z^{-1}}\,.
% \end{aligned}
% \ee
Moreover, it is easy to see that both sides of
Equation~\eqref{eq.fourier.I.to.bracket} are power series in $\hbar$ with
coefficients rational functions of $z$ of nonpositive degree, thus we can work
with the ring $\BQ\llbracket z^{-1}\rrbracket\llbracket \hbar\rrbracket$
instead. In this case, the Newton polygon of this first order linear
$q$-difference equation implies that it has a unique solution in
$\BQ\llbracket z^{-1}\rrbracket\llbracket \hbar\rrbracket$ determined by its
value at $z=\infty$. Therefore, the identity in Equation~\eqref{eq.fourier0}
follows from its specialization at $z=\infty$. Since
\be
\begin{aligned}
\psi_\hbar(0,\infty)
&\=
\exp\Big(\frac{\hbar}{12}\Big)\,,\\
e^{\frac{-\hbar}{24}}\left\langle
\exp\Big(
\frac w2\hroot\Big)
\psi_\hbar\left(w,0\right)
\right\rangle_{w,1}
&\=
e^{\frac{-\hbar}{24}}\left\langle
\exp\left(
\frac w2\hroot
\right)
\right\rangle_{w,1}
\=
\exp\Big(\frac{\hbar}{12}\Big)\,,
\end{aligned}
\ee
this completes the proof of Equation~\eqref{eq.fourier0}.

\medskip

Going back to the general case of $z$, Equation~\eqref{eq.fourier} follows
from Equation~\eqref{eq.psid}
together with~\eqref{eq.fourier0} using Lemma~\ref{lem.ff2}
to shift the Gaussian and Lemma~\ref{lem.ff1}
to change integration variable via the transformation
\be
w \mapsto w-\frac 1 \hroot \log\left(\frac{1-z}{1-ze^{x\hbar^{1/2}}}\right)
-\frac{x}{1-z^{-1}}
\,.
\ee
The detailed calculation is given in Appendix~\ref{app.fourier}. 
\end{proof}

Theorem~\ref{thm.fourier} implies the following identity relating
$\psi_\hbar(\bullet,z)$ and $\psi_\hbar(\bullet,z'')$.

\begin{corollary}
\label{cor.fourier}
We have:
\bea
\label{eq.fourier2}
\psi_\hbar (x,z)
\=e^{\frac{\hbar}{24}}\Big\langle
\exp\Big(\!
-\frac{x}{2}\hbar^{\frac{1}{2}}\Big)\,
\psi_\hbar\Big(y-\frac{x}{1-z}, 1-z^{-1}\Big)
\Big\rangle_{y,z-1} 
\eea
in the ring $\BQ(z)[x]\llbracket\hroot\rrbracket$.
\end{corollary}
\begin{proof}
Apply Equation~\eqref{eq.fourier} to the $\psi_{\hbar}$
that appears on the right hand side of the same Equation~\eqref{eq.fourier}.
Then apply a change of variables and Fubini's
theorem of Lemma~\ref{lem.ff1} to calculate
the bracket for the variable that doesn't appear
in the argument of the remaining $\psi_{\hbar}$.
\end{proof}

\subsection{\texorpdfstring{Pentagon identity for $\psi_\hbar$}{Pentagon identity for psi}}
\label{sub.pentagon}

In this section, we prove a pentagon identity for the functions
$\psi_\hbar(\bullet,z)$ where $z$ takes the five values
\be
\label{5z}
[z_1]+[z_2]\mapsto
[z_1 z_0^{-1}] +[z_0] +[z_2 z_0^{-1}], 
\qquad z_0\=z_1+z_2-z_1 z_2
\ee
what appear in the 5-term relation for the dilogarithm.
The next theorem is a formal Gaussian integration version of the pentagon
identity~\eqref{phipentagon} of the Faddeev quantum dilogarithm and will be
used in Section~\ref{sec.pachner} to prove that the series $\Phi^\Xi(\hbar)$
is independent of 2--3 Pachner moves. Let us denote
\be
\label{eq.delta23}
\delta
\=2+\Li_{0}(z_1z_{0}^{-1})+\Li_{0}(z_{0})+\Li_{0}(z_2z_{0}^{-1})
\=
\frac{(z_1+z_2-z_1z_2)^2}{z_1z_2(1-z_1)(1-z_2)}\,.
\ee

\begin{theorem}
\label{thm.pentagon}
We have: 
\be
\label{eq.pentagon}
\begin{aligned}
&\psi_\hbar(x,z_1)\,\psi_\hbar(y,z_2)
\=e^{- \frac \hbar {24}}
\Big\langle
\psi_\hbar\Big(-w -y
+\frac{xz_2+yz_1}{z_0},
z_1z_0^{-1}\Big)\\
&
\psi_\hbar\Big(w+x+y
-\frac{xz_2+yz_1}{z_0},
z_0\Big)
\psi_\hbar\Big(-w-x
+\frac{xz_2+yz_1}{z_0},
z_2z_0^{-1}\Big)
\Big\rangle_{w,\delta}
\end{aligned}
\ee
in the ring $\BQ(z_1,z_2)[x,y]\llbracket\hroot\rrbracket$. 
\end{theorem}

% see Mathematica files:
% topological-invariance-Phi-series/
%    PowerSeries.Pentagon.Faddeev.Quantum.Dilogarithm.v4.nb
% Claims checked up to $O(h^3)$.

\begin{proof}
It is clear from the definition that both sides of Equation~\eqref{eq.pentagon}
are elements of the ring $\BQ(z_1,z_2)[x,y]\llbracket\hroot\rrbracket$. To prove this,
we follow the same approach we used to prove Theorem~\ref{thm.fourier}. First, we
prove the identity~\eqref{eq.pentagon} when $u=v=0$, i.e., we will show that
\be
\label{pentagon0}
\psi_\hbar(0,z_1)\,\psi_\hbar(0,z_2) 
\=e^{- \frac \hbar {24}}
\big\langle\psi_\hbar(-w,z_1z_0^{-1})
\,\psi_\hbar(w,z_0)
\,\psi_\hbar(-w,z_2z_0^{-1})
\big\rangle_{w,\delta}
\ee
in the ring $\BQ(z_1,z_2)\llbracket\hbar\rrbracket$
and to prove this we will again use $q$--holonomic methods.
To do so, consider the function
\be
\label{eq.Ip}
\begin{aligned}
&I_{m,\hbar}(z_1,z_2,z)
\=
e^{\frac{\pi^{2}}{6\hbar}-\frac{\hbar}{24}}z^{m}\,
\widehat{\psi}_{\hbar}(0,z_1z^{-1})\,
\widehat{\psi}_{\hbar}(0,z)\,
\widehat{\psi}_{\hbar}(0,z_2z^{-1})\\
&\times\exp\Big(-\frac{\log(z_1)\log(z_2)}{\hbar}
+\frac{\log(z)\log(z_1)}{\hbar}
+\frac{\log(z)\log(z_2)}{\hbar}
-\frac{\log(z)^2}{\hbar}\Big)\,,
\end{aligned}
\ee
which is again an element of a larger ring discussed in relation to
$\widehat{\psi}_{\hbar}$ of Equation~\eqref{eq.psicomp}.
Using Equation~\eqref{eq.psid} and the five term relation for the
dilogarithm~\cite{Zagier:dilog}, it is easy to see that the right hand side of
Equation~\eqref{pentagon0} is given by
\be
\label{eq.Ip1}
\Big\langle\frac{\sqrt{z_1z_2}(1-z_1)(1-z_2)}{z_1+z_2-z_1z_2}
\exp\Big(\frac{\Li_{2}(z_1)}{\hbar}
+\frac{\Li_{2}(z_2)}{\hbar}
+\frac{\delta}{2}z^2\Big)
I_{0,\hbar}(z_1,z_2,z_0e^{w\hbar^{1/2}})
\Big\rangle_{w,\delta}\,.
\ee
The function $I_{m,\hbar}$ satisfies the system of linear $q$--difference equations
\be\label{eq.Ip2}
\begin{aligned}
I_{m,\hbar}(e^{\hbar}z_1,z_2,z)
&\=z\,z_2^{-1}(1-z_1z^{-1})\,
I_{m,\hbar}(z_1,z_2,z)\,,\\
I_{m,\hbar}(z_1,e^{\hbar}z_2,z)
&\=z\,z_1^{-1}(1-z_2z^{-1})\,
I_{m,\hbar}(z_1,z_2,z)\,,\\
I_{m,\hbar}(z_1,z_2,e^{\hbar}z)
&\=(1-z)\,z_1z_2z^{-2}e^{(m-1)\hbar}
I_{m,\hbar}(z_1,z_2,z)\,,\\
I_{m+1,\hbar}(z_1,z_2,z)
&\=
z\,I_{m,\hbar}(z_1,z_2,z)\,,
\end{aligned}
\ee
which can be derived from Equation~\eqref{eq.qdiff.of.psih}. In fact, the
function $I_{m,\hbar}$ is holonomic of rank $1$, a fact that we will not use
explicitly~\cite{PWZ}. Therefore, we see that
\be\label{eq.Ip3}
\begin{tiny}
\begin{aligned}
&I_{m,\hbar}(e^{\hbar}z_1,z_2,z)
\=z_2^{-1}I_{m+1,\hbar}(z_1,z_2,z)-z_1z_2^{-1}I_{m,\hbar}(z_1,z_2,z)\,,\\
&I_{m,\hbar}(e^{\hbar}z_1,z_2,(e^{\hbar}z_1+z_2-e^{\hbar}z_1z_2)e^{w\hbar^{1/2}})\\
&\=z_2^{-1}I_{m+1,\hbar}(z_1,z_2,(e^{\hbar}z_1+z_2-e^{\hbar}z_1z_2)e^{w\hbar^{1/2}})
-z_1z_2^{-1}I_{m,\hbar}(z_1,z_2,(e^{\hbar}z_1+z_2-e^{\hbar}z_1z_2)e^{w\hbar^{1/2}})\,,\\
% &\=z_2^{-1}I_{m+1,\hbar}(z_1,z_2,(z_1+z_2-z_1z_2)e^{w\hbar^{1/2}
% +\log(\frac{e^{\hbar}z_1+z_2-e^{\hbar}z_1z_2}{z_1+z_2-z_1z_2})})\\
% &\qquad\qquad-z_1z_2^{-1}I_{m,\hbar}(z_1,z_2,(z_1+z_2-z_1z_2)e^{w\hbar^{1/2}
% +\log(\frac{e^{\hbar}z_1+z_2-e^{\hbar}z_1z_2}{z_1+z_2-z_1z_2})})\,,\\
&I_{m,\hbar}(z_1,z_2,(z_1+z_2-z_1z_2)e^{z\hbar^{1/2}+\hbar})\\
&\=z_1z_2e^{(m-1)\hbar}I_{m-2,\hbar}(z_1,z_2,(z_1+z_2-z_1z_2)e^{z\hbar^{1/2}})
-z_1z_2e^{(m-1)\hbar}I_{m-1,\hbar}(z_1,z_2,(z_1+z_2-z_1z_2)e^{z\hbar^{1/2}})\,.\\
\end{aligned}
\end{tiny}
\ee
Let us define $\widehat{J}_{m,\hbar}$ and $J_{m,\hbar}$ by the equation
\be\label{eq.Ip4}
\begin{aligned}
\widehat{J}_{m,\hbar}(z_1,z_2)
&\=
\frac{1}{\sqrt{(1-z_1)(1-z_2)}}
\exp\Big(-\frac{\Li_{2}(z_1)}{\hbar}
-\frac{\Li_{2}(z_2)}{\hbar}\Big)
J_{m,\hbar}(z_1,z_2)\\
&\=\Big\langle\frac{\sqrt{z_1z_2(1-z_1)(1-z_2)}}{z_1+z_2-z_1z_2}
\exp\Big(
\frac{\delta}{2}z^2\Big)
I_{m,\hbar}(z_1,z_2,z_0e^{w\hbar^{1/2}})
\Big\rangle_{w,\delta}\,.
\end{aligned}
\ee
If we multiply both sides of the Equations~\eqref{eq.Ip3} with the
factors from Equation~\eqref{eq.Ip1}, take the bracket of both sides,
and apply Lemma~\ref{lem.ff2} to change
coordinates and the Gaussian, we find that $\widehat{J}_{\hbar,m}(z_1,z_2)$
satisfies the $q$--difference equations
\be
\begin{aligned}
\widehat{J}_{m,\hbar}(e^{\hbar}z_1,z_2)
&\=z_2^{-1}\widehat{J}_{m+1,\hbar}(z_1,z_2)
-z_1z_2^{-1}\widehat{J}_{m,\hbar}(z_1,z_2)\,,\\
\widehat{J}_{m,\hbar}(z_1,e^{\hbar}z_2)
&\=z_1^{-1}\widehat{J}_{m+1,\hbar}(z_1,z_2)
-z_1^{-1}z_2\widehat{J}_{m,\hbar}(z_1,z_2)\,,\\
\widehat{J}_{m,\hbar}(z_1,z_2)
&\=z_1z_2e^{(m-1)\hbar}\widehat{J}_{m-2,\hbar}(z_1,z_2)
-z_1z_2e^{(m-1)\hbar}\widehat{J}_{m-1,\hbar}(z_1,z_2)\,.
\end{aligned}
\ee
From these relations we can derive the equations
\be
\begin{aligned}
0&\=(z_1-z_1e^{(m+1)\hbar}+z_1^2e^{(m+1)\hbar}-z_1)\widehat{J}_{m,\hbar}(z_1,z_2)\\
&\qquad+(z_1z_2e^{(m+1)\hbar}-z_2+qz_1)\widehat{J}_{m,\hbar}(e^{\hbar}z_1,z_2)
+z_2\widehat{J}_{m,\hbar}(e^{2\hbar}z_1,z_2)\,,\\
0&\=(z_2-z_2e^{(m+1)\hbar}+z_2^2e^{(m+1)\hbar}-z_2)\widehat{J}_{m,\hbar}(z_1,z_2)\\
&\qquad+(z_2z_1e^{(m+1)\hbar}-z_1+qz_2)\widehat{J}_{m,\hbar}(z_1,e^{\hbar}z_2)
+z_1\widehat{J}_{m,\hbar}(z_1,e^{2\hbar}z_2)\,.\\
\end{aligned}
\ee
These $q$--difference equations give rise to equations for $J_{m,\hbar}$ given by
\be\label{eq.qJ}
\begin{aligned}
0&\=(z_1-z_1e^{\hbar}+z_1^2e^{\hbar}-z_1)J_{m,\hbar}(z_1,z_2)\\
&\qquad+(z_1z_2e^{\hbar}-z_2+qz_1)\sqrt{\frac{1-z_1}{1-e^{\hbar}z_1}}
\exp\Big(\frac{\Li_{2}(z_1)-\Li_{2}(e^{\hbar}z_1)}{\hbar}\Big)
J_{m,\hbar}(e^{\hbar}z_1,z_2)\\
&\qquad+z_2\sqrt{\frac{1-z_1}{1-e^{2\hbar}z_1}}\exp\Big(\frac{\Li_{2}(z_1)
-\Li_{2}(e^{2\hbar}z_1)}{\hbar}\Big)J_{m,\hbar}(e^{2\hbar}z_1,z_2)\,,\\
0&\=(z_2-z_2e^{\hbar}+z_2^2e^{\hbar}-z_2)J_{m,\hbar}(z_1,z_2)\\
&\qquad+(z_2z_1e^{\hbar}-z_1+qz_2)\sqrt{\frac{1-z_2}{1-e^{\hbar}z_2}}
\exp\Big(\frac{\Li_{2}(z_2)-\Li_{2}(e^{\hbar}z_2)}{\hbar}\Big)
J_{m,\hbar}(z_1,e^{\hbar}z_2)\\
&\qquad+z_1\sqrt{\frac{1-z_2}{1-e^{2\hbar}z_2}}
\exp\Big(\frac{\Li_{2}(z_2)-\Li_{2}(e^{2\hbar}z_2)}{\hbar}\Big)
J_{m,\hbar}(z_1,e^{2\hbar}z_2)\,.
\end{aligned}
\ee
These equations are also satisfied by the left hand side of
Equation~\eqref{pentagon0}, which follows from repeated application of
Equation~\eqref{eq.qdiff.of.psi}.

It is easy to see that both sides of Equation~\eqref{pentagon0} are formal power
series in $\hbar$ with coefficients rational functions of $(z_1,z_2)$ of
nonpositive degree with respect to both $z_1$ and $z_2$, thus their coefficients
embed in $\BQ\llbracket z_1^{-1}, z_2^{-1} \rrbracket$, and it suffices to prove
~\eqref{pentagon0} in the ring $\BQ\llbracket z_1^{-1}, z_2^{-1} \rrbracket
\llbracket \hbar \rrbracket$. In this case, we consider the linear $q$-difference
equations~\eqref{eq.qJ} over the ring $\BQ\llbracket z_1^{-1}, z_2^{-1} \rrbracket
\llbracket \hbar \rrbracket$. Looking at the corresponding Newton polytopes, 
the leading monomials that appear as coefficients of this $q$--difference equation
~\eqref{eq.qJ} at $\infty$ are $z_1^{2}(1+\mathrm{O}(z_1^{-1})\mathrm{O}(z_2^{0}))$
and $z_1^{2}z_2(1+\mathrm{O}(z_1^{-1})\mathrm{O}(z_2^{0}))$ for the first equation,
and $z_2^{2}(1+\mathrm{O}(z_1^{0})\mathrm{O}(z_2^{-1}))$ and
$z_1z_2^{2}(1+\mathrm{O}(z_1^{0})\mathrm{O}(z_2^{-1}))$ for the second equation,
respectively. The structure of these monomials and their Newton polytopes
shows that there is a unique solution to these equations in
$\BQ\llbracket z_1^{-1},z_2^{-1}\rrbracket\llbracket\hbar\rrbracket$ determined
by the value at $z_1=z_2=\infty$. Therefore, the identity in
Equation~\eqref{pentagon0} reduces to its specialization at $z_1=z_2=\infty$.
Since
\be
\begin{aligned}
\psi_\hbar(0,\infty)\,\psi_\hbar(0,\infty)
&\=
e^{\frac{\hbar}{6}}\quad\text{and}\\
e^{- \frac \hbar {24}}
\big\langle\psi_\hbar(-w,0)^2
\,\psi_\hbar(w,\infty)
\big\rangle_{w,\delta}
&\=
e^{- \frac \hbar {24}}
\Big\langle\exp\Big(
\frac{\hbar}{12}-\frac{z\hbar^{\frac{1}{2}}}{2}
\Big)\Big\rangle_{w,\delta}
\=
e^{\frac{\hbar}{6}}\,,
\end{aligned}
\ee
this completes the proof of Equation~\eqref{pentagon0}.

\medskip

Going back to the general case of $x,y$, Equation~\eqref{eq.pentagon} follows
from Equation~\eqref{eq.psid} together with~\eqref{pentagon0} using
Lemma~\ref{lem.ff2} to shift the Gaussian and with a change of integration
variable
\be
w \mapsto w+\frac{1}{\hroot}\log\left(\frac{z_1+z_2-z_1z_2}{
z_1e^{x_1\hbar^{1/2}}+z_2e^{x_2\hbar^{1/2}}-z_1z_2e^{(x_1+x_2)\hbar^{1/2}}}\right)
+x+y - \frac{xz_1+yz_2}{z_0}\,.
\ee
The detailed calculation is given in Appendix~\ref{app.pentagon}.
\end{proof}

\subsection{\texorpdfstring{Inversion formula for $\psi_\hbar$}{Inversion formula for psi}}
\label{sub.inversion}

In this section, we give an inversion formula for $\psi_\hbar$ analogous to the
inversion formula~\eqref{doublephi} for the Faddeev quantum dilogarithm.
Although we will not use this formula explicitly in our paper, we include it
for completeness.

\begin{lemma}
\label{lem.inversion}
We have:
\be
\label{inv1}
\psi_{\hbar}(x,z)
  \frac{1-z^{-1}e^{-x\hbar^{1/2}}}{1-z^{-1}}
  \psi_{\hbar}(-x,1/z) \= \exp\Big(-
  \frac{x\hbar^{\frac{1}{2}}}{2}+\frac{\hbar}{12}\Big),
\ee  
\be
\label{inv2}
\widehat{\psi}_{\hbar}(0,z)
\widehat{\psi}_{\hbar}(0,1/z) 
\= \frac{\sqrt{-z}}{1-z}\exp\Big(\frac{\pi^2}{6\hbar}+\frac{1}{2\hbar}
\log(-z)^2+\frac{\hbar}{12}\Big) \,.
\ee
\end{lemma}

\begin{proof}
These formulas all follow from the well-known inversion formulas
for the polylogarithms (see for example~\cite{Oesterle}):
\bea
  \Li_{2}(z)+\Li_{2}(1/z)
  &\=
  -\frac{\pi^2}{6}-\frac{1}{2}\log(-z)^2\,,\\
  \Li_{1}(z)-\Li_{1}(1/z)
  &\=
  -\log(-z)\,,\\
  \Li_{0}(z)+\Li_{0}(1/z)
  &\=
  -1\,,\\
  \Li_{-n}(z)+(-1)^n\Li_{-n}(1/z)
  &\=0 \qquad (n>0) \,.
\eea
Using these relations, we have
\bea
\widehat{\psi}_{\hbar}(0,z)
\widehat{\psi}_{\hbar}(0,1/z)
&\=
\exp\Big(\!
-\sum_{k\in\BZ_{\geq0}}
\frac{B_{k}\,\hbar^{k-1}}{k!}(\Li_{2-k}(z)+\Li_{2-k}(1/z))\Big)\\
&\=
\frac{\sqrt{-z}}{1-z}\exp\Big(\frac{\pi^2}{6\hbar}+\frac{1}{2\hbar}
\log(-z)^2+\frac{\hbar}{12}\Big)\,.
\eea
Similarly, we find that
\be
\begin{aligned}
  &\psi_{\hbar}(x,z)
  \frac{1-z^{-1}e^{-x\hbar^{1/2}}}{1-z^{-1}}
  \psi_{\hbar}(-x,1/z)\\
  &\=
  \exp\bigg(\!-\!\!\sum_{\underset{k+\frac{\ell}{2}>1}{k,\ell\in\BZ_{\geq0}}}
  \frac{B_{k}\,x^{\ell}\,\hbar^{k+\frac{\ell}{2}-1}}{\ell!\,k!}
  \big(\Li_{2-k-\ell}(z)+(-1)^{k+\ell}\Li_{2-k-\ell}(1/z)\big)\bigg)\\
  &\=
  \exp\Big(-
  \frac{x\hbar^{\frac{1}{2}}}{2}+\frac{\hbar}{12}\Big)\,.
\end{aligned}
\ee
Notice that this is exactly the factor that appeared in the remark after
Equation~\eqref{psih}, which gave the relation between the $\psi_{\hbar}$ used
in the current paper and the ones used in~\cite[Eq.~1.9]{DG}.
Indeed, $\frac{1-z^{-1}}{1-z^{-1}e^{-x\hbar^{1/2}}}\psi_{\hbar}(-x,1/z)^{-1}$ is
exactly the series used in~\cite{DG} where the factor
$\frac{1-z^{-1}}{1-z^{-1}e^{-x\hbar^{1/2}}}$ amounts to swapping the sign of $B_{1}$
as done there.
\end{proof}

%%%%%%%%%%%%%%%%%%%%%%%%%%%%%%%%%%%%%%%%%%%%%%%%%%%%%%%%%%%%%%%%%%%%%%%%%%%% 
%%%%%%%%%%%%%%%%%%%%%%%%%%%%%%%%%%%%%%%%%%%%%%%%%%%%%%%%%%%%%%%%%%%%%%%%%%%%

\section{Elementary invariance properties}
\label{sec.elementary}

In this section, we review some basic choices needed to define the
the Neumann--Zagier data for a triangulation with $N$ tetrahedra, namely:
\begin{itemize}
\item[(a)] an ordering of the $N$ tetrahedra,
\item[(b)] an ordering of the $N$ edges,
\item[(c)] an edge equation to remove,   
\item[(d)] a path to represent the meridian curve,
\item[(e)] a flattening.
\end{itemize}
A change of these choices changes the corresponding Neumann--Zagier data in a simple
way, which we now describe. Fix a triangulation and the choices needed to define
Neumann--Zagier data
\be
\Xi\=(A,B,\nu,z,f,f'')\,.
\ee
Regarding choice (a), suppose that $\sigma\in S_{N}$ is a permutation (and also the
associated matrix) of our labelling of the tetrahedra. Then the Neumann--Zagier
matrices transform as follows
\be
  \Xi\=(A,B,\nu,z,f,f'')
  \mapsto
  (A\sigma,B\sigma,\nu,\sigma^{-1} z,\sigma^{-1} f,\sigma^{-1} f'')
  \=\Xi\cdot\sigma \,.
\ee
This implies that the integrand $f^\Xi_\hbar(x,z)$ of Equation~\eqref{integrand}
and the quadratic form $\Lambda$ of Equation~\eqref{lambda} satisfy
\be
f_{\Xi\cdot\sigma,\hbar}(\sigma^{-1} x,\sigma^{-1} z) \= f^\Xi_\hbar(x,z)
\qquad
\Lambda^{\Xi\cdot\sigma} \= \sigma^t \Lambda^\Xi \sigma\,,
\ee
which combined with the fact that integration is invariant under a linear change
of variables (see part (a) of Lemma~\ref{lem.ff1}), implies that
$\Phi^{\Xi\cdot\sigma}(\hbar)=\Phi^\Xi(\hbar)$.  

Choices (b), (c) and (d) are a special case of the following transformation of
$P\in\GL_{N}(\BZ)$ acting on Neumann--Zagier data via
\be
\Xi\=(A,B,\nu,z,f,f'') \mapsto (PA,PB,P\nu,z,f,f'') \=P\cdot\Xi\,.
\ee
It follows that the integrand and the quadratic form satisfy 
\be
f^{P\cdot\Xi}_\hbar(x,z) \= f^\Xi_\hbar(x,z), \qquad
\Lambda^{P \cdot\Xi} \= \Lambda^\Xi \,,
\ee
which implies again that $\Phi^{\Xi\cdot\sigma}(\hbar)=\Phi^\Xi(\hbar)$.

Finally, if $\Xi$ and $\widetilde\Xi$ differ by a choice of flattening, then
it is easy to see that
\be
f^{\widetilde\Xi}_\hbar(x,z) \= e^{c\hbar} f^\Xi_\hbar(x,z), \qquad
\Lambda^{\widetilde\Xi} \= \Lambda^\Xi 
\ee
for some $c \in \frac{1}{8}\BZ$, which implies that
$\Phi^{\widetilde\Xi}(\hbar)=e^{c \hbar}\, \Phi^\Xi(\hbar)$. 

%\begin{proof}
%This follows from the fact that
%\be
%  B^{-1}A(f-g)+f''-g''\=0\,,
%\ee
%and so
%\be
%\begin{aligned}
%  f^{t}B^{-1}Af  -g^{t}B^{-1}Ag
%  &\=
%  f^{t}B^{-1}Af  -g^{t}B^{-1}Af +g^{t}B^{-1}Af -g^{t}B^{-1}Ag\\
%  &\=
%  (f-g)^{t}B^{-1}Af +g^{t}B^{-1}A(f-g)\\
%  &\=
%  (f+g)^{t}B^{-1}A(f-g)\\
%  &\=
%  (f+g)^{t}(f''-g'') \in\BZ\,.
%\end{aligned}
%\ee
%\end{proof}

%%%%%%%%%%%%%%%%%%%%%%%%%%%%%%%%%%%%%%%%%%%%%%%%%%%%%%%%%%%%%%%%%%%%%%%%%%%% 
%%%%%%%%%%%%%%%%%%%%%%%%%%%%%%%%%%%%%%%%%%%%%%%%%%%%%%%%%%%%%%%%%%%%%%%%%%%%

\section{Invariance under the choice of quad}
\label{sec.quad}

The definition of the series $\Phi^\Xi(\hbar)$ requires some choices,
some of which were described and dealt with in the previous
Section~\ref{sec.elementary}. What remains to complete the proof of
Theorem~\ref{thm.1} is the independence of $\Phi^\Xi(\hbar)$
under the choice of a nondegenerate quad, and the independence
under the 2--3 Pachner moves that connect two ideal triangulations.

In this section, we will prove that $\Phi^\Xi(\hbar)$ is independent
of the choice of a nondegenerate quad. Recall that to each pair of opposite
edges of an ideal tetrahedron there is an associated variable
called a shape variable. When defining Neumann--Zagier data (see
Section~\ref{sub.NZ}), we must
choose an edge for each tetrahedron, which associates one of these shapes.
There are three possible choices, which leads to the action on $\BZ/3\BZ$ on
the each column of the Neumann--Zagier data. All in all, for a triangulation
with $N$ tetrahedra, this leads to $3^{N}$ choices on of Neumann--Zagier
data. On the other hand, the definition of the series $\Phi^{\Xi}(\hbar)$
requires a choice of a nondegenerate quad, i.e., one for which $\det(B) \neq 0$
(such quads always exist~\cite[Lem.A.3]{DG}). In this section, we will show
that any of the $3^N$ choices of quad with $\det(B) \neq 0$
lead to the same $\Phi^{\Xi}(\hbar)$ series.

\begin{theorem}
\label{thm.quadinvariance}  
The series $\Phi^{\Xi}(\hbar)$ is independent of the choice of a nondegenerate
quad. 
\end{theorem}

\begin{proof}
Fix two non--degenerate NZ data $\Xi=(A,B,\nu,z,f,f'')$ and 
$\widetilde{\Xi}=(\widetilde{A},\widetilde{B},\widetilde{\nu},\widetilde{z},
\widetilde{f},\widetilde{f}'')$ related by a quad change of the same ideal
triangulation. The nondegeneracy assumption implies that $\det(B) \neq 0$ and 
$\det(\widetilde{B})\neq0$. 

After reordering the tetrahedra (which does not change the $\Phi^\Xi(\hbar)$
series, as follows from Section~\ref{sec.elementary}), we can assume
that $\widetilde{\Xi}$ is obtained by applying a change in quad that fixes
the first $N_0$ shapes $z^{(0)}$, replaces the next $N_1$ shapes $z^{(1)}$
by $z'^{(1)}$ and replaces the next $N_2$ shapes $z^{(2)}$ by $z'^{(2)}$. (Recall
that $z'=1/(1-z)$ and $z''=1-1/z$).

\medskip

This partitions the shapes $z=(z^{(0)},z^{(1)},z^{(2)})$ into three sets of 
size $N_0,N_1,N_2$ and the matrices $A$ and $B$ into three block matrices
$A_i$ and $B_i$ of size $N \times N_i$ for $i=0,1,2$
\bea
A &\= (A_0\,|\,A_1\,|\,A_2)\,,\qquad
B &\= (B_0\,|\,B_1\,|\,B_2)\,,\qquad
\eea
and similarly for the flattening $f=(f^{(0)},f^{(1)},f^{(2)})$ and
$f''=(f''^{(0)},f''^{(1)},f''^{(2)})$. After the quad moves, the corresponding
matrices and vectors are given by
\bea
\tilde A &\= (A_0\,|-B_1\,|-A_2+B_2)\,,\quad
& &\tilde B & &\!\!\!=\; (B_0\,|\,A_1-B_1\,|-A_2)\,,\\
\tilde \nu &\= \nu - B_1  1 - A_2 1\,,\quad
& &\tilde z & &\!\!\!=\; 
(z^{(0)},z'^{(1)},z''^{(2)})\,,\\
\tilde f &\= (f^{(0)},1-f^{(1)}-f''^{(1)},f''^{(2)})\,,\quad
& &\tilde f'' & &\!\!\!=\; 
(f''^{(0)},f^{(2)},1-f^{(2)}-f''^{(2)})\,.
\eea
We also partition the vector of formal Gaussian integration variables
$x=(x^{(0)},x^{(1)},x^{(2)})$, as well as the symmetric matrix $Q:=B^{-1}A$
\be
Q\=
\begin{pmatrix}
Q_{00} & Q_{01} & Q_{02}\\
Q_{01}^t & Q_{11} & Q_{12}\\
Q_{02}^t & Q_{12}^t & Q_{22}\\
\end{pmatrix}\,.
\ee
With the above notation, we have
\be
\Phi^{\Xi}(\hbar) \= \langle I_0 \rangle_{x,\Lambda_0}
\ee
where
\[
I_0 \= \exp\Big(\frac\hbar8 f^t B^{-1} A f
-\frac{\hsqrt}2 x^t(B^{-1}\nu-1)\Big)
\prod_{j=1}^N \psi_{\hbar}(x_j,z_j)
\]
and
\begin{equation*}
\Lambda_0\=\diag(z')-Q \,.
\end{equation*}

Applying the first quad move as in Theorem~\ref{thm.fourier} to the $\psi_\hbar$
with arguments of $z^{(1)}$ and the second quad move as in its
Corollary~\ref{cor.fourier} to the $\psi_\hbar$ with arguments $z^{(2)}$, we obtain
that
\be
\Phi^{\Xi}(\hbar) \= \langle I_1 \rangle_{(x,y,w), \Lambda_1}
\ee
where
\begin{align*}
I_1 &\= 
\exp\Big(\frac\hbar8 f^t B^{-1} A f-\frac{\hsqrt}2{x^{(0)}}^t(B^{-1}\nu-1)^{(0)}
-
\frac{N_1\hbar}{24}
- \frac \hsqrt 2 {x^{(1)}}^t(B^{-1}\nu-1)^{(1)}\\
&\qquad\qquad\quad+ \sum_{j=1}^{N_1}
\Big(y_j+\frac{x_j^{(1)}z_j^{(1)}}{1-z_j^{(1)}}\Big)
\frac\hsqrt2
+
\frac{N_2\hbar}{24}
- \frac \hsqrt 2 {x^{(2)}}^t(B^{-1}\nu-1)^{(2)}
-\sum_{j=1}^{N_2}x_j^{(2)} \frac\hsqrt2\Big)\\
&\times\prod_{j=1}^{N_0} \psi_{\hbar}(x_j^{(0)},z_j^{(0)})
\prod_{j=1}^{N_1}\psi_{\hbar}\Big(y_j+\frac{x_j^{(1)}z_j^{(1)}}{1-z_j^{(1)}},
\frac1{1-z_j^{(1)}}\Big)
\prod_{j=1}^{N_2}
\psi_{\hbar}\Big(w_j - \frac{x_j^{(2)}}{1-z_j^{(2)}},1-{z_j^{(2)}}^{-1}\Big),
\end{align*}
$y$ and $w$ are vectors of size $N_1$ and $N_2$, respectively, and  
\[
\Lambda_1 \= 
\begin{pmatrix}
\Lambda & 0 & 0 \\
0 & \diag(z''^{(1)}) & 0 \\
0 & 0 & \diag(z^{(2)})-I
\end{pmatrix} \,.
\]
Making the change of variables
$y_j \mapsto y_j - \frac{x_j^{(1)}z_j^{(1)}}{1-z_j^{(1)}}$
and $w_j \mapsto w_j + \frac{x_j^{(2)}}{1-z_j^{(2)}}$ and using Lemma~\ref{lem.ff1}
we obtain that
\be
\Phi^{\Xi}(\hbar) \= \langle I_2 \rangle_{(x,y,w), \Lambda_2}
\ee
where
\begin{align*}
I_2 &\= \exp\Big(\frac\hbar8 f^t B^{-1} A f
-\frac{\hsqrt}2{x^{(0)}}^t(B^{-1}\nu-1)^{(0)}
-
\frac{N_1\hbar}{24}
- \frac \hsqrt 2 {x^{(1)}}^t(B^{-1}\nu-1)^{(1)}\\
&\qquad\qquad\quad+ y^{t}1\frac\hsqrt2
+
\frac{N_2\hbar}{24}
- \frac \hsqrt 2 {x^{(2)}}^t(B^{-1}\nu)^{(2)}
\Big)\\
&\qquad\qquad\times\prod_{j=1}^{N_0} \psi_{\hbar}(x_j^{(0)},z_j^{(0)})
\prod_{j=1}^{N_1}\psi_{\hbar}\Big(y_j,\frac1{1-z_j^{(1)}}\Big)
\prod_{j=1}^{N_2}
\psi_{\hbar}\big(w_j,1-{z_j^{(2)}}^{-1}\big)
\end{align*}
and
\[
\Lambda_2 \=
\left(\!\!
\def\arraystretch{1.25}
\begin{array}{ccc|cc}
\diag(z'^{(0)}) - Q_{00}
& -Q_{01} & -Q_{02} & 0 & 0\\
-Q_{01}^t & I-Q_{11} & -Q_{12} & I & 0\\
-Q_{02}^t & -Q_{12}^t & -Q_{22} & 0 & -I\\
\hline
0 & I & 0 & \diag(z''^{(1)})& 0 \\
0 & 0 & -I & 0 & \diag(z^{(2)})-I
\end{array}
\!\!\right) \,.
\]
Note that $x^{(1)}$ and $x^{(2)}$ do not appear in the arguments of $\psi_\hbar$.
Moreover, since $BQ=A$, we see that
\be\label{eq.BBt}
B
\begin{pmatrix}
I & Q_{01} & Q_{02}\\
0 & Q_{11}-I & Q_{12}\\
0 & Q_{12}^{t} & Q_{22}
\end{pmatrix}
\begin{pmatrix}
I & 0 & 0\\
0 & I & 0\\
0 & 0 & -I
\end{pmatrix}
\=
(B_0\,|\,A_1-B_1\,|-A_2)
\=
\widetilde{B}\,,
\ee
which implies that
\be
\det
\begin{pmatrix}
Q_{11}-I & Q_{12}\\
Q_{12}^{t} & Q_{22}
\end{pmatrix}\neq0\,,
\quad\text{and}\quad
\begin{pmatrix}
Q_{11}-I & Q_{12}\\
Q_{12}^{t} & Q_{22}
\end{pmatrix}^{-1}
\=
\begin{pmatrix}
0 & I & 0\\
0 & 0 & -I\\
\end{pmatrix}
\widetilde{B}^{-1}B
\begin{pmatrix}
0 & 0\\
I & 0\\
0 & I
\end{pmatrix}\,.
\ee
Therefore, we can apply Fubini's Theorem (Lemma~\ref{lem.ff1}) with the
integration variables $x^{(1)},x^{(2)}$,
% so that in the notation there
% (beware the abuse of notation with
% the two uses of $A$ and $B$ in
% these equations)
% \begin{equation*}
% \begin{small}
% \begin{aligned}
% A \= \begin{pmatrix}
% I-Q_{11} & -Q_{12}\\-Q_{12}^t&-Q_{22}
% \end{pmatrix}, \quad 
% B \=\begin{pmatrix}
% -Q_{01}^t & I & 0 \\
% -Q_{02}^t & 0 & -I
% \end{pmatrix}, \quad
% F \= -\frac{1}{2}\left[\begin{pmatrix}0 & I & 0\\0 & 0 & I \end{pmatrix}
% B^{-1}\nu
% -\begin{pmatrix}1 \\ 0\end{pmatrix}\right]^t 
% \end{aligned}
% \end{small}
% \end{equation*}
% to the $x^{(1)}$ and $x^{(2)}$ integration variables.
and use Lemma~\ref{lem.quad.inv} and
the equality $Qf+f''=B^{-1}\nu$
to obtain that 
\be
\label{eq.qpf}
\Phi^{\Xi}(\hbar) \= e^{c\hbar}\langle I_3 \rangle_{\widetilde{x}, \Lambda_3}\,.
\ee
where $\widetilde{x}=(x^{(0)},y,w)$, $c\in\frac{1}{24}\BZ$
\begin{equation*}
I_3 \= \exp\Big(\frac\hbar8 \widetilde{f}^t \widetilde{B}^{-1}\widetilde{A}
\widetilde{f}
-\frac{\hsqrt}2 \widetilde{x}^t(\widetilde{B}^{-1}\widetilde{\nu}-1)\Big)
\prod_{j=1}^N \psi_{\hbar}(\widetilde{x}_j,\widetilde{z}_j)
\end{equation*}
and
\begin{equation*}
\Lambda_3\=\diag(\widetilde{z}')-\widetilde{B}^{-1}\widetilde{A}\,.
\end{equation*}
The right hand side of Equation~\eqref{eq.qpf} is exactly equal to
$e^{c\hbar}\Phi^{\widetilde{\Xi}}(\hbar)$, completing the proof of
Theorem~\ref{thm.quadinvariance}.
\end{proof}

\begin{lemma}
\label{lem.quad.inv}
With the notation as in the proof of Theorem~\ref{thm.quadinvariance}, we have
the following identities
\begin{align}
\label{eq.ABtildeq}
\widetilde{B}^{-1}\widetilde{A}
&\=
\begin{pmatrix}
Q_{00} & 0 & 0\\
0 & 0 & 0 \\
0 & 0 & -I
\end{pmatrix}
+
\begin{pmatrix}
Q_{01} & Q_{02}\\
-I & 0 \\
0 & I
\end{pmatrix}
\begin{pmatrix}
Q_{11}-I & Q_{12}\\
Q_{12}^{t} & Q_{22}
\end{pmatrix}^{-1}
\begin{pmatrix}
Q_{01}^t & -I & 0 \\
Q_{02}^t & 0 & I
\end{pmatrix},\\
\widetilde{B}^{-1}\widetilde{\nu}-1
&\=
\begin{pmatrix}
I & 0 & 0\\
0 & 0 & 0\\
0 & 0 & 0
\end{pmatrix}(B^{-1}\nu-1)
-\begin{pmatrix}0\\ 1 \\ 0\end{pmatrix}\notag\\
&\qquad+\begin{pmatrix}
-Q_{01} & -Q_{02}\\
I & 0 \\
0 & -I
\end{pmatrix}
\begin{pmatrix}
Q_{11}-I & Q_{12}\\
Q_{12}^{t} & Q_{22}
\end{pmatrix}^{-1}\left(\begin{pmatrix}0 & I & 0\\0 & 0 & I \end{pmatrix}B^{-1}\nu
-\begin{pmatrix}1 \\ 0\end{pmatrix}\right)\label{eq.Bvtildeq}
\end{align}
and for some $d\in\BZ$
\be\label{eq.fftildeq}
\begin{aligned}
&\widetilde{f}\widetilde{B}^{-1}\widetilde{\nu}
\=
f^tB^{-1}\nu+d\\
&-\Big((B^{-1}\nu)^t\begin{pmatrix}
0 & 0\\
I & 0\\
0 & I
\end{pmatrix}
-\begin{pmatrix}1 & 0\end{pmatrix}
\Big)
\begin{pmatrix}
Q_{11}-I & Q_{12}\\
Q_{12}^{t} & Q_{22}
\end{pmatrix}^{-1}
\left(\begin{pmatrix}0 & I & 0\\0 & 0 & I \end{pmatrix}B^{-1}\nu
-\begin{pmatrix}1 \\ 0\end{pmatrix}\right).
\end{aligned}
\ee
\end{lemma}

\begin{proof}
We will show that boths sides of Equation~\eqref{eq.ABtildeq} and
Equation~\eqref{eq.Bvtildeq} are equal after multiplying by the invertible
matrix $\widetilde{B}$. Denote
\be
\begin{pmatrix}
Q_{11}-I & Q_{12}\\
Q_{12}^{t} & Q_{22}
\end{pmatrix}^{-1}
\=
\begin{pmatrix}
\Gamma_{11} & \Gamma_{12}\\
\Gamma_{12}^{t} & \Gamma_{22}
\end{pmatrix}\,,
\ee
and note that
\be\label{eq.qlq}
\begin{tiny}
\begin{aligned}
&
\widetilde{B}
\begin{pmatrix}
-Q_{01} & -Q_{02}\\
I & 0 \\
0 & -I
\end{pmatrix}
\begin{pmatrix}
Q_{11}-I & Q_{12}\\
Q_{12}^{t} & Q_{22}
\end{pmatrix}^{-1}
\=
\widetilde{B}
\begin{pmatrix}
  -Q_{01}\Gamma_{11}-Q_{02}\Gamma_{12}^t
  &-Q_{01}\Gamma_{12}-Q_{02}\Gamma_{22}\\
  \Gamma_{11} & \Gamma_{12}\\
  -\Gamma_{12}^{t} &-\Gamma_{22}
\end{pmatrix}\\
&\=
(
-B_{0}Q_{01}\Gamma_{11}-B_0Q_{02}\Gamma_{12}^t+(A_1-B_1)\Gamma_{11}+A_2\Gamma_{12}^{t}
\;|\;-B_0Q_{01}\Gamma_{12}-B_0Q_{02}\Gamma_{22}+(A_1-B_1)\Gamma_{12}+A_2\Gamma_{22}
)
\\
&\=
(
(B_1(Q_{11}-I)+B_2Q_{12}^t)\Gamma_{11}+(B_1Q_{12}+B_2Q_{22})\Gamma_{12}^t
\;|\;(B_1(Q_{11}-I)+B_2Q_{12}^t)\Gamma_{12}+(B_1Q_{12}+B_2Q_{22})\Gamma_{22}
)
\\
&\=
(
B_1\;|\;B_2
)\,.
\end{aligned}
\end{tiny}
\ee
Therefore, since
\be
\widetilde{B}
\begin{pmatrix}
Q_{00} & 0 & 0\\
0 & 0 & 0 \\
0 & 0 & -I
\end{pmatrix}
\=(B_0Q_{00}\;|\;0\;|-A_2)\,,
\ee
we see that multiplying the right hand side of
Equations~\eqref{eq.ABtildeq} on the left by $\widetilde{B}$ we have
\be
  (B_0Q_{00}+B_1Q_{01}^t+B_{2}Q_{02}^t\;|-B_1\;|-A_2+B_2)
  \=\widetilde{A}\,,
\ee
which completes the proof of Equation~\eqref{eq.ABtildeq}.
Similarly, since
\be
\widetilde{B}
\begin{pmatrix}
I & 0 & 0\\
0 & 0 & 0\\
0 & 0 & 0
\end{pmatrix}(B^{-1}\nu-1)
-\widetilde{B}\begin{pmatrix}0\\ 1 \\ 0\end{pmatrix}
\=
B_0(B^{-1}\nu)^{(0)}-B_01
-A_11+B_11
\ee
we see that multiplying the right hand side of
Equations~\eqref{eq.Bvtildeq} on the left by $\widetilde{B}$ we have
\be
  B_0(B^{-1}\nu)^{(0)}-B_01
  -A_11+B_11
  +B_1(B^{-1}\nu)^{(1)}-B_11
  +B_2(B^{-1}\nu)^{(2)}
  \=\widetilde{\nu}-\widetilde{B}1\,,
\ee
which completes the proof of Equation~\eqref{eq.Bvtildeq}.
For Equation~\eqref{eq.fftildeq}, we will compute the three terms that appear there.
Firstly, use Equation~\eqref{eq.Bvtildeq} to obtain that
\be
\begin{pmatrix}
0 & I & 0\\
0 & 0 & -I
\end{pmatrix}
(\widetilde{B}^{-1}\widetilde{\nu}-1)
\=
-\begin{pmatrix}
1\\0
\end{pmatrix}
+
\begin{pmatrix}
Q_{11}-I & Q_{12}\\
Q_{12}^{t} & Q_{22}
\end{pmatrix}^{-1}\left(\begin{pmatrix}0 & I & 0\\0 & 0 & I \end{pmatrix}B^{-1}\nu
-\begin{pmatrix}1 \\ 0\end{pmatrix}\right),
\ee
to obtain that
\be\label{eq.lem.qinter}
\begin{aligned}
&\left((B^{-1}\nu)^t\begin{pmatrix}
0 & 0\\
I & 0\\
0 & I
\end{pmatrix}
-\begin{pmatrix}1 & 0\end{pmatrix}
\right)
\left(
\begin{pmatrix}
0 & I & 0\\
0 & 0 & -I
\end{pmatrix}
(\widetilde{B}^{-1}\widetilde{\nu}-1)+\begin{pmatrix}1 \\ 0\end{pmatrix}
\right)\\
&\=
({(B^{-1}\nu)^{(1)}}^t-1)
(\widetilde{B}^{-1}\widetilde{\nu})^{(1)}
-{(B^{-1}\nu)^{(2)}}^t
(\widetilde{B}^{-1}\widetilde{\nu})^{(2)}+{(B^{-1}\nu)^{(2)}}^t1\,.
\end{aligned}
\ee
Secondly, expanding out
\be
\begin{aligned}
\widetilde{f}^t\widetilde{B}^{-1}\widetilde{\nu}
&\=
f^{(0)}(\widetilde{B}^{-1}\widetilde{\nu})^{(0)}
+(1-f^{(1)}-f''^{(1)})(\widetilde{B}^{-1}\widetilde{\nu})^{(1)}
+f''^{(2)}(\widetilde{B}^{-1}\widetilde{\nu})^{(2)}\\
&\=
f^{(0)}(\widetilde{B}^{-1}\widetilde{\nu})^{(0)}\\
&\qquad+(1-f^{(1)}-{(B^{-1}\nu)^{(1)}}^t+f^{(0)}Q_{01}+f^{(1)}Q_{11}+f^{(2)}Q_{12}^t)
(\widetilde{B}^{-1}\widetilde{\nu})^{(1)}\\
&\qquad+({(B^{-1}\nu)^{(2)}}^t-f^{(0)}Q_{02}-f^{(1)}Q_{12}-f^{(2)}Q_{22})
(\widetilde{B}^{-1}\widetilde{\nu})^{(2)}\,.
\end{aligned}
\ee
Finally, using Equation~\eqref{eq.BBt} we obtain that
\be
\begin{aligned}
f^tB^{-1}\nu
-f^{(1)}1
+f''^{(2)}1
&\=
f^tB^{-1}\nu
-f^{(1)}1
-f^{t}Q(0,0,1)^t
+{(B^{-1}\nu)^{(2)}}^t1\\
&\=f^tB^{-1}\widetilde{\nu}+{(B^{-1}\nu)^{(2)}}^t1\\
&\=
f^{(0)}(\widetilde{B}^{-1}\widetilde{\nu})^{(0)}\\
&\qquad+(f^{(0)}Q_{01}+f^{(1)}(Q_{11}-I)+f^{(2)}Q_{12}^t)
(\widetilde{B}^{-1}\widetilde{\nu})^{(1)}\\
&\qquad-(f^{(0)}Q_{02}+f^{(1)}Q_{12}+f^{(2)}Q_{22})
(\widetilde{B}^{-1}\widetilde{\nu})^{(2)}
+{(B^{-1}\nu)^{(2)}}^t1\,.
\end{aligned}
\ee
\end{proof}

%%%%%%%%%%%%%%%%%%%%%%%%%%%%%%%%%%%%%%%%%%%%%%%%%%%%%%%%%%%%%%%%%%%%%%%%%%%% 
%%%%%%%%%%%%%%%%%%%%%%%%%%%%%%%%%%%%%%%%%%%%%%%%%%%%%%%%%%%%%%%%%%%%%%%%%%%%

\section{Invariance under Pachner moves}
\label{sec.pachner}

In this section, we will prove that $\Phi^{\Xi}$ is invariant under 2--3 Pachner
moves. There are several versions of the 2--3 move (and of the corresponding
pentagon identity in Teichm\"uller TQFT~\cite{Kashaev:beta}) and the
one we choose in the next theorem is slightly different from the one
in~\cite[Sec.3.6]{DG} and can be related by
composing with quad moves.

The 2--3 move involves two triangulations $\calT$ and $\widetilde\calT$ with $N+2$
and $N+3$ tetrahedra, respectively, shown in Figure~\ref{fig.23move}.

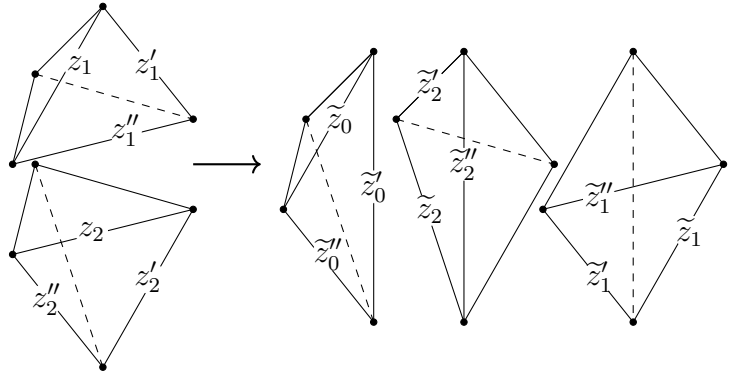
\begin{figure}[htpb!]
\begin{center}  
%\begin{tikzpicture}[scale=0.8,baseline=-3]
\begin{tikzpicture}[scale=0.6,baseline=-3]
\filldraw (2,3.5) circle (2pt);
\filldraw (0,0) circle (2pt);
\filldraw (4,1) circle (2pt);
\filldraw (0.5,2) circle (2pt);
\draw (0,0) -- (2,3.5);
\draw (0,0) -- (4,1);
\draw (0,0) -- (0.5,2);
\draw (4,1) -- (2,3.5);
\draw (0.5,2) -- (2,3.5);
\draw[dashed] (0.5,2) -- (4,1);

\filldraw (0,0-2) circle (2pt);
\filldraw (4,1-2) circle (2pt);
\filldraw (0.5,2-2) circle (2pt);
\filldraw (2,-2.5-2) circle (2pt);
\draw (0,0-2) -- (2,-2.5-2);
\draw (0,0-2) -- (4,1-2);
\draw (0,0-2) -- (0.5,2-2);
\draw (4,1-2) -- (2,-2.5-2);
\draw[dashed] (0.5,2-2) -- (2,-2.5-2);
\draw (0.5,2-2) -- (4,1-2);

\draw[thick,->] (4,0) -- (5.5,0);

\filldraw (6+2,3.5-1) circle (2pt);
\filldraw (6+0,0-1) circle (2pt);
\filldraw (6+0.5,2-1) circle (2pt);
\filldraw (6+2,-2.5-1) circle (2pt);
\draw (6+0,0-1) -- (6+2,3.5-1);
\draw (6+0,0-1) -- (6+0.5,2-1);
\draw (6+0.5,2-1) -- (6+2,3.5-1);
\draw (6+0.5,2-1) -- (6+2,3.5-1);
\draw (6+0,0-1) -- (6+2,-2.5-1);
\draw[dashed] (6+0.5,2-1) -- (6+2,-2.5-1);
\draw (6+2,3.5-1) -- (6+2,-2.5-1);

\filldraw (8+2,3.5-1) circle (2pt);
\filldraw (8+4,1-1) circle (2pt);
\filldraw (8+0.5,2-1) circle (2pt);
\filldraw (8+2,-2.5-1) circle (2pt);
\draw (8+4,1-1) -- (8+2,3.5-1);
\draw (8+0.5,2-1) -- (8+2,3.5-1);
\draw[dashed] (8+0.5,2-1) -- (8+4,1-1);
\draw (8+0.5,2-1) -- (8+2,3.5-1);
\draw (8+4,1-1) -- (8+2,-2.5-1);
\draw (8+0.5,2-1) -- (8+2,-2.5-1);
\draw (8+2,3.5-1) -- (8+2,-2.5-1);

\filldraw (11.75+2,3.5-1) circle (2pt);
\filldraw (11.75+0,0-1) circle (2pt);
\filldraw (11.75+4,1-1) circle (2pt);
\filldraw (11.75+2,-2.5-1) circle (2pt);
\draw (11.75+0,0-1) -- (11.75+2,3.5-1);
\draw (11.75+0,0-1) -- (11.75+4,1-1);
\draw (11.75+4,1-1) -- (11.75+2,3.5-1);
\draw (11.75+0,0-1) -- (11.75+2,-2.5-1);
\draw (11.75+4,1-1) -- (11.75+2,-2.5-1);
\draw[dashed] (11.75+2,3.5-1) -- (11.75+2,-2.5-1);

\filldraw[white] (1.5,2.2) circle (10pt);
\filldraw[white] (3,2.2) circle (10pt);
\filldraw[white] (2.5,0.75) circle (10pt);
\filldraw[white] (1.75,-1.5) circle (10pt);
\filldraw[white] (3,-2.5) circle (10pt);
\filldraw[white] (0.75,-3) circle (10pt);
\filldraw[white] (8,-0.5) circle (10pt);
\filldraw[white] (7.25,1) circle (10pt);
\filldraw[white] (7,-2) circle (10pt);
\filldraw[white] (15,-1.5) circle (10pt);
\filldraw[white] (13,-2.5) circle (10pt);
\filldraw[white] (13,-0.75) circle (10pt);
\filldraw[white] (9.25,-1) circle (10pt);
\filldraw[white] (10,0) circle (10pt);
\filldraw[white] (9.25,1.75) circle (10pt);

\node at (1.5,2.2) {$z_1$};
\node at (3,2.2) {$z_1'$};
\node at (2.5,0.75) {$z_1''$};
\node at (1.75,-1.5) {$z_2$};
\node at (3,-2.5) {$z_2'$};
\node at (0.75,-3) {$z_2''$};
\node at (7.25,1) {$\widetilde{z}_0$};
\node at (8,-0.5) {$\widetilde{z}_0'$};
\node at (7,-2) {$\widetilde{z}_0''$};
\node at (15,-1.5) {$\widetilde{z}_1$};
\node at (13,-2.5) {$\widetilde{z}_1'$};
\node at (13,-0.75) {$\widetilde{z}_1''$};
\node at (9.25,-1) {$\widetilde{z}_2$};
\node at (10,0) {$\widetilde{z}_2''$};
\node at (9.25,1.75) {$\widetilde{z}_2'$};

\end{tikzpicture}
\caption{The 2--3 Pachner move.}
\label{fig.23move}
\end{center}
\end{figure}

\noindent
Using $z_0=z_1+z_2-z_1z_2$ from Equation~\eqref{5z} used in
Theorem~\ref{thm.pentagon}, it follows that the shapes $z$ of $\calT$ and
$\widetilde z$ of $\widetilde\calT$ are related by
\bea
\label{z23}
z\=(z_1,z_2,z^{*}) \mapsto \widetilde z
\= &\big(
\widetilde{z}_0,\; \widetilde{z}_1,\; \widetilde{z}_2,\;\widetilde{z}^{*}
\big)\\
\=&\big(
z_0,\; z_1z_0^{-1},\; z_2z_0^{-1},\;z^{*}
\big)\,.
\eea
Similarly to~\cite[Eq. (3.20) and (3.21)]{DG}, these shapes satisfy the relations
\bea
\label{eq.23sr}
\widetilde{z}_0'&\=z_1'z_2'\,,\quad
& \widetilde{z}_1''&\=z_1''z_2\,,\quad
& \widetilde{z}_2''&\=z_1z_2''\,,\\
z_1&\=\widetilde{z}_0\widetilde{z}_1\,,\quad
& z_1'&\=\widetilde{z}_1'\widetilde{z}_2\,,\quad
& z_1''&\=\widetilde{z}_0''\widetilde{z}_2'\,,\\
z_2&\=\widetilde{z}_0\widetilde{z}_2\,,\quad
& z_2'&\=\widetilde{z}_1\widetilde{z}_2'\,,\quad
& z_2''&\=\widetilde{z}_0''\widetilde{z}_1'\,.
\eea
If we write the Neumann--Zagier matrices of $\calT$ in the form
\be
\label{AB2}
A \= (a_1\,|\,a_2\,|\,a_*)\,,\qquad B=\; (b_1\,|\,b_2\,|\,b_*)\,,
\ee
where $a_1$, $a_2$ are the first two columns of $A$ and $a_*$ is the block
$(N+2)\times N$ matrix of the remaining $N$ columns of $A$, and likewise for $B$,
then using Equations~\eqref{eq.23sr}
the corresponding Neumann--Zagier matrices of $\widetilde\calT$ are given by
\be
\label{AB3}
\widetilde A \= \begin{pmatrix}
 -1 & 0 & 0 & 0 \\
 -b_1-b_2+a_1+a_2 & a_1-b_2 & a_2-b_1 & a_*
 \end{pmatrix}, \qquad 
\widetilde B \= \begin{pmatrix}
 -1 & 1 & 1 & 0\\
 0 & b_1 & b_2 & b_*
\end{pmatrix}
\ee
and the corresponding vector $\widetilde\nu=(1,\nu)$. Analogously to the shapes,
we will fix flattenings
$(\widetilde{f}_0,\widetilde{f}_1,\widetilde{f}_2,\widetilde{f}^{*})$ and
$(\widetilde{f}_0'',\widetilde{f}_1'',\widetilde{f}_2'',\widetilde{f}''^{*})$ for
$(\widetilde{A}\,|\,\widetilde{B})$ and
$(f_1,f_2,f^{*})=(\widetilde{f}_0+\widetilde{f}_1,\widetilde{f}_0
+\widetilde{f}_2,f^{*})$ and $(f''_1,f''_2,f''^{*})=
(\widetilde{f}_0''+\widetilde{f}_2', \widetilde{f}_0''+\widetilde{f}_1',f''^{*})$
for $(A\,|\,B)$. The data of the flattenings then satisfy the additive versions
of Equations~\eqref{eq.23sr}.

\begin{theorem}
\label{thm.23}
The series $\Phi^{\Xi}(\hbar)$ is invariant under 2--3 Pachner moves.
\end{theorem}

The proof involves an application of the pentagon identity for $\psi_{\hbar}$
of Theorem~\ref{thm.pentagon}. 

For two triangulations $\calT$ and $\widetilde\calT$ with $N+2$ and $N+3$
tetrahedra and NZ matrices $(\calA \,|\,\calB)$, respectively, related by a Pachner
2--3 move. To define the corresponding series $\Phi^{\Xi}(\hbar)$ and
$\Phi^{\widetilde\Xi}(\hbar)$, we need to possibly change quads on $\calT$ and
$\widetilde\calT$ so that both $\calB$ and $\widetilde\calB$ are invertible. By a
quad move $q$, we can replace $(\calA \,|\,\calB)$ by $(\mathbf{A} \,|\,\mathbf{B})$
where $\det(\mathbf{B}) \neq 0$. Recalling that quad moves act on tetrahedra and
actions on different tetrahedra commute, we can write the move $q=q_2 \times q_N$
as a product of quad moves $q_2$ on the first 2 tetrahedra times moves $q_N$ on
the remaining $N$ tetrahedra of $\calT$. Let $(A \,|\,B)$ denote the result of
applying the move $1 \times q_N$ on the $N+2$ tetrahedra of $\calT$. 

Since the $(N+2)\times(N+2)$ matrix $\mathbf{B}$ has full rank and $B$ and
$\mathbf{B}$ have the same last $N$ columns, it follows that $B$ has nullity
$0$, $1$ or $2$. 

Now on $\widetilde\calT$, we can apply the identity move on the first three
tetrahedra and the $q_N$ moves on the remaining $N$ tetrahedra, which transforms
the NZ matrices $(\widetilde{\calA} \,|\, \widetilde{\calB})$
to $(\widetilde A \,|\,\widetilde B)$, and further apply the $q_2$ moves on
the second and third tetrahedra to obtain the NZ matrices
$(\widetilde{\mathbf{A}} \,|\,\widetilde{\mathbf{B}})$. By looking at how the
matrix $B$ transforms under a 2--3 move (see Equation~\eqref{AB3}), it follows
that $\widetilde{\mathbf{B}}$ has the same rank as $\mathbf{B}$, and hence
$\det(\widetilde{\mathbf{B}}) \neq 0$.

The above discussion can be summarized in the following commutative diagram.

\be
\begin{tikzcd}
(\mathcal{A}\;|\;\mathcal{B})
\arrow[rrr,"2 \to 3"]\arrow[rd,"1 \times q_N"]\arrow[dd,"q_2 \times q_N"] &  &
& (\widetilde{\mathcal{A}}\;|\;\widetilde{\mathcal{B}})
\arrow[dl,"\!\!\!\!1 \times 1 \times q_N"]\arrow[dd,"1 \times q_2 \times q_N"]\\
& (A\;|\;B)\arrow[r,"2 \to 3"]\arrow[dl,"q_2 \times 1"]
& (\widetilde{A}\;|\;\widetilde{B})\arrow[dr,"1 \times q_2 \times 1"] & \\
(\mathbf{A}\;|\;\mathbf{B}) &   &
& (\widetilde{\mathbf{A}}\;|\;\widetilde{\mathbf{B}})
\end{tikzcd}
\ee
where $\det(\mathbf{B}) \neq 0$ and $\det(\widetilde{\mathbf{B}}) \neq 0$.
Now $\Phi^\Xi(\hbar)$ and $\Phi^{\widetilde\Xi}(\hbar)$ can be defined using 
the nondegenerate NZ data with matrices $(\mathbf{A} \,|\,\mathbf{B})$ and 
$(\widetilde{\mathbf{A}} \,|\,\widetilde{\mathbf{B}})$, respectively. 
We will show that the two $\hbar$-series are equal using a diagram
\be
\label{2Phi}
\begin{tikzcd}
\Phi^\Xi(\hbar)\arrow[rr,"\otimes^i \text{Fourier}"] & & \bullet
\arrow[rr,"\text{Pentagon}"] & & \bullet
\arrow[rr,"\otimes^i \text{Fourier}^{-1}"] & & \bullet
\arrow[rr,"2i-\text{Fubini}"] & & \Phi^{\widetilde\Xi}(\hbar)
\end{tikzcd}
\ee
where $i=0,1,2$ denotes the nullity of $B$. We will treat each case in a separate
section. With this discussion the Theorem~\ref{thm.23} follows from
Propositions~\ref{prop.case0},~\ref{prop.case1}, and~\ref{prop.case2}.
Before proceeding with the proof, we will set some
notation. Similarly to the proof of the quad invariance, denote
$\mbfQ=\mathbf{B}^{-1}\mathbf{A}$
%and if $\det(B)\neq0$ denote $Q=B^{-1}A$
\bea
\label{eq:BinvA_def}
\mathbf Q \=
\def\arraystretch{1.15}
\left(\!\!
\begin{array}{cc|c}
\mathbf Q_{11} & \mathbf Q_{12} & \mathbf Q_1^*\\
\mathbf Q_{12} & \mathbf Q_{22} & \mathbf Q_2^*\\
\hline
\mathbf {Q_1^*}^t & \mathbf {Q_2^*}^t & \mathbf Q^*\\
\end{array}\!\!\right),
%\qquad\text{and}\qquad
%Q \= 
%\left(\!\!
%\def\arraystretch{1.15}
%\begin{array}{cc|c}
%Q_{11} & Q_{12} & Q_1^{*}\\
%Q_{12}^t & Q_{22} & Q_2^{*}\\
%\hline
%{Q_1^{*}}^t & {Q_2^{*}}^t & Q^*
%\end{array}\!\!\right),
\eea
where $\mbfQ_{ij}$ are matrices of size $1\times 1$, $\mbfQ_{j}^{*}$
are matrices of size $1\times N$ and $\mbfQ^{*}$ is a matrix of size
$N\times N$.

\subsection{\texorpdfstring{The case of $B$ with full rank}{The case of B with full rank}}
\label{sub.case0}

In this section, we prove Theorem~\ref{thm.pentagon} under the assumption that
the matrix $B$ has full rank. In this case, Equation~\eqref{2Phi} simplifies
to the following one
\be
\label{2Phi.0}
\begin{tikzcd}
\Phi^\Xi(\hbar)\arrow[rr,"\text{Pentagon}"] & & \Phi^{\widetilde\Xi}(\hbar)
\end{tikzcd}
\ee
since we do not need to apply Fourier transform, nor Fubini's theorem.

\begin{proposition}
\label{prop.case0}
If $B$ has full rank, then $\Phi^{\Xi}(\hbar)$ is invariant under the 2--3
Pachner move given in Equation~\eqref{AB3}.
\end{proposition}

\begin{proof}
To begin with, noting that in this case $A=\mathbf{A}$ and $B=\mathbf{B}$, we have:
\be
\Phi^{\Xi}(\hbar) \= \langle I_0 \rangle_{x,\Lambda}
\ee
where
\[
I_0 \= \exp\Big(\frac\hbar8 f^t B^{-1} A f
-\frac{\hsqrt}2 x^t(B^{-1}\nu-1)\Big)
\prod_{j=1}^{N+2} \psi_{\hbar}(x_j,z_j)
\]
and
\[
\Lambda_0\=\diag(z')-\mbfQ 
\]
and $x=(x_1,x_2,x^*)$ are the integration variables. Applying the pentagon
identity of Theorem~\ref{thm.pentagon}, to the $\psi_{\hbar}$ with arguments
$x_1,x_2$ and introducing a new integration variable $x_0$, we obtain that
\be
\Phi^{\Xi}(\hbar) \= \langle I_1 \rangle_{(x_0,x),\Lambda_1}
\ee
where
\[
\begin{small}
\begin{aligned}
I_1 &\= \exp\Big(\!-\frac{\hbar}{24}+\frac\hbar8 f^t B^{-1} A f
-\frac{\hsqrt}2 x^t(B^{-1}\nu-1)\Big)\,
\psi_\hbar\Big(-x_0 -x_2
+\frac{x_1z_2+x_2z_1}{z_0},
z_1z_0^{-1}\Big)\\
&\!\!\times
\psi_\hbar\Big(x_0+x_1+x_2
-\frac{x_1z_2+x_2z_1}{z_0},
z_0\Big)
\,\psi_\hbar\Big(-x_0-x_1
+\frac{x_1z_2+x_2z_1}{z_0},
z_2z_0^{-1}\Big)
\prod_{j=1}^N \psi_{\hbar}(x_j^{*},z_j^{*})
\end{aligned}
\end{small}
\]
and 
\[
\Lambda_1 \= 
\begin{pmatrix}
 \frac{({z_1}+{z_2}-z_1z_2)^2}{({z_1}-1) {z_1} ({z_2}-1){z_2}}&
 0 \\
 0 & \Lambda
\end{pmatrix} \,.
\]
Making a change of variables $x_0\mapsto x_0+x_1(-1+z_2z_0^{-1})+x_2(-1+z_1z_0^{-1})$
using Lemma~\ref{lem.ff1}, we obtain that
\bea
\Phi^{\Xi}(\hbar) \= \langle I_2 \rangle_{(x_0,x),\Lambda_2},
\eea
where
\[
\begin{aligned}
I_2 &\= 
\exp\Big(-\frac{\hbar}{24}+\frac\hbar8 f^t B^{-1} A f
-\frac{\hsqrt}2 x^t(B^{-1}\nu-1)\Big)\,
\psi_\hbar(-x_0 +x_1,z_1z_0^{-1})\\
&\qquad\times\psi_\hbar(x_0,z_0)\,
\psi_\hbar(-x_0+x_2,z_2z_0^{-1})\,
\prod_{j=1}^N \psi_{\hbar}(x_j^{*},z_j^{*})
\end{aligned}
\]
and 
\[
\def\arraystretch{1.5}
\Lambda_2 \= 
\left(\!\!
\begin{array}{ccc|c}
 \frac{({z_1}+{z_2}-z_1z_2)^2}{({z_1}-1) {z_1} ({z_2}-1){z_2}}&
 \frac{{z_1+z_2-z_1z_2}}{z_2({z_1}-1)} &
 \frac{{z_1}+{z_2}-z_1z_2}{{z_1} ({z_2}-1)}&
 0 \\
\frac{{z_1+z_2-z_1z_2}}{z_2({z_1}-1)}&
-{\mbfQ_{11}}{-}\frac{{z_1+z_2-z_1z_2}}{z_2({z_1}-1)} &
 1-{\mbfQ_{12}} &
 -{\mbfQ_{1}^*} \\
 \frac{{z_1}+{z_2}-z_1z_2}{{z_1} ({z_2}-1)}&
 1-{\mbfQ_{12}} &
 -\mbfQ_{22}{-}\frac{{z_1+z_2-z_1z_2}}{z_1({z_2}-1)} &
 -{\mbfQ_{2}^*} \\
 \hline
 0 &
 -{\mbfQ_{1}^*}^t &
 -{\mbfQ_{2}^*}^t &
 {\diag({z^*}')}-{\mbfQ^*} \\
\end{array}
\!\!\right).
\]
Making a change of variables $x_1\mapsto x_1+x_0$ and $x_2\mapsto x_2+x_0$
using Lemma~\ref{lem.ff1}, and denoting $\widetilde{x}=(x_0,x)$,
$\widetilde{z}=(z_0,z_1z_0^{-1},z_2z_0^{-1},z^{*})$ we obtain that
\bea
\Phi^{\Xi}(\hbar) \= e^{c\hbar}\langle I_3 \rangle_{\widetilde{x},\Lambda_3},
\eea
where
\[
I_3 \= 
\exp\Big(\frac\hbar8 \widetilde{f}^t \widetilde{B}^{-1} \widetilde{A} \widetilde{f}
-\frac{\hsqrt}2 \widetilde{x}^t(\widetilde{B}^{-1}\widetilde{\nu}-1)\Big)
\prod_{j=1}^{N+3} \psi_{\hbar}(\widetilde{x}_j,\widetilde{z}_j)
\]
and
\[
\Lambda_3
% &\= 
% \left(
% \def\arraystretch{1.5}
% \begin{array}{ccc|c}
%  1-{Q_{11}}-2 {Q_{12}}-{Q_{22}}+\frac{1}{1-z_0}&
%  1-{Q_{11}}-{Q_{12}} &
%  1-{Q_{12}}-{Q_{22}} &
%  -{Q_{1}^*}-{Q_{2}^*} \\
%  1-{Q_{11}}-{Q_{12}} &
%  \frac{1}{1-z_1z_0^{-1}}-{Q_{11}} &
%  1-{Q_{12}} &
%  -{Q_{1}^*} \\
%  1-{Q_{12}}-{Q_{22}} & 1-{Q_{12}} &
%  \frac{1}{1-z_2z_0^{-1}}-{Q_{22}} &
%  -{Q_{2}^*} \\
%  \hline
%  -{Q_{1}^*}-{Q_{2}^*} & -{Q_{1}^*} & -{Q_{2}^*} & \diag({z^*}')-{Q^*}
% \end{array}
% \right)\\
\=\diag(\tilde z') - \tilde B^{-1}\tilde A\,.
\]
\end{proof}

The next lemma identifies vectors and matrices of the two triangulations
$\calT$ and $\widetilde\calT$.

\begin{lemma}
\label{lem.23.c0}
With the notation used in the proof of the previous Proposition~\ref{prop.case0},
\begin{align}\label{eq.c1bia}
\widetilde B^{-1}\widetilde A &\= \left(\!\!
\def\arraystretch{1.25}
\begin{array}{ccc|c}
 {\mbfQ_{11}}+2 {\mbfQ_{12}}+{\mbfQ_{22}}-1 &
 {\mbfQ_{11}}+{\mbfQ_{12}}-1 & {\mbfQ_{12}}+{\mbfQ_{22}}-1 &
 {\mbfQ_{1}^*}+{\mbfQ_{2}^*} \\
 {\mbfQ_{11}}+{\mbfQ_{12}}-1 & {\mbfQ_{11}} & {\mbfQ_{12}}-1 & {\mbfQ_{1}^*} \\
 {\mbfQ_{12}}+{\mbfQ_{22}}-1 & {\mbfQ_{12}}-1 &
 {\mbfQ_{22}} & {\mbfQ_{2}^*} \\
 \hline
 {\mbfQ_{1}^*}^t+{\mbfQ_{2}^*}^t &
 {\mbfQ_{1}^*}^t & {\mbfQ_{2}^*}^t & {\mbfQ^*} \\
\end{array}
\!\!\right)\\
\widetilde B^{-1}\widetilde{\nu}
&\=
\begin{pmatrix}
(B^{-1}\nu)_1+(B^{-1}\nu)_2-1\\
B^{-1}\nu
\end{pmatrix}\label{eq.c1biv}\\
\widetilde{f}^{t}\widetilde{B}^{-1}\widetilde{\nu}
&\=f^{t}B^{-1}\nu-\widetilde{f}_{0}\label{eq.c1ff}
\end{align}
\end{lemma}

\begin{proof}
% The relation \eqref{eq.BinvA_def} i.e. $A \= BQ$ implies that
% \bea
% \a_1 &\= \beta_1Q_{11} + \beta_2 Q_{12} + b_*{Q_1^*}^t,\\
% \a_2 &\= \beta_1Q_{12} + \beta_2 Q_{22} + b_*{Q_2^*}^t,\\
% a_* &\= \beta_1Q_{1}^* + \beta_2 Q_{2}^* + b_*Q^*.\\
% \eea
We will show that both sides of Equations~\eqref{eq.c1bia} and
Equations~\eqref{eq.c1biv} are equal after multiplying by the invertible
matrix $\widetilde{B}$. For Equation~\eqref{eq.c1bia}, we have
\bea[]
&\left(\!\!\begin{array}{ccc|c}
 -1 & 1 & 1 & 0\\
 0 & b_1 & b_2 & b_{*}
 \end{array}\!\!\right)
\left(\!\!
\def\arraystretch{1.25}
\begin{array}{ccc|c}
 {\mbfQ_{11}}+2 {\mbfQ_{12}}+{\mbfQ_{22}}-1 &
 {\mbfQ_{11}}+{\mbfQ_{12}}-1 & {\mbfQ_{12}}+{\mbfQ_{22}}-1 &
 {\mbfQ_{1}^*}+{\mbfQ_{2}^*} \\
 {\mbfQ_{11}}+{\mbfQ_{12}}-1 & {\mbfQ_{11}} &
 {\mbfQ_{12}}-1 & {\mbfQ_{1}^*} \\
 {\mbfQ_{12}}+{\mbfQ_{22}}-1 & {\mbfQ_{12}}-1 &
 {\mbfQ_{22}} & {\mbfQ_{2}^*} \\
 \hline
 {\mbfQ_{1}^*}^t+{\mbfQ_{2}^*}^t &
 {\mbfQ_{1}^*}^t & {\mbfQ_{2}^*}^t & {\mbfQ^*} \\
\end{array}
\!\!\right)\\
&\=\left(\!\!
\begin{array}{ccc|c}
-1 & 0 & 0 & {0}\\
\substack{b_1(-1+\mbfQ_{11}+\mbfQ_{12})+b_*{\mbfQ_1^*}^t\\
+b_2(-1+\mbfQ_{12}+\mbfQ_{22})+b_*{\mbfQ_2^*}^t}&
\substack{b_1\mbfQ_{11}+b_*{\mbfQ_1^*}^t\\+b_2(-1+\mbfQ_{12})} &
\substack{b_2\mbfQ_{22}+b_*{\mbfQ_2^*}^t\\+b_1(-1+\mbfQ_{12})} &
b_1\mbfQ_1^*+b_2\mbfQ_2^*+b_*\mbfQ_*
\end{array}
\!\!\right)\\
&\= \widetilde A.
\eea
For Equation~\eqref{eq.c1biv}, we have
\be
\begin{pmatrix}
 -1 & 1 & 1 & 0\\
 0 & b_1 & b_2 & b_{*}
\end{pmatrix}
\begin{pmatrix}
(B^{-1}\nu)_1+(B^{-1}\nu)_2-1\\
B^{-1}\nu
\end{pmatrix}
\=
\begin{pmatrix}
1\\
\nu
\end{pmatrix}.
\ee
Finally, for Equation~\eqref{eq.c1ff},
\be
\begin{aligned}
\widetilde{f}^{t}\widetilde{B}^{-1}\nu
&\=
\widetilde{f}^{t}
\begin{pmatrix}
(B^{-1}\nu)_1+(B^{-1}\nu)_2-1\\
B^{-1}\nu
\end{pmatrix}\\
&\=(\widetilde{f}_0+\widetilde{f}_1)(B^{-1}\nu)_1+(\widetilde{f}_0
+\widetilde{f}_2)(B^{-1}\nu)_2+f^{*}(B^{-1}\nu)_1-\widetilde{f}_{0}\,,
\end{aligned}
\ee
and we recall we choose $f_i=\widetilde{f}_0+\widetilde{f}_i$ for $i=1,2$.
\end{proof}

\subsection{\texorpdfstring{The case of $B$ with nullity one}{The case of B with nullity one}}
\label{sub.case1}

In this section, we prove Theorem~\ref{thm.pentagon} under the assumption that
the matrix $B$ has nullity 1. In this case, starting from the series
$\Phi^{\Xi}(\hbar)$, there are three intermediate formulas (shown as bullets)
in Equation~\eqref{2Phi} that eventually identify the result with 
$\Phi^{\widetilde\Xi}(\hbar)$. The intermediate formulas involve a Fourier
transform (adding one integration variable), a pentagon (adding a second
variable), an inverse Fourier transform (adding a third), and an application of
Fubini's theorem that removes two integration variables. The detailed computation
is given in the next proposition. 

\begin{proposition}
\label{prop.case1}
If $B$ has rank $N+1$ then $\Phi^{\Xi}(\hbar)$ is invariant under the 2--3
Pachner move given in Equation~\eqref{2Phi}.
\end{proposition}

\begin{proof}
Following the discussion above we can assume that
$\operatorname{rank}(b_1\,|\,b_2\,|\,b_*) = \operatorname{rank}(b_2\,|\,b_*) = N+1$.
Then by \cite[Lem.A.3]{DG}, the matrix $(a_1\,|\,b_2\,|\,b_*)$ has full rank and
we can apply a quad move to the first columns of $(A\,|\,B)$ to obtain
\bea
\mathbf A \= 
(-b_1 \,|\, a_2 \,|\, a_*)\,,\qquad
\mathbf B \= (a_1-b_1 \,|\, b_2 \,|\, b_*)\,,\qquad
\boldsymbol{\nu}
\=
\nu-b_1\,,
\eea
where $\mathbf B$ has full rank. The proof will use the following sequence
of intermediate matrices
\be
\label{manyAB}
\begin{small}
\begin{tikzpicture}[baseline=(current  bounding  box.center)]
\node[right] at (0,0) {$(\mathbf A\,|\, \mathbf B\,|\,\boldsymbol{\nu}) \,=
\left(\!\!\begin{array}{c|c|c|c|c|c|c}
-b_1 & a_2 & a_* & a_1-b_1 & b_2 & b_* & \nu-b_1
\end{array}\!\!\right)$};

\node[right] at (1,-0.9) {$q_1^{-1}\times 1$};
\draw[->,thick] (1,-0.5)--(1,-1.5);

\node[right] at (0,-2) {$(A\,|\,B\,|\,\nu) \,= \left(\!\!\begin{array}{c|c|c|c|c|c|c}
a_1 & a_2 & a_* & b_1 & b_2 & b_* &\nu
\end{array}\!\!\right)$};

\node[right] at (1,-2.9) {$2\to 3$};
\draw[->,thick] (1,-2.5)--(1,-3.5);

\node[right] at (0,-4.25) {$(\widetilde A\,|\,\widetilde B\,|\,\widetilde{\nu}) \,=
\left(\!\!\begin{array}{c|c|c|c|c|c|c|c|c}
-1 & 0 & 0 & 0 & -1 & 1 & 1 & 0 & 1\\
a_1+a_2-b_1-b_2 & a_1-b_2 & a_2-b_1 & a_*&
0 & b_1 & b_2 & b_* & \nu
\end{array}\!\!\right)$};

\node[right] at (1,-5.4) {$1\times q_1\times 1$};
\draw[->,thick] (1,-5)--(1,-6);

\node[right] at (0,-6.75) {$(\widetilde {\mathbf A}\,|\,\widetilde
{\mathbf B}\,|\,\widetilde{\boldsymbol{\nu}})
\,= 
\left(\!\!\begin{array}{c|c|c|c|c|c|c|c|c}
-1 & -1 & 0 & 0 & -1 & -1 & 1 & 0 & 0\\
a_1+a_2-b_1-b_2 & -b_1 & a_2-b_1 & a_*&
0 & a_1-b_1-b_2 & b_2 & b_* & \nu-b_1
\end{array}\!\!\right)$};

\end{tikzpicture}
\end{small}
\ee
% \bea
% \label{manyAB}
% (\mathbf A\; \mathbf B) &\= 
% \left(\!\!\begin{array}{c|c|c|c|c|c}
% -b_1 & a_2 & a_* & a_1-b_1 & b_2 & b_*
% \end{array}\!\!\right)\\
% &\;\downarrow q_1^{-1}\times 1\\
% (A\;B) &\= \left(\!\!\begin{array}{c|c|c|c|c|c}
% a_1 & a_2 & a_* & b_1 & b_2 & b_*
% \end{array}\!\!\right)\\
% &\;\downarrow 2 \to 3 \\
% (\tilde A\;\tilde B) &\=
% \left(\!\!\begin{array}{c|c|c|c|c|c|c|c}
% 1 & 0 & 0 & 0 & 1 & -1 & -1 & 0\\
% a_1+a_2-b_1-b_2 & a_1-b_2 & a_2-b_1 & a_*&
% 0 & b_1 & b_2 & b_*
% \end{array}\!\!\right)\\
% &\;\downarrow 1 \times q_1 \times 1\\
% (\widetilde {\mathbf A}\;\widetilde {\mathbf B})
% &\= 
% \left(\!\!\begin{array}{c|c|c|c|c|c|c|c}
% 1 & 0 & 0 & 0 & 1 & 1 & -1 & 0\\
% a_1+a_2-b_1-b_2 & -b_1 & a_2-b_1 & a_*&
% 0 & a_1-b_1-b_2 & b_2 & b_*
% \end{array}\!\!\right) \,,
% \eea
where $\mathbf B$ and $\widetilde {\mathbf B}$ are invertible.
With $x=(x_1,x_2,x^*)$ and $z=(z_1,z_2,z^*)$ vectors of size $N+2$,
$\Phi^{\Xi}(\hbar)$ is defined by 
\be
\Phi^{\Xi}(\hbar) \= \langle I_0 \rangle_{(x_1,x_2,x^*), \mathbf \Lambda_0}
\ee
where with $\mathbf f = (f_1',f_2,f^*)$
\begin{align*}
I_0 &\=\exp\Big(\frac\hbar8 \mathbf f^t \mathbf B^{-1} \mathbf A \mathbf f
-\frac{\hsqrt}2 x
(\mathbf B^{-1}\boldsymbol{\nu}-1)\Big)\,
\psi_{\hbar}\Big( x_1,\frac 1 {1-z_1}\Big)\,
\psi_{\hbar}(x_2,z_2)\,
\prod_{j=1}^{N} \psi_{\hbar}( x^*_j, z^*_j)
\end{align*}
and 
\[
\mathbf \Lambda_0\=\diag(z_1'',z_2',{z^*}')-\mathbf Q \,.
\]
% with $\mathbf Q \= \mathbf B^{-1} \mathbf A$ written in block form
% \bea
% \mathbf Q \=
% \def\arraystretch{1.15}
% \left(\!\!
% \begin{array}{cc|c}
% \mathbf Q_{11} & \mathbf Q_{12} & \mathbf Q_1^*\\
% \mathbf Q_{12} & \mathbf Q_{22} & \mathbf Q_2^*\\
% \hline
% \mathbf {Q_1^*}^t & \mathbf {Q_2^*}^t & \mathbf Q^*\\
% \end{array}\!\!\right) \,.
% \eea
We apply Corollary \ref{cor.fourier} to the $\psi_\hbar$ with argument $ x_1$
and introduce a new variable $w_1$ to obtain
\be
\Phi^{\Xi}(\hbar) \= \langle I_1
\rangle_{(w_1, x_1, x_2,x^*),\mathbf \Lambda_1}
\ee
where
\begin{align*}
I_1 &\= \exp\Big(\frac{\hbar}{24}
+\frac\hbar8 \mathbf f^t \mathbf B^{-1} \mathbf A \mathbf f
-\frac{\hsqrt}2 x
(\mathbf B^{-1}\boldsymbol{\nu}-1)
-x_1\frac {\hsqrt}{2}\Big)\\
&\qquad\times\psi_{\hbar}\Big(w_1- x_1(1-z_1^{-1}) ,z_1\Big)\,
\psi_{\hbar}(x_2,z_2)
\prod_{j=1}^{N} \psi_{\hbar}( x_j^*, z_j^*)
\end{align*}
and
\[
\mathbf \Lambda_1 \=
\begin{pmatrix}
\frac {z_1} {1-z_1} & 0\\
0 & \mathbf \Lambda
\end{pmatrix}.
\]
We substitute $w_1 \mapsto w_1+ x_1(1-z_1^{-1})$ and obtain, from
Lemma~\ref{lem.ff1}, that
\be
\Phi^{\Xi}(\hbar) \=
\langle I_2 \rangle_{(w_1, x_1, x_2,x^*),\mathbf \Lambda_2}
\ee
where
\begin{align*}
I_2 &\=\exp\Big(\frac{\hbar}{24}
+\frac\hbar8 \mathbf f^t \mathbf B^{-1} \mathbf A \mathbf f
-\frac{\hsqrt}2 x
(\mathbf B^{-1}\boldsymbol{\nu}-1)
-x_1\frac \hsqrt2\Big)\,
\psi_{\hbar}(w_1,z_1)\,
\psi_{\hbar}(x_2,z_2)
\prod_{j=1}^{N} \psi_{\hbar}( x_j^*, z_j^*)
\end{align*}
and with $\mathbf Q_{11} = 0$ from Lemma \ref{lem.23.c1}
\[
\mathbf \Lambda_2 \= 
\left(\!\!
\def\arraystretch{1.15}
\begin{array}{ccc|c}
 \frac{z_1}{1-{z_1}} & -1 & 0 & 0 \\
 -1 & 0 & -{\mathbf Q_{12}} & -{\mathbf Q_1^*} \\
 0 & -{\mathbf Q_{12}} & \frac{1}{1-{z_2}}-{\mathbf Q_{22}} & -{\mathbf Q_2^*} \\
 \hline
 0 & -{\mathbf Q_1^*}^t & -{\mathbf Q_2^*}^t & {z^*}'-{\mathbf Q^*} \\
\end{array}
\!\!\right).
\]
We apply the pentagon identity of Theorem \ref{thm.pentagon} to the $\psi_h$
with arguments $w_1$ and $x_1$ and obtain, with the new integration variable $x_0$,
that
\bea
\Phi^{\Xi}(\hbar) \=
\langle I_3 \rangle_{(x_0,w_1, x_1, x_2,x^*),
\mathbf \Lambda_3}
\eea
where
\begin{align*}
I_3 &\= \exp\Big(\frac\hbar8 \mathbf f^t \mathbf B^{-1} \mathbf A \mathbf f
-\frac{\hsqrt}2 x
(\mathbf B^{-1}\boldsymbol{\nu}-1)
-x_1 \frac {\hsqrt}{2}\Big)\,
\psi_\hbar\Big(-x_0-x_2+\frac{w_1z_2+x_2z_1}{z_0},z_1z_0^{-1}\Big)\\
&\psi_\hbar\Big(x_0+w_1+x_2-\frac{w_1z_2+x_2z_1}{z_0},z_0\Big)\,
\psi_\hbar\Big(-x_0-w_1+\frac{w_1z_2+x_2z_1}{z_0},z_2z_0^{-1}\Big)
\prod_{j=1}^{N} \psi_{\hbar}( x_j^*, z_j^*)
\end{align*}
and
\[
\mathbf \Lambda_3 \= 
\left(\!\!
\def\arraystretch{1.15}
\begin{array}{cccc|c}
 \frac{({z_1}+{z_2}-z_1z_2)^2}{({z_1}-1) {z_1} ({z_2}-1)
 {z_2}} & 0 & 0 & 0 & 0 \\
 0 & \frac{z_1}{1-{z_1}} & -1 & 0 & 0 \\
 0 & -1 & 0 & -{\mathbf Q_{12}} & -{\mathbf Q_{1}^*} \\
 0 & 0 & -{\mathbf Q_{12}} &
 \frac{1}{1-{z_2}}-{\mathbf Q_{22}} & -{\mathbf Q_{2}^*} \\
 \hline
 0 & 0 & -{\mathbf Q_{1}^*}^t &
 -{\mathbf Q_{2}^*}^t & {z^*}'-{\mathbf Q^*} \\
\end{array}
\!\!\right).
\]
We use Lemma \ref{lem.ff1} to change the variables
$x_0\mapsto x_0 -w_1-x_2+\frac{w_1z_2+x_2z_1}{z_0}$ and obtain
that
\bea
\Phi^{\Xi}(\hbar) \=
\langle I_4 \rangle_{(x_0,w_1,x_1, x_2,x^*),
\mathbf \Lambda_4}
\eea
where
\begin{align*}
I_4 &\= \exp\Big(\frac\hbar8 \mathbf f^t \mathbf B^{-1} \mathbf A \mathbf f
-\frac{\hsqrt}2 x
(\mathbf B^{-1}\boldsymbol{\nu}-1)
-x_1 \frac {\hsqrt}{2}\Big)\\
&\qquad\times\psi_\hbar(-x_0+w_1,z_1z_0^{-1})\,
\psi_\hbar(x_0,z_0)\,
\psi_\hbar(-x_0+x_2,z_2z_0^{-1})
\prod_{j=1}^{N} \psi_{\hbar}( x_j^*, z_j^*)
\end{align*}
and
\[
\mathbf \Lambda_4 \= 
\left(\!\!
\def\arraystretch{1.15}
\begin{array}{cccc|c}
\frac{({z_1}+{z_2}-z_1z_2)^2}{({z_1}-1) {z_1} ({z_2}-1) {z_2}} &
\frac{{z_1}+{z_2}-z_1z_2}{({z_1}-1) {z_2}} & 0 &
\frac{{z_1}+{z_2}-z_1z_2}{{z_1} ({z_2}-1)} & 0 \\
\frac{{z_1}+{z_2}-z_1z_2}{({z_1}-1) {z_2}} &
\frac{{z_1}}{{z_2}(1-{z_1})} & -1 & 1 & 0 \\
 0 & -1 & 0 & -{\mathbf Q_{12}} & -{\mathbf Q_{1}^*} \\
\frac{{z_1}+{z_2}-z_1z_2}{{z_1} ({z_2}-1)}  & 1 & -{\mathbf Q_{12}} &
-\mathbf Q_{22}-\frac{z_1+z_2-z_1z_2}{z_1(z_2-1)} & -{\mathbf Q_{2}^*} \\
\hline
 0 & 0 & -{Q_{1}^*}^t & -{\mathbf Q_{2}^*}^t & {z^*}'-{\mathbf Q_{}^*} \\
\end{array}
\!\!\right).
\]
We substitute $w_1\mapsto w_1+x_0$ and $x_2\mapsto x_2+x_0$ to obtain, with
Lemma \ref{lem.ff1}, that
\bea
\Phi^{\Xi}(\hbar) \=
\langle I_5 \rangle_{(x_0,w_1,x_1, x_2,x^*),
\mathbf \Lambda_5}
\eea
where
\begin{align*}
I_5 &\= \exp\Big(\frac\hbar8 \mathbf f^t \mathbf B^{-1} \mathbf A \mathbf f
-\frac{\hsqrt}2 x
(\mathbf B^{-1}\boldsymbol{\nu}-1)
-\frac{\hsqrt}2 x_0(\mathbf B^{-1}\boldsymbol{\nu}_2-1)
-x_1 \frac {\hsqrt}{2}\Big)\\
&\qquad\times\psi_\hbar(w_1,z_1z_0^{-1})\,
\psi_\hbar(x_0,z_0)\,
\psi_\hbar(x_2,z_2z_0^{-1})
\prod_{j=1}^{N} \psi_{\hbar}( x_j^*, z_j^*)
\end{align*}
and
\[
\mathbf \Lambda_5 \= 
\left(\!\!
\def\arraystretch{1.15}
\begin{array}{cccc|c}
 -{\mathbf Q_{22}}+\frac{1}{({z_1}-1) ({z_2}-1)} &
 0 & -1-{\mathbf Q_{12}} & 1-{\mathbf Q_{22}} & -{\mathbf Q_{2}^*} \\
 0 & \frac{{z_1}}{{z_2}(1-{z_1})}& -1 & 1 & 0 \\
 -1-{\mathbf Q_{12}} & -1 & 0 & -{\mathbf Q_{12}} & -{\mathbf Q_{1}^*} \\
 1-{\mathbf Q_{22}} & 1 & -{\mathbf Q_{12}} &
-\mathbf Q_{22}+\frac{z_1+z_2-z_1z_2}{z_1(z_2-1)}& -{\mathbf Q_{2}^*} \\
\hline
-{\mathbf Q_{2}^*}^t & 0 & -{\mathbf Q_{1}^*}^t & -{\mathbf Q_{2}^*}^t &
{z^*}'-{\mathbf Q^*} \\
\end{array}
\!\!\right).
\]
By applying the first quad move from Theorem \ref{thm.fourier} to
$\psi_\hbar(w_1,z_1z_0^{-1})$ we obtain, with a new integration variable $y_1$, that
\bea
\Phi^{\Xi}(\hbar) \=
e^{-\frac \hbar {24}}\langle I_6 \rangle_{(y_1,x_0,w_1, x_1, x_2,x^*),
\mathbf \Lambda_6}
\eea
where
\begin{align*}
I_6 &\=\exp\Big(-\frac{\hbar}{24}
+\frac\hbar8 \mathbf f^t \mathbf B^{-1} \mathbf A \mathbf f
-\frac{\hsqrt}2 x
(\mathbf B^{-1}\boldsymbol{\nu}-1)
\\ & \quad
-\frac{\hsqrt}2 x_0((\mathbf B^{-1}\boldsymbol{\nu})_2-1)
-x_1 \frac {\hsqrt}{2}
+\Big(y_1+\frac{w_1z_1z_0^{-1}}{1-z_1z_0^{-1}}\Big)\frac{\hsqrt}{2}\Big)\\
&\qquad\times
\psi_\hbar\Big(y_1+\frac{w_1z_1z_0^{-1}}{1-z_1z_0^{-1}},\frac1{1-z_1z_0^{-1}}\Big)\,
\psi_\hbar(x_0,z_0)\,
\psi_\hbar(x_2,z_2z_0^{-1})
\prod_{j=1}^{N} \psi_{\hbar}( x_j^*, z_j^*)
\end{align*}
and
\begin{small}
\[
\mathbf \Lambda_6 \=
\left(\!\!
\def\arraystretch{1.15}
\begin{array}{ccccc|c}
\frac{(z_1-1)z_2}{z_1} & 0& 0& 0& 0& 0\\
 0&-{\mathbf Q_{22}}+\frac{1}{({z_1}-1) ({z_2}-1)} &
 0 & -1-{\mathbf Q_{12}} & 1-{\mathbf Q_{22}} & -{\mathbf Q_{2}^*} \\
 0&0 & \frac{{z_1}}{{z_2}(1-{z_1})}& -1 & 1 & 0 \\
 0&-1-{\mathbf Q_{12}} & -1 & 0 & -{\mathbf Q_{12}} & -{\mathbf Q_{1}^*} \\
 0&1-{\mathbf Q_{22}} & 1 & -{\mathbf Q_{12}} &
-\mathbf Q_{22}+\frac{z_1+z_2-z_1z_2}{z_1(z_2-1)}& -{\mathbf Q_{2}^*} \\
\hline
0&-{\mathbf Q_{2}^*}^t & 0 & -{\mathbf Q_{1}^*}^t &
-{\mathbf Q_{2}^*}^t & {z^*}'-{\mathbf Q^*} 
\end{array}
\!\!\right).
\]
\end{small}
We change variables $y_1\mapsto y_1-w_1\frac{z_1z_0^{-1}}{1-z_1z_0^{-1}}$ using
Lemma \ref{lem.ff1} and obtain that
\bea
\Phi^{\Xi}(\hbar) \=
\langle I_7 \rangle_{(y_1,x_0,w_1, x_1, x_2,x^*),
\mathbf \Lambda_7}
\eea
where
\begin{align*}
I_7 &\=\exp\Big(-\frac{\hbar}{24}
+\frac\hbar8 \mathbf f^t \mathbf B^{-1} \mathbf A \mathbf f
-\frac{\hsqrt}2 x
(\mathbf B^{-1}\boldsymbol{\nu}-1)
-\frac{\hsqrt}2 x_0((\mathbf B^{-1}\boldsymbol{\nu})_2-1)
-x_1 \frac {\hsqrt}{2}
+y_1\frac{\hsqrt}{2}\Big)\\
& \qquad\times
\psi_\hbar\Big(y_1,\frac1{1-z_1z_0^{-1}}\Big)\,
\psi_\hbar(x_0,z_0)\,
\psi_\hbar(x_2,z_2z_0^{-1})
\prod_{j=1}^{N} \psi_{\hbar}( x_j^*, z_j^*)
\end{align*}
and
\[
\mathbf \Lambda_7 \= 
\left(\!\!
\def\arraystretch{1.15}
\begin{array}{ccccc|c}
\frac{(z_1-1)z_2}{z_1} & 0& 1& 0& 0& 0\\
 0&-{\mathbf Q_{22}}+\frac{1}{({z_1}-1) ({z_2}-1)} &
 0 & -1-{\mathbf Q_{12}} & 1-{\mathbf Q_{22}} & -{\mathbf Q_{2}^*} \\
 1&0 & 0& -1 & 1 & 0 \\
 0&-1-{\mathbf Q_{12}} & -1 & 0 & -{\mathbf Q_{12}} & -{\mathbf Q_{1}^*} \\
 0&1-{\mathbf Q_{22}} & 1 & -{\mathbf Q_{12}} &
-\mathbf Q_{22}+\frac{z_1+z_2-z_1z_2}{z_1(z_2-1)}& -{\mathbf Q_{2}^*} \\
\hline
0&-{\mathbf Q_{2}^*}^t & 0 & -{\mathbf Q_{1}^*}^t &
-{\mathbf Q_{2}^*}^t & {z^*}'-{\mathbf Q^*} \\
\end{array}
\!\!\right).
\]
Therefore, we can apply Fubini's Theorem (Lemma~\ref{lem.ff1}) with the
integration variables $w_1,x_1$,
to obtain that
% We integrate with respect to the variables $(w_1, x_1)$ using
% Lemma \ref{lem.ff1}
% with
% \bea
% A &\= \begin{pmatrix}
% 	0 & -1 \\ -1 & 0
% \end{pmatrix},\qquad
% B \= \begin{pmatrix}
% 1 & 0 & 1 & 0 \\
% 0 & -1-\mathbf Q_{12} & -\mathbf Q_{12} & -\mathbf Q_1^*
% \end{pmatrix},\\
% C &\= 
% \left(\!\!\begin{array}{ccc|c}
% \frac{(z_1-1)z_2}{z_1} & 0& 0& 0\\
%  0&-{\mathbf Q_{22}}+\frac{1}{({z_1}-1) ({z_2}-1)} &
%  1-{\mathbf Q_{22}} & -{\mathbf Q_{2}^*} \\
% %asdf
%  0&1-{\mathbf Q_{22}} &
% -\mathbf Q_{22}+\frac{z_1+z_2-z_1z_2}{z_1(z_2-1)}& -{\mathbf Q_{2}^*} \\
% \hline
% 0&-{\mathbf Q_{2}^*}^t &
% -{\mathbf Q_{2}^*}^t & {z^*}'-{\mathbf Q^*} \\
% \end{array}\!\!\right),\\
% F &\= \begin{pmatrix}
% 0 & -\frac 1 2(\mathbf B^{-1}\boldsymbol{\nu})	
% \end{pmatrix}
% \eea
% with respect to the integration variables $(w_1,x_1)$,
% resp. $(y_1,x_0,x_2,x^*)$ and obtain
\bea
\Phi^{\Xi}(\hbar) \=
\langle I_8 \rangle_{(y_1,x_0, x_2,x^*),
\mathbf \Lambda_8}
\eea
where 
\begin{align*}
I_8 &\= \exp\Big(-\frac{\hbar}{24}
+\frac\hbar8 \mathbf f^t \mathbf B^{-1} \mathbf A \mathbf f
-\frac{\hsqrt}2(
x_0 ((\mathbf B^{-1}\boldsymbol{\nu})_2-1)
+y_1((\mathbf B^{-1}\boldsymbol{\nu})_1-1)
\\ & \qquad\qquad +x_2 ((\mathbf B^{-1}\boldsymbol{\nu})_2
+(\mathbf B^{-1}\boldsymbol{\nu})_1-1)
+x^*((\mathbf B^{-1}\boldsymbol{\nu})^*-1)
)\Big)\\
& \qquad 
\times
\psi_\hbar\Big(y_1,\frac1{1-z_1z_0^{-1}}\Big)\,
\psi_\hbar(x_0,z_0)\,
\psi_\hbar(x_2,z_2z_0^{-1})
\prod_{j=1}^{N} \psi_{\hbar}( x_j^*, z_j^*)
\end{align*}
and
\[
\mathbf \Lambda_8 \=
\left(\!\!
\begin{array}{ccc|c}
\frac{({z_1}-1) {z_2}}{{z_1}} & -{\mathbf Q_{12}}-1 &
-{\mathbf Q_{12}} & -{Q_{1}^*} \\
-{\mathbf Q_{12}}-1 & \frac{1}{({z_1}-1) ({z_2}-1)}-{\mathbf Q_{22}} &
-{\mathbf Q_{12}}-{\mathbf Q_{22}} & -{\mathbf Q_{2}^*}
\\
-{\mathbf Q_{12}} & -{\mathbf Q_{12}}-{\mathbf Q_{22}} & 
-2 {\mathbf Q_{12}}-\mathbf Q_{22} + \frac{z_1+z_2-z_1z_2}{z_1(1-z_2)}
& -{\mathbf Q_{1}^*}-{\mathbf Q_{2}^*} \\
\hline
-{\mathbf Q_{1}^*}^t & -{\mathbf Q_{2}^*}^t &
-{\mathbf Q_{1}^*}^t-{\mathbf Q_{2}^*}^t & {z^*}'-{\mathbf Q^*} \\
\end{array}
\!\!\right) \,.
\]
Using Lemma \ref{lem.23.c1} we obtain, with respect to the integration
variables $\tilde x = (x_0,y_1,x_2,x^*)$ and some $c\in\frac{1}{24}\BZ$, that
\bea
\Phi^{\Xi}(\hbar) \=e^{c\hbar}
\langle I_9 \rangle_{\tilde x, \mathbf \Lambda_9}
\eea
where
\[
I_9 \= \exp\Big(\frac\hbar8 \widetilde{\mathbf f}^t \widetilde{\mathbf B}^{-1}
\widetilde{\mathbf A} \widetilde{\mathbf f}
-\frac{\hsqrt}2 \tilde x^t(\widetilde {\mathbf B}^{-1}
\widetilde{\boldsymbol{\nu}}-1)\Big)\,
\prod_{j=0}^{N+2} \psi_{\hbar}( \widetilde{x}_j^*, \widetilde{z}_j^*)
\]
and with $\widetilde {\mathbf A}$, $ \widetilde {\mathbf B}$ as in 
\eqref{manyAB}
\[
\mathbf \Lambda_9 \=
\diag(\widetilde{z}_0',
\widetilde{z}_1'',
\widetilde{z}_2',
\widetilde{z}^*{}')
- {\widetilde {\mathbf B}}^{-1}\widetilde {\mathbf A}.
\]
% In particular $(\widetilde {\mathbf A}\,|\, \widetilde {\mathbf B})$ can be
% obtained from $(\widetilde A\,|\, \widetilde B)$ by applying a quad move
% to the second column such that
% \bea
% \Phi^{\Xi}(\hbar) \=
% \Phi^{\tilde \Xi}(\hbar)
% \eea
% which completes the proof of Proposition~\ref{prop.case1}.
\end{proof}

\begin{lemma}
\label{lem.23.c1}
With the notation used in the previous proof of Proposition~\ref{prop.case1}
we have\\ $\mbfQ_{11}=0$ and the following equalities:
\begin{align}
\label{eq.c2.bia}
\widetilde {\mathbf B}^{-1} \widetilde {\mathbf A} &\=
\left(\!\!\begin{array}{ccc|c}
\mathbf Q_{22} & \mathbf Q_{12} + 1&
\mathbf Q_{12} + \mathbf Q_{22}& \mathbf Q_2^*\\
\mathbf Q_{12} + 1& 0&
\mathbf Q_{12}& \mathbf Q_1^*\\
\mathbf Q_{12} +\mathbf Q_{22}& \mathbf Q_{12}&
2\mathbf Q_{12} +\mathbf  Q_{22}& \mathbf Q_1^* + \mathbf Q_2^*\\
\hline
{\mathbf Q_2^*}^t& {\mathbf Q_1^*}^t&
{\mathbf Q_1^*}^t + {\mathbf Q_2^*}^t& \mathbf Q^*
\end{array}
\!\!\right),
\\\label{eq.c2.binu}
\widetilde {\mathbf B}^{-1}\widetilde{\boldsymbol{\nu}}
&\= \left(\!\!
\begin{array}{c}
(\mathbf B^{-1}\boldsymbol{\nu})_2\\
(\mathbf B^{-1}\boldsymbol{\nu})_1\\
(\mathbf B^{-1}\boldsymbol{\nu})_1 + (\mathbf B^{-1}\boldsymbol{\nu})_2\\
(\mathbf B^{-1}\boldsymbol{\nu})^*
\end{array}\!\!\right),
\\\label{eq.c3.ff}
\widetilde{\mathbf{f}}^{t}\widetilde {\mathbf B}^{-1}\widetilde{\boldsymbol{\nu}}
&\= \mathbf{f}^{t}\mathbf{B}^{-1}\boldsymbol{\nu}\,.
\end{align}
\end{lemma}

\begin{proof}
Similarly to the proof of Lemma \ref{lem.23.c0} we will proof Equations
\eqref{eq.c2.bia} and \eqref{eq.c2.binu} by showing that both sides are equal
after multiplying by the invertible matrix $\widetilde {\mathbf B}$.
Using $\mathbf Q=\mathbf B^{-1} \mathbf A$ and the fact that
$(a_1\,|\,b_2\,|\,b_*)$ are linearly independent we conclude $\mbfQ_{11}=0$.
For \eqref{eq.c2.bia} we compute

\bea[]
&\begin{pmatrix}
-1 & -1 & 1 & 0 \\
0 & a_1-b_1-b_2 & b_2 & b_*
\end{pmatrix}
\left(\!\!\begin{array}{ccc|c}
\mathbf Q_{22} &\mathbf  Q_{12} + 1&
\mathbf Q_{12} + \mathbf Q_{22}& \mathbf Q_2^*\\
\mathbf Q_{12} + 1& 0& \mathbf Q_{12}& \mathbf Q_1^*\\
\mathbf Q_{12} + \mathbf Q_{22}& \mathbf Q_{12}&
2\mathbf Q_{12} + \mathbf Q_{22}& \mathbf Q_1^* + \mathbf Q_2^*\\
\hline
{\mathbf Q_2^*}^t& {\mathbf Q_1^*}^t&
{\mathbf Q_1^*}^t + {\mathbf Q_2^*}^t& \mathbf Q^*
\end{array}\!\!\right)\\
\=&\begin{pmatrix}
-1 & -1 & 0 & 0\\
\substack{(a_1-b_1)\mathbf Q_{12}\\
 +b_2\mathbf Q_{22}+b_*{\mathbf Q_2^*}^t}&
\substack{b_2\mathbf Q_{12}+b_*{\mathbf Q_1^*}^t}&
\substack{b_2\mathbf Q_{12}+b_*{\mathbf Q_1^*}^t+\\
(a_1-b_1)\mathbf Q_{12}+b_2\mathbf Q_{22}+b_*{\mathbf Q_1^*}^t}&
\substack{(a_1-b_1-b_2)\mathbf Q_1^*\\
+b_2(\mathbf Q_1^*+\mathbf Q_2^*)+b_*\mathbf Q^*}
\end{pmatrix}\\
%\=&\begin{pmatrix}
%1 & 1 & 0 & 0\\
%a_1+a_2-b_1-b_2 & -b_1 & a_2-b_1 & a_*
%\end{pmatrix}
\= &\widetilde {\mathbf A}
\eea
For \eqref{eq.c2.binu} we have
\bea[]
&\begin{pmatrix}
-1 & -1 & 1 & 0 \\
0 & a_1-b_1-b_2 & b_2 & b_*
\end{pmatrix}
\left(\!\!
\begin{array}{c}
(\mathbf B^{-1}\boldsymbol{\nu})_2\\
(\mathbf B^{-1}\boldsymbol{\nu})_1\\
(\mathbf B^{-1}\boldsymbol{\nu})_1 + (\mathbf B^{-1}\boldsymbol{\nu})_2\\
(\mathbf B^{-1}\boldsymbol{\nu})^*
\end{array}\!\!\right)\\
&\= 
\left(\!\!\begin{array}{c}
0\\ \boldsymbol{\nu}
\end{array}\!\!\right)
\eea
For Equation~\eqref{eq.c3.ff}, we have
\be
\begin{aligned}
&\widetilde{\mathbf{f}}^{t}
\left(\!\!
\begin{array}{c}
(\mathbf B^{-1}\boldsymbol{\nu})_2\\
(\mathbf B^{-1}\boldsymbol{\nu})_1\\
(\mathbf B^{-1}\boldsymbol{\nu})_1 + (\mathbf B^{-1}\boldsymbol{\nu})_2\\
(\mathbf B^{-1}\boldsymbol{\nu})^*
\end{array}\!\!\right)\\
% &\=
% \widetilde{f}_0(\mathbf B^{-1}(\nu-b_1))_2
% +
% \widetilde{f}_1'(\mathbf B^{-1}(\nu-b_1))_1\\
% &\qquad+
% \widetilde{f}_2((\mathbf B^{-1}(\nu-b_1))_1 + (\mathbf B^{-1}(\nu-b_1))_2)
% +
% \widetilde{f}^{*}(\mathbf B^{-1}(\nu-b_1))^*\\
&\=
(\widetilde{f}_1'+\widetilde{f}_2)(\mathbf B^{-1}\boldsymbol{\nu})_1
+
(\widetilde{f}_0-\widetilde{f}_2)(\mathbf B^{-1}\boldsymbol{\nu})_2
+
\widetilde{f}^{*}(\mathbf B^{-1}\boldsymbol{\nu})^*\,.
\end{aligned}
\ee
Then, using the analogous relations between the flattening as given in
Equation~\eqref{eq.23sr} for the shapes, we have
$\widetilde{f}_1'+\widetilde{f}_2=f_1'=\mathbf{f}_1$ and
$\widetilde{f}_0-\widetilde{f}_2=f_2=\mathbf{f}_2$.
\end{proof}

\subsection{\texorpdfstring{The case of $B$ with nullity two}{The case of B with nullity two}}
\label{sub.case2}

In this section, we prove Theorem~\ref{thm.pentagon} under the assumption
that the matrix $B$ has nullity 2. Similarly to the proof of
Proposition~\ref{prop.case1}, we use three intermediate formulas in
Equation~\eqref{2Phi}. They involve two Fourier transforms (adding two
variables), a pentagon (adding a third variable), two inverse Fourier
transforms (adding two variables) and an applications of Fubini's Theorem
(removing four integration variables). The details are given in the following
proposition.

\begin{proposition}
\label{prop.case2}
If $B$ has rank $N$ then $\Phi^{\Xi}(\hbar)$ is invariant under the 2--3
Pachner move given in Equation~\eqref{2Phi}.
\end{proposition}

\begin{proof}

Following the discussion above, we can assume that
$\operatorname{rank}(b_1\,|\,b_2\,|\,b_*) = \operatorname{rank}(\,b_*) = N$.
Then by \cite[Lem.A.3]{DG}, the matrix $(a_1-b_1\,|\,a_2-b_2\,|\,b_*)$ has full rank
and we can apply a quad move to the first columns of $(A\,|\,B)$ to obtain
\bea
\mathbf A \= 
(-b_1 \,|\, -b_2 \,|\, a_*),\qquad
\mathbf B \= (a_1-b_1 \,|\, a_2-b_2 \,|\, b_*)
\eea
where $\mathbf B$ has full rank. The proof will use the following sequence
of intermediate matrices
\be
\label{manyAB.c2}
\begin{small}
\begin{tikzpicture}[baseline=(current  bounding  box.center)]
\node[right] at (0,0) {$(\mathbf A\,|\, \mathbf B\,|\,\boldsymbol{\nu}) \,=
\left(\!\!\begin{array}{c|c|c|c|c|c|c}
-b_1 & -b_2 & a_* & a_1-b_1 & a_2-b_2 & b_* & \nu-b_1-b_2
\end{array}\!\!\right)$};

\node[right] at (1,-0.9) {$q_2^{-1}\times 1$};
\draw[->,thick] (1,-0.5)--(1,-1.5);

\node[right] at (0,-2) {$(A\,|\,B\,|\,\nu) \,= \left(\!\!\begin{array}{c|c|c|c|c|c|c}
a_1 & a_2 & a_* & b_1 & b_2 & b_* &\nu
\end{array}\!\!\right)$};

\node[right] at (1,-2.9) {$2\to 3$};
\draw[->,thick] (1,-2.5)--(1,-3.5);

\node[right] at (0,-4.25) {$(\widetilde A\,|\,\widetilde B\,|\,\widetilde{\nu}) \,=
\left(\!\!\begin{array}{c|c|c|c|c|c|c|c|c}
-1 & 0 & 0 & 0 & -1 & 1 & 1 & 0 & 1\\
a_1+a_2-b_1-b_2 & a_1-b_2 & a_2-b_1 & a_*&
0 & b_1 & b_2 & b_* & \nu
\end{array}\!\!\right)$};

\node[right] at (1,-5.4) {$1\times q_2\times 1$};
\draw[->,thick] (1,-5)--(1,-6);

\node[right] at (0,-6.75) {$(\widetilde {\mathbf A}\,|\,\widetilde
{\mathbf B}\,|\,\widetilde{\boldsymbol{\nu}})=$};

\node[right] at (0.25,-7.75)
{$\left(\!\!\begin{array}{c|c|c|c|c|c|c|c|c}
-1 & -1 & -1 & 0 & -1 & -1 & -1 & 0 & -1\\
a_1+a_2-b_1-b_2 & -b_1 & -b_2 & a_*&
0 & a_1-b_1-b_2 & a_2-b_1-b_2 & b_* & \nu - b_1-b_2
\end{array}\!\!\right)$};
\end{tikzpicture}
\end{small}
\ee

\noindent
where $\mathbf B$ and $\widetilde {\mathbf B}$ are invertible.
With $x=(x_1,x_2,x^*)$ and $z=(z_1,z_2,z^*)$ vectors of size $N+2$,
$\Phi^\Xi(\hbar)$ is defined by
\bea
\Phi^{\Xi}(\hbar) \= \langle I_0 \rangle_{(x_1,x_2,x^*), \mathbf \Lambda_0}
\eea
where with $\mathbf f = (f_1',f_2',f^*)$ 
\begin{align*}
I_0 &\= \exp\Big(
\frac\hbar8 \mathbf f^t \mathbf B^{-1} \mathbf A \mathbf f
-\frac{\hsqrt}2 x^t
(\mathbf B^{-1}\boldsymbol{\nu}-1)\Big)\\
&\qquad\times\psi_{\hbar}\Big( x_1,\frac 1 {1-z_1}\Big)\,
\psi_{\hbar}\Big(x_2,\frac 1 {1-z_2}\Big)\,
\prod_{j=1}^{N} \psi_{\hbar}( x^*_j, z^*_j)
\end{align*}
and
\[
\mathbf \Lambda_0 \= \diag(1-z_1^{-1}, 1-z_2^{-1}, {z^*}') - \mbfQ\,.
\]
% with $\mbfQ = \mathbf B^{-1} \mathbf A$ written in block form
% \bea
% \mathbf Q \=
% \def\arraystretch{1.15}
% \left(\!\!
% \begin{array}{cc|c}
% \mathbf Q_{11} & \mathbf Q_{12} & \mathbf Q_1^*\\
% \mathbf Q_{12} & \mathbf Q_{22} & \mathbf Q_2^*\\
% \hline
% \mathbf {Q_1^*}^t & \mathbf {Q_2^*}^t & \mathbf Q^*\\
% \end{array}\!\!\right) \,.
% \eea
By applying Corollary \ref{cor.fourier} to both $\psi_\hbar(x_1,\frac1{1-z_1})$
and $\psi_\hbar(x_1,\frac1{1-z_1})$ we obtain with new integration
variables $w_1$ and $w_2$ that
\bea
\Phi^{\Xi}(\hbar) \= \langle I_1
\rangle_{(w_1,w_2, x_1, x_2,x^*),\mathbf \Lambda_1}
\eea
where
\begin{align*}
I_1 &\= \exp\Big(\frac{\hbar}{12}
+\frac\hbar8 \mathbf f^t \mathbf B^{-1} \mathbf A \mathbf f
-\frac{\hsqrt}2 x^t
(\mathbf B^{-1}\boldsymbol{\nu}-1)
-(x_1+x_2)\frac {\hsqrt}{2}\Big)\\
&\qquad\times\psi_{\hbar}\Big(w_1- x_1(1-z_1^{-1}) ,z_1\Big)\,
\psi_{\hbar}\Big(w_2- x_2(1-z_2^{-1}) ,z_1\Big)\,
\prod_{j=1}^{N} \psi_{\hbar}( x_j^*, z_j^*)
\end{align*}
and
\[
\mathbf \Lambda_1 \= 
\begin{pmatrix}
\frac{z_2}{1-z_2} & 0 & 0\\
0 & \frac{z_1}{1-z_1} & 0\\
0 & 0 & \mathbf\Lambda
\end{pmatrix}
\]
The substitutions $w_1 \mapsto w_1+ x_1(1-z_1^{-1})$ and
$w_2 \mapsto w_2+ x_2(1-z_2^{-1})$ imply according to Lemma \ref{lem.ff1} that
\be
\Phi^{\Xi}(\hbar) \=
\langle I_2 \rangle_{(w_1,w_2,x_1, x_2,x^*),\mathbf \Lambda_2}
\ee
where
\begin{align*}
I_2 &\=\exp\Big(\frac{\hbar}{12}
+\frac\hbar8 \mathbf f^t \mathbf B^{-1} \mathbf A \mathbf f
-\frac{\hsqrt}2 x^t
(\mathbf B^{-1}\boldsymbol{\nu}-1)
-(x_1+x_2)\frac \hsqrt2\Big)\\
&\qquad\times\psi_{\hbar}(w_1,z_1)\,
\psi_{\hbar}(w_2,z_2)\,
\prod_{j=1}^{N} \psi_{\hbar}( x_j^*, z_j^*)
\end{align*}
and with $\mathbf Q_{11} = \mbfQ_{12} = \mbfQ_{22} = 0$ from Lemma \ref{lem.23.c2}
\[
\mathbf \Lambda_2 \= 
\left(\!\!
\begin{array}{cccc|c}
 \frac{z_1}{1-{z_1}} & 0 & -1 & 0 & 0 \\
 0 & \frac{z_2}{1-{z_2}} & 0 & -1 & 0\\
 -1 & 0 & 0 &0 & -{\mathbf Q_1^*} \\
 0 & -1 & 0 & 0 & -{\mathbf Q_2^*} \\
 \hline
 0 & 0 &-{\mathbf Q_1^*}^t & -{\mathbf Q_2^*}^t & {z^*}'-{\mathbf Q^*} \\
\end{array}
\!\!\right).
\]
We apply the pentagon identity of Theorem \ref{thm.pentagon} to the $\psi_h$
with arguments $w_1$ and $w_1$ and obtain with the new integration variable $x_0$
\bea
\Phi^{\Xi}(\hbar) \=
\langle I_3 \rangle_{(x_0,w_1,w_2, x_1, x_2,x^*),
\mathbf \Lambda_3}
\eea
where
\begin{align*}
I_3 &\=\exp\Big(\frac{\hbar}{24}
+\frac\hbar8 \mathbf f^t \mathbf B^{-1} \mathbf A \mathbf f
-\frac{\hsqrt}2 x^t
(\mathbf B^{-1}\boldsymbol{\nu}-1)
-(x_1+x_2) \frac {\hsqrt}{2}\Big)\\
&\qquad\times
\psi_\hbar\Big(-x_0-w_2+\frac{w_1z_2+w_2z_1}{z_0},z_1z_0^{-1}\Big)\,
\psi_\hbar\Big(x_0+w_1+w_2-\frac{w_1z_2+w_2z_1}{z_0},z_0\Big)\\
&\qquad\times
\psi_\hbar\Big(-x_0-w_1+\frac{w_1z_2+w_2z_1}{z_0},z_2z_0^{-1}\Big)\,
\prod_{j=1}^{N} \psi_{\hbar}( x_j^*, z_j^*)
\end{align*}
and
\[
\mathbf \Lambda_3 \= 
\left(\!\!
\begin{array}{ccccc|c}
 \frac{({z_1}+{z_2}-z_1z_2)^2}{({z_1}-1) {z_1} ({z_2}-1){z_2}}&
 0 & 0 & 0 & 0 & 0 \\
 0 & \frac{z_1}{1-{z_1}} & 0 & -1 & 0 & 0 \\
 0 & 0 & \frac{z_2}{1-{z_2}} & 0 & -1 & 0\\
 0 & -1 & 0 & 0 &0 & -{\mathbf Q_1^*} \\
 0 & 0 & -1 & 0 & 0 & -{\mathbf Q_2^*} \\
 \hline
 0 & 0 & 0 &-{\mathbf Q_1^*}^t & -{\mathbf Q_2^*}^t & {z^*}'-{\mathbf Q^*} \\
\end{array}
\!\!\right).
\]
With change of variables $x_0 \mapsto x_0-w_1-w_2+\frac{w_1z_2+w_2z_1}{z_0}$
Lemma \ref{lem.ff1} gives that
\bea
\Phi^{\Xi}(\hbar) \=
\langle I_4 \rangle_{(x_0,w_1,w_2, x_1, x_2,x^*),
\mathbf \Lambda_4}
\eea
where
\begin{align*}
I_4 &\= \exp\Big(\frac{\hbar}{24}
+\frac\hbar8 \mathbf f^t \mathbf B^{-1} \mathbf A \mathbf f
-\frac{\hsqrt}2 x^t
(\mathbf B^{-1}\boldsymbol{\nu}-1)
-(x_1+x_2) \frac {\hsqrt}{2}\Big)\\
&\qquad\times\psi_\hbar(-x_0+w_1,z_1z_0^{-1})\,
\psi_\hbar(x_0,z_0)\,
\psi_\hbar(-x_0+w_2,z_2z_0^{-1})\,
\prod_{j=1}^{N} \psi_{\hbar}( x_j^*, z_j^*)
\end{align*}
and
\[
\mathbf \Lambda_4 \= 
\left(\!\!
\begin{array}{ccccc|c}
 \frac{({z_1}+{z_2}-z_1z_2)^2}{({z_1}-1) {z_1} ({z_2}-1){z_2}}&
 \frac{{z_1}+{z_2}-z_1z_2}{({z_1}-1) {z_2}} &
 \frac{{z_1}+{z_2}-z_1z_2}{({z_2}-1) {z_1}} & 0 & 0 & 0 \\
 \frac{{z_1}+{z_2}-z_1z_2}{({z_1}-1) {z_2}} &
 \frac{z_1}{1-{z_1}} & 1 & -1 & 0 & 0 \\
 \frac{{z_1}+{z_2}-z_1z_2}{({z_2}-1) {z_1}} &
 1 & \frac{z_2}{1-{z_2}} & 0 & -1 & 0\\
 0 & -1 & 0 & 0 &0 & -{\mathbf Q_1^*} \\
 0 & 0 & -1 & 0 & 0 & -{\mathbf Q_2^*} \\
 \hline
 0 & 0 & 0 &-{\mathbf Q_1^*}^t & -{\mathbf Q_2^*}^t & {z^*}'-{\mathbf Q^*} \\
\end{array}
\!\!\right).
\]
We substitute $w_1\mapsto w_1+x_0$ and $w_2\mapsto w_2+x_0$ to obtain with
Lemma \ref{lem.ff1} that
\bea
\Phi^{\Xi}(\hbar) \=
\langle I_5 \rangle_{(x_0,w_1,w_2,x_1, x_2,x^*),
\mathbf \Lambda_5}
\eea
where
\begin{align*}
I_5 &\= \exp\Big(\frac{\hbar}{24}
+\frac\hbar8 \mathbf f^t \mathbf B^{-1} \mathbf A \mathbf f
-\frac{\hsqrt}2 x^t
(\mathbf B^{-1}\boldsymbol{\nu}-1)
-(x_1+x_2) \frac {\hsqrt}{2}\Big)\\
&\qquad\times\psi_\hbar(w_1,z_1z_0^{-1})\,
\psi_\hbar(x_0,z_0)\,
\psi_\hbar(w_2,z_2z_0^{-1})\,
\prod_{j=1}^{N} \psi_{\hbar}( x_j^*, z_j^*)
\end{align*}
and
\[
\mathbf \Lambda_5 \=
\left(\!\!
\begin{array}{ccccc|c}
 \frac{{z_1}+{z_2}-z_1z_2}{({z_1}-1) ({z_2}-1)}&
 0 & 0& -1 & -1 & 0 \\
 0& \frac{z_1}{(1-{z_1})z_2} & 1 & -1 & 0 & 0 \\
 0& 1 & \frac{z_2}{(1-{z_2})z_1} & 0 & -1 & 0 \\
 -1 & -1 & 0 & 0 &0 & -{\mathbf Q_1^*} \\
 -1 & 0 & -1 & 0 & 0 & -{\mathbf Q_2^*} \\
 \hline
 0 & 0 & 0 &-{\mathbf Q_1^*}^t & -{\mathbf Q_2^*}^t & {z^*}'-{\mathbf Q^*} \\
\end{array}
\!\!\right).
\]
By applying the first quad move from Theorem \ref{thm.fourier} to both
$\psi_\hbar(w_1,z_1z_0^{-1})$ and $\psi_\hbar(w_2,z_2z_0^{-1})$
 we obtain with new integration variables $y_1$ and $y_2$ that
\bea
\Phi^{\Xi}(\hbar) \=
\langle I_6 \rangle_{(y_1,y_2,x_0,w_1,w_2, x_1, x_2,x^*),
\mathbf \Lambda_6}
\eea
where
\begin{align*}
I_6 &\= \exp\Big(-\frac{\hbar}{24}
+\frac\hbar8 \mathbf f^t \mathbf B^{-1} \mathbf A \mathbf f
-\frac{\hsqrt}2 x^t
(\mathbf B^{-1}\boldsymbol{\nu}-1)
-(x_1+x_2) \frac {\hsqrt}{2}\\
&\qquad\qquad+
\Big(y_1+\frac{w_1z_1z_0^{-1}}{1-z_1z_0^{-1}}\Big)\frac{\hsqrt}{2}
+\Big(y_2+\frac{w_2z_2z_0^{-1}}{1-z_2z_0^{-1}}\Big)\frac{\hsqrt}{2}\Big)
\prod_{j=1}^{N} \psi_{\hbar}( x_j^*, z_j^*)\\
&\qquad\times
\psi_\hbar\Big(y_1+\frac{w_1z_1z_0^{-1}}{1-z_1z_0^{-1}},
\frac 1 {1-z_1z_0^{-1}}\Big)\,
\psi_\hbar\Big(y_2+\frac{w_2z_2z_0^{-1}}{1-z_2z_0^{-1}},
\frac 1 {1-z_2z_0^{-1}}\Big)\,
\psi_\hbar(x_0,z_0)
\end{align*}
and
\[
\mathbf \Lambda_6 \=
\left(\!\!
\begin{array}{ccccccc|c}
 \frac{(z_1-1)z_2}{z_1} & 0 & 0 &0 &0 &0 &0 &0 \\
 0 & \frac{(z_2-1)z_1}{z_2} & 0 & 0 &0 &0 &0 &0\\
 0 & 0 & \frac{{z_1}+{z_2}-z_1z_2}{({z_1}-1) ({z_2}-1)}&
 0 & 0& -1 & -1 & 0 \\
 0 & 0& 0& \frac{z_1}{(1-{z_1})z_2} & 1 & -1 & 0 & 0 \\
 0&0&0& 1 & \frac{z_2}{(1-{z_2})z_1} & 0 & -1 & 0 \\
 0&0&-1 & -1 & 0 & 0 &0 & -{\mathbf Q_1^*} \\
 0&0&-1 & 0 & -1 & 0 & 0 & -{\mathbf Q_2^*} \\
 \hline
 0&0&0 & 0 & 0 &
 -{\mathbf Q_1^*}^t & -{\mathbf Q_2^*}^t & {z^*}'-{\mathbf Q^*} \\
\end{array}
\!\!\right).
\]
We change variables $y_1\mapsto y_1-w_1\frac{z_1z_0^{-1}}{1-z_1z_0^{-1}}$
and $y_2\mapsto y_2-w_2\frac{z_2z_0^{-1}}{1-z_2z_0^{-1}}$ using
Lemma \ref{lem.ff1} to obtain that
\bea
\Phi^{\Xi}(\hbar) \=
\langle I_7 \rangle_{(y_1,y_2,x_0,w_1,w_2, x_1, x_2,x^*),
\mathbf \Lambda_7}
\eea
where
\begin{align*}
I_7 &\= \exp\Big(-\frac{\hbar}{24}
+\frac\hbar8 \mathbf f^t \mathbf B^{-1} \mathbf A \mathbf f
-\frac{\hsqrt}2 x^t
(\mathbf B^{-1}\boldsymbol{\nu}-1)
-(x_1+x_2) \frac {\hsqrt}{2}
+(y_1+y_2)\frac{\hsqrt}{2}\Big)\\
&\qquad\times\psi_\hbar\Big(y_1,
\frac 1 {1-z_1z_0^{-1}}\Big)\,
\psi_\hbar\Big(y_2,
\frac 1 {1-z_2z_0^{-1}}\Big)\,
\psi_\hbar(x_0,z_0)\,
\prod_{j=1}^{N} \psi_{\hbar}( x_j^*, z_j^*)
\end{align*}
and
\[
\mathbf \Lambda_7 \=
\left(\!\!
\begin{array}{ccccccc|c}
 \frac{(z_1-1)z_2}{z_1} & 0 & 0 &1 &0 &0 &0 &0 \\
 0 & \frac{(z_2-1)z_1}{z_2} & 0 & 0 &1 &0 &0 &0\\
 0 & 0 & \frac{{z_1}+{z_2}-z_1z_2}{({z_1}-1) ({z_2}-1)}&
 0 & 0& -1 & -1 & 0 \\
 1 & 0& 0& 0 & 1 & -1 & 0 & 0 \\
 0&1&0& 1 & 0 & 0 & -1 & 0 \\
 0&0&-1 & -1 & 0 & 0 &0 & -{\mathbf Q_1^*} \\
 0&0&-1 & 0 & -1 & 0 & 0 & -{\mathbf Q_2^*} \\
 \hline
 0&0&0 & 0 & 0 &
 -{\mathbf Q_1^*}^t & -{\mathbf Q_2^*}^t & {z^*}'-{\mathbf Q^*} \\
\end{array}
\!\!\right).
\]
Therefore, we can apply Fubini's Theorem (Lemma~\ref{lem.ff1}) with the
integration variables $w_1,w_2,x_1,x_2$,
to obtain that
\bea
\Phi^{\Xi}(\hbar) \=
\langle I_8 \rangle_{(y_1,y_2,x_0,x^*),
\mathbf \Lambda_8}
\eea
where
\begin{align*}
I_8 &\= \exp\Big(-\frac{\hbar}{4}(\mathbf B^{-1}\boldsymbol{\nu})_1
(\mathbf B^{-1}\boldsymbol{\nu})_2-\frac{\hbar}{24}
+\frac\hbar8 \mathbf f^t \mathbf B^{-1} \mathbf A \mathbf f
+\frac{\hsqrt}2\big(
x_0\left((\mathbf B^{-1}\boldsymbol{\nu})_1
 +(\mathbf B^{-1}\boldsymbol{\nu})_2\right)\big)\\
&\qquad\qquad\quad
-\frac{\hsqrt}2y_1\left((\mathbf B^{-1}\boldsymbol{\nu})_1-1\right)
-\frac{\hsqrt}2y_2\left((\mathbf B^{-1}\boldsymbol{\nu})_2-1\right)\\
&\qquad\qquad\quad
-\frac{\hsqrt}2{x^*}^t\left((\mathbf B^{-1}\boldsymbol{\nu})^*-
 (\mathbf B^{-1}\boldsymbol{\nu})_2\mbfQ_1^*
 -(\mathbf B^{-1}\boldsymbol{\nu})_1\mbfQ_2^*-1\right)\Big)\\
&\qquad\times\psi_\hbar\Big(y_1,
\frac 1 {1-z_1z_0^{-1}}\Big)\,
\psi_\hbar\Big(y_2,
\frac 1 {1-z_2z_0^{-1}}\Big)\,
\psi_\hbar(x_0,z_0)\,
\prod_{j=1}^{N} \psi_{\hbar}( x_j^*, z_j^*)
\end{align*}
and
\[
\mathbf \Lambda_8 \=
\left(\!\!
\begin{array}{ccc|c}
 \frac{({z_1}-1) {z_2}}{{z_1}} & 0 & -1 & -{\mbfQ_1^*} \\
 0 & \frac{({z_2}-1){z_1} }{{z_2}} & -1 & -{\mbfQ_2^*} \\
 -1 & -1 & \frac{{z_1}+{z_2}-z_2z_2}{({z_1}-1) ({z_2}-1)}+2 &
   {\mbfQ_1^*}+{\mbfQ_2^*} \\
   \hline
 -{\mbfQ_1^*}^t & -{\mbfQ_2^*}^t & {\mbfQ_1^*}^t+{\mbfQ_2^*}^t &
-{\mbfQ^*}+ {\mbfQ_1^*} {\mbfQ_2^*} + {\mbfQ_2^*} {\mbfQ_1^*}
\end{array}
\!\!\right).
\]
Using Lemma~\ref{lem.23.c2}, we obtain with respect to the integration variables
$\tilde x = (x_0,y_1,y_2,x^*)$ and some $c\in\frac 1 {24}\Z$ that
\bea
\Phi^{\Xi}(\hbar) \=
e^{c\hbar}\langle I_9 \rangle_{\tilde x, \mathbf \Lambda_9}
\eea
where with $\widetilde{\mathbf A}, \widetilde{\mathbf B}$ and
$\widetilde{\boldsymbol{\nu}}$ as in Equation~\eqref{manyAB.c2} and
\bea
I_9 \= \exp\Big(\frac\hbar8 \widetilde{\mathbf f}^t \widetilde{\mathbf B}^{-1}
\widetilde{\mathbf A} \widetilde{\mathbf f}
-\frac{\hsqrt}2 \tilde x^t(\widetilde {\mathbf B}^{-1}
\widetilde{\boldsymbol{\nu}}-1)\Big)\,
\prod_{j=0}^{N+2} \psi_{\hbar}( \widetilde{x}_j^*, \widetilde{z}_j^*)
\eea
and
\bea
\mathbf \Lambda_9
%&\= \diag\Big(\frac1{1-z_0},
%\frac{({z_1}-1) {z_2}}{{z_1}},
%\frac{({z_2}-1) {z_1}}{{z_2}},
%{z^*}' \Big)
%- {\widetilde {\mathbf B}}^{-1}\widetilde {\mathbf A}.\\
&\=\diag(\widetilde z_0',\,\widetilde z_1'',\,
\widetilde z_2'',\,\widetilde{z}^*{}')
- {\widetilde {\mathbf B}}^{-1}\widetilde {\mathbf A}.\\
\eea
%In particular, $(\widetilde A\;\widetilde B)$ and
%$\widetilde{\mathbf A}, \widetilde{\mathbf B}$
%are related via quad moves on the second and third column. Hence, we have
%\bea
%\Phi^{\Xi}(\hbar) \=
%\Phi^{\tilde \Xi}(\hbar)
%\eea
%which completes the proof of Proposition~\ref{prop.case2}.
\end{proof}

\begin{lemma}
\label{lem.23.c2}
With the notation used in the proof of the previous Proposition~\ref{prop.case2}
we have $\mbfQ_{11} = \mbfQ_{12} =\mbfQ_{22} = 0$ and
the following equalities: 
\begin{align}
\label{eq.c3.bia}
 \widetilde {\mathbf B}^{-1} \widetilde {\mathbf A} &\=
 \left(\!\!\begin{array}{ccc|c}
-1 & 1 & 1 & -\mbfQ_1^*-\mbfQ_2^*\\
1 & 0 & 0 & \mbfQ_1^*\\
1 & 0 & 0 & \mbfQ_2^*\\
\hline
-{\mbfQ_1^*}^t-{\mbfQ_2^*}^t & {\mbfQ_1^*}^t& {\mbfQ_2^*}^t &
\mbfQ^*-{\mbfQ_1^*}^t{\mbfQ_2^*}-{\mbfQ_2^*}^t{\mbfQ_1^*}
 \end{array}\!\!\right),\\\label{eq.c3.binu}
 \widetilde {\mathbf B}^{-1} \widetilde {\boldsymbol{\nu}} &\=
 \left(\!\!\begin{array}{c}
 -({\mathbf B}^{-1}\boldsymbol{\nu})_1
 -({\mathbf B}^{-1}\boldsymbol{\nu})_2+1\\
 ({\mathbf B}^{-1}\boldsymbol{\nu})_1\\
 {(\mathbf B}^{-1}\boldsymbol{\nu})_2\\
 ({\mathbf B}^{-1}\boldsymbol{\nu})^*
 -({\mathbf B}^{-1}\boldsymbol{\nu})_1\mbfQ_2^*
 -({\mathbf B}^{-1}\boldsymbol{\nu})_2\mbfQ_1^*
 \end{array}\!\!\right),\\\label{eq.c3.ff2}
\widetilde{\mathbf{f}}^{t}\widetilde {\mathbf B}^{-1}\widetilde{\boldsymbol{\nu}}
&\= \mathbf{f}^{t}\mathbf{B}^{-1}\boldsymbol{\nu}+\widetilde{f}_0-2
({\mathbf B}^{-1}\boldsymbol{\nu})_1({\mathbf B}^{-1}\boldsymbol{\nu})_2\,.
\end{align}
\end{lemma}

\begin{proof}
The relation $\mathbf B^{-1} \mathbf A = \mathbf Q$ and the fact that the
columns of $(a_1\,|\,a_2\,|\,b_*)$ are linearly independent imply
that $\mbfQ_{11} = \mbfQ_{12} =\mbfQ_{22} = 0$. We will prove
Equation~\eqref{eq.c3.bia} by showing the identity after multiplying by the
invertible matrix $\widetilde{\mathbf B}$.
For \eqref{eq.c3.bia} we compute
\bea[]
&\left(\!\!\begin{array}{ccc|c}
 -1 & -1 & -1 & 0\\
 0 & a_1-b_1-b_2 & a_2-b_1-b_2 & b_*
\end{array}\!\!\right)
 \left(\!\!\begin{array}{ccc|c}
-1 & 1 & 1 & -\mbfQ_1^*-\mbfQ_2^*\\
1 & 0 & 0 & \mbfQ_1^*\\
1 & 0 & 0 & \mbfQ_2^*\\
\hline
-{\mbfQ_1^*}^t-{\mbfQ_2^*}^t & {\mbfQ_1^*}^t& {\mbfQ_2^*}^t &
\mbfQ^*-{\mbfQ_1^*}^t{\mbfQ_2^*}-{\mbfQ_2^*}^t{\mbfQ_1^*}
 \end{array}\!\!\right)\\
&\=\left(\!\!
\begin{array}{ccc|c}
-1 & -1 & -1 & {0}\\
\hline
 \substack{a_1+a_2-2b_1-2b_2\\-b_*{\mbfQ_1^*}^t-b_*{\mbfQ_2^*}^t}&
{b_*{\mbfQ_1^*}^t} &
{b_*{\mbfQ_2^*}^t} &
\substack{(a_1-b_1){\mbfQ_1^*}^t + (a_2-b_2){\mbfQ_2^*}^t\\
- b_2{\mbfQ_1^*}^t - b_1{\mbfQ_2^*}^t
+b_*{\mbfQ^*}-b_*{\mbfQ_1^*}^t{\mbfQ_2^*} - b_*{\mbfQ_2^*}^t{\mbfQ_1^*}}
\end{array}
\!\!\right)\\
&\= \widetilde{\mathbf{A}}.
\eea
For Equation~\eqref{eq.c3.binu}, we compute
\bea[]
&\left(\!\!\begin{array}{ccc|c}
-1& -1& -1& 0\\
0 & a_1-b_1-b_2 & a_2-b_1-b_2 & b_*
\end{array}\!\!\right) \left(\!\!\begin{array}{c}
 -( {\mathbf B}^{-1}\boldsymbol{\nu})_1
 -( {\mathbf B}^{-1}\boldsymbol{\nu})_2+1\\
 ( {\mathbf B}^{-1}\boldsymbol{\nu})_1\\
 ( {\mathbf B}^{-1}\boldsymbol{\nu})_2\\
 ( {\mathbf B}^{-1}\boldsymbol{\nu})^{*}
 -( {\mathbf B}^{-1}\boldsymbol{\nu})_1\mbfQ_2^*
 -( {\mathbf B}^{-1}\boldsymbol{\nu})_2\mbfQ_1^*
 \end{array}\!\!\right)\\
 &\=\left(\!\!\begin{array}{c}
 -1\\ \boldsymbol\nu
 \end{array}\!\!\right)\,.
% \=\left(\!\!\begin{array}{c}
%  -1\\ \nu - b_1-b_2
% \end{array}\!\!\right)
\eea
For Equation~\eqref{eq.c3.ff2}, we note that for $i=1,2$ we have
\be
\widetilde{\mathbf{f}}^{*}{}^t\mbfQ_i^*
\=
(\widetilde {\mathbf B}^{-1}\widetilde{\boldsymbol{\nu}})_i
-\widetilde{\mathbf{f}}_0-\widetilde{\mathbf{f}}_i''
\ee
and so
\be
\begin{small}
\begin{aligned}
\widetilde{\mathbf{f}}^{t}\widetilde {\mathbf B}^{-1}\widetilde{\boldsymbol{\nu}}
\=
\widetilde{f}_0+
(\widetilde{f}_1'+\widetilde{f}_2)({\mathbf B}^{-1}\boldsymbol{\nu})_1
+(\widetilde{f}_1+\widetilde{f}_2')({\mathbf B}^{-1}\boldsymbol{\nu})_2
+\widetilde{f}^{*}{}^t({\mathbf B}^{-1}\boldsymbol{\nu})^{*}
-2({\mathbf B}^{-1}\boldsymbol{\nu})_1({\mathbf B}^{-1}\boldsymbol{\nu})_2
\end{aligned}
\end{small}
\ee
Then, using the analogous relations between the flattening as given in
Equation~\eqref{eq.23sr} for the shapes, we have
$\widetilde{f}_1'+\widetilde{f}_2=f_1'=\mathbf{f}_1$ and
$\widetilde{f}_1+\widetilde{f}_2'=f_2'=\mathbf{f}_2$.
\end{proof}

%%%%%%%%%%%%%%%%%%%%%%%%%%%%%%%%%%%%%%%%%%%%%%%%%%%%%%%%%%%%%%%%%%%%%%%%%%%% 
%%%%%%%%%%%%%%%%%%%%%%%%%%%%%%%%%%%%%%%%%%%%%%%%%%%%%%%%%%%%%%%%%%%%%%%%%%%%

\section{\texorpdfstring{The series of the simplest hyperbolic $4_1$ knot}{The series of the simplest hyperbolic 4\_1 knot}}
\label{sec.41}

In this section, we discuss an effective computation of the power series
$\Phi^{\Xi}(\hbar)$ for the simplest hyperbolic knot, namely the $4_1$ knot.
This example was studied extensively in~\cite{DG}. From~\cite[Ex. 2.6]{DG},
we obtain that the NZ datum $\Xi_{4_1}$ is given by
\be
A\=
\begin{pmatrix}
2 & 2\\
1 & 1
\end{pmatrix},\quad
B\=
\begin{pmatrix}
1 & 1\\
1 & 0
\end{pmatrix},\quad
\nu\=
\begin{pmatrix}
2\\
1
\end{pmatrix},\quad
f\=
\begin{pmatrix}
0\\
1
\end{pmatrix},\quad
f''\=
\begin{pmatrix}
0\\
0
\end{pmatrix},\quad
\ee
and $z_1=z_2=\z_6=e^{2\pi i/6}$. Therefore, we have
\be
\Phi^{\Xi_{4_1}}(\hbar) \= e^{\frac{\hbar}{8}}
\langle
\psi_{\hb}(x_1,\z_6)\psi_{\hb}(x_2,\z_6)\rangle_{(x_1,x_2),
\Lambda_0}\,,
\ee
where
\be
\Lambda_0
\=
\begin{pmatrix}
\z_6-1 & -1\\
-1 & \z_6-1
\end{pmatrix}\,.
\ee
This is a two dimensional formal Gaussian integral which can be simplified. 
Using $\z_6=1/(1-\z_6)=1-\z_6^{-1}$ and applying a change of coordinates
$x_1\mapsto x_1-\z_6x_2$, we find that
\be
\Phi^{\Xi_{4_1}}(\hbar) \= e^{\frac{\hbar}{8}}
\langle
\psi_{\hb}(x_1-\z_6x_2,\z_6)\psi_{\hb}(x_2,\z_6)\rangle_{(x_1,x_2),
\Lambda_1}\,,
\ee
where
\be
\Lambda_1
\=
\begin{pmatrix}
\z_6-1 & 0\\
0 & 2\z_6-1
\end{pmatrix}\,.
\ee
We can apply Fubini's theorem~\cite[Prop.2.13]{AarhusII} and
Corollary~\ref{cor.fourier} to perform the integral over $x_1$. After
renaming the variable $x_2$ by $x$, we express $\Phi^{\Xi_{4_1}}(\hbar)$ by a
one-dimensional formal Gaussian integral
\be
\label{eq.phi41.1d}
\Phi^{\Xi_{4_1}}(\hbar) \= e^{\frac{\hbar}{6}}
\Big\langle
\exp\Big(\frac{x}{2}\hbar^{\frac{1}{2}}\Big)\psi_{\hb}(x,\z_6)^2\Big\rangle_{x,
2\z_6-1}\,.
\ee
Using the definition of $\psi_\hbar$ from Equation~\eqref{psih} and expanding
to $O(\hbar^{5/2})$, we obtain that
\be
\begin{small}
\begin{aligned}
&\exp\Big(\frac{x}{2}\hbar^{\frac{1}{2}}\Big)\psi_{\hb}(x,\z_6)^2\\
&\=
1
+\bigg(\frac{1}{3}x^3 + \Big(\z_6 - \frac{1}{2}\Big)x\bigg)\hbar^{1/2}
+\bigg(\frac{1}{18}x^6 + \Big(\frac{1}{2}\z_6 - \frac{1}{4}\Big)x^4
- \frac{7}{8}x^2 + \Big(-\frac{1}{6}\z_6 + \frac{1}{6}\Big)\bigg)\hbar\\
&\qquad+ \bigg(\frac{1}{162}x^9 + \Big(\frac{1}{9}\z_6 - \frac{1}{18}\Big)x^7
- \frac{1}{2}x^5 + \Big(-\frac{73}{72}\z_6 + \frac{77}{144}\Big)x^3
+ \Big(\frac{1}{12}\z_6 + \frac{1}{4}\Big)x\bigg)\hbar^{3/2}\\
&\qquad+ \bigg(\frac{1}{1944}x^{12} + \Big(\frac{5}{324}\z_6
- \frac{5}{648}\Big)x^{10} - \frac{37}{288}x^8 + \Big(-\frac{1337}{2160}\z_6
+ \frac{1357}{4320}\Big)x^6\\
&\qquad+ \Big(\frac{1}{24}\z_6 + \frac{1027}{1152}\Big)x^4
+ \Big(\frac{23}{48}\z_6 - \frac{5}{16}\Big)x^2
- \frac{1}{72}\z_6\bigg)\hbar^{2}+O(\hbar^{5/2})\,.
\end{aligned}
\end{small}
\ee
Noting that $2\z_6-1=\sqrt{-3}$, we can then evaluate Equation~\eqref{eq.phi41.1d}
with Equation~\eqref{formalG1} to obtain that
\be
e^{-\frac{\hbar}{4}}\Phi^{\Xi_{4_1}}(\hbar) \=
1 + \frac{11}{72\sqrt{-3}}\hbar + \frac{697}{2(72\sqrt{-3})^2}\hbar^{2}+O(\hbar^{4})\,.
\ee
This is in agreement with computations in~\cite{DG,DGLZ,GZ:kashaev}.
The one-dimensional formal Gaussian integral~\eqref{eq.phi41.1d} gives an effective
computation of the series $\Phi^{\Xi_{4_1}}(\hbar)$.
%
% Pari program:
% topological-invariance-Phi-series/gau.int.4_1
%
Indeed, using a \texttt{pari-gp} program one can compute one
hundred coefficients in a few seconds and two hundred coefficients in a few minutes,
the first few of them are given by
\be
\begin{tiny}
\begin{aligned}
&e^{-\frac{\hbar}{4}}\Phi^{\Xi_{4_1}}(\hbar)\\
&=\;
1
- \frac{11}{216}\sqrt{-3}\,\hbar - \frac{697}{31104}\hbar^2
+ \frac{724351}{100776960}\sqrt{-3}\,\hbar^3
+ \frac{278392949}{29023764480}\hbar^4
- \frac{244284791741}{43883931893760}\sqrt{-3}\,\hbar^5\\
&\qquad- \frac{1140363907117019}{94789292890521600}\hbar^6
+ \frac{212114205337147471}{20474487264352665600}\sqrt{-3}\,\hbar^7
+ \frac{367362844229968131557}{11793304664267135385600}\hbar^8\\
&\qquad- \frac{44921192873529779078383921}{1260940134703442115428352000}\sqrt{-3}\,
\hbar^9
- \frac{3174342130562495575602143407}{23109593741473993679123251200}\hbar^{10}
+O(\hbar^{11})\,.
\end{aligned}
\end{tiny}
\ee
Similar to the case of the $4_1$, one can obtain one-dimensional formal Gaussian
integrals for the next two simplest hyperbolic knots, the $5_2$ and the
$(-2,3,7)$-pretzel knot, whose details we omit.

% \section*{\texorpdfstring{}{Acknowledgements}}
% \subsection*{Acknowledgements}
\newpage

\addcontentsline{toc}{section}{Acknowledgements}
\noindent \textbf{Acknowledgements.} The authors wish to thank Don Zagier for enlightening conversations.
The work of M.S. and C.W. has been supported by the Max-Planck-Gesellschaft.
C.W. wishes to thank the Southern University of Science and Technology's
International Center for Mathematics in
Shenzhen for their hospitality where the paper was completed.

%%%%%%%%%%%%%%%%%%%%%%%%%%%%%%%%%%%%%%%%%%%%%%%%%%%%%%%%%%%%%%%%%%%%%%%%%%%% 
%%%%%%%%%%%%%%%%%%%%%%%%%%%%%%%%%%%%%%%%%%%%%%%%%%%%%%%%%%%%%%%%%%%%%%%%%%%%

\appendix

\section{Complements on the Fourier transform}
\label{app.fourier}

In this appendix, we give the omitted details in the last step of the proof
of Theorem~\ref{thm.fourier}. They are an affine change of coordinates, followed
by the corresponding computation of the formal Gaussian integration.

The proof requires a version of formal Gaussian integration where the symmetric
matrix $\Lambda_\hbar \in \GL_N(\BQ(z)[x]\llbracket \hbar^{1/2} \rrbracket)$
depends on $\hbar$, such that $\Lambda_0$ is invertible. In this case, for an
integrable function $f_\hbar(x,z) \in \BQ(z)[x]\llbracket \hbar^{1/2}\rrbracket$,
we define 
\be
\label{FGIh}
\begin{aligned}
\llangle f_\hbar(x,z) \rrangle_{x,\Lambda_\hbar}
\;&:=\; \sqrt{\frac{\det(\Lambda_\hb)}{\det(\Lambda_0)}}
\big\langle
\exp\big(-\tfrac{1}{2} x^t(\Lambda_\hb-\Lambda_0)x \big)  f_\hbar(x,z)
\big\rangle_{x, \Lambda_0} \\
&\=
\frac{\int e^{-\frac12x^t \Lambda(\hb) \,x}f_{\hb}(x,z) \, dx}{
  \int e^{-\frac12x^t \Lambda(\hb) \,x} \, dx}
\in \BQ(z)\llbracket \hbar \rrbracket \,.
\end{aligned}
\ee
This version of formal Gaussian integration satisfies the properties of
Lemmas~\ref{lem.ff1} and~\ref{lem.ff2}.

We use Equations~\eqref{eq.psid} and Equation~\eqref{eq.fourier0} to obtain 
that
\bea
\psi_\hbar(x,z) \= e^{-\frac{\hbar}{24}} C_\hbar(x,z)\,
\Big\langle\!\!\Big\langle\exp\Big(\frac y 2 \hsqrt\Big)\,\psi_{\hbar}
\Big(y,\frac{1}{1-ze^{x\hbar^{1/2}}}\Big)\Big\rangle\!\!\Big\rangle_{y,1-z^{-1}e^{-x\hbar^{1/2}}}.
\eea
Lemma~\ref{lem.ff2} implies that 
\bea\label{eq.a1.eq1}
\psi_\hbar(x,z) &\=
 C_\hbar(x,z)\,
 \exp\Big(-\frac \hbar{24}-\frac {\hbar^{1/2}}2 (a+x)\Big)\\
&\qquad\times\Big\langle\exp\Big(\frac y 2 \hsqrt
+\Big({\frac{1}{ze^{x\hbar^{1/2}}}}-\frac1z\Big)\frac{y^2}2\Big)\,
\psi_{\hbar}\Big(y,\frac{1}{1-ze^{x\hbar^{1/2}}}\Big)
\Big\rangle_{y,1-{z}^{-1}},
\eea
where
$a=a_\hbar(x,z) \in \BQ(z)[x]\llbracket \hbar^{1/2} \rrbracket $ is given by
\bea
a &\;:=\;\frac 1 {\hbar^{1/2}} \log\Big(\frac {1-z}{1-ze^{x\hbar^{1/2}}}\Big)
\in \BQ(z)[x]\llbracket \hbar^{1/2} \rrbracket \,.
\eea
Similar to Equation~\eqref{eq.psid} we write
\bea
\psi_\hbar\Big(y,\frac{1}{1-ze^{x\hbar^{1/2}}}\Big)
\= \exp\Big(A_0 - (a(1-z^{-1})+x)y
-\Big({\frac{1}{ze^{x\hbar^{1/2}}}}-\frac1z\Big)\frac{y^2}2\Big)\,
\psi_\hbar\Big(y+a,\frac{1}{1-z}\Big)
\eea
where 
$A_0=A_{0,\hbar}(x,z) \in  \frac 1 \hbar \BQ(z)[x]\llbracket \hbar^{1/2} \rrbracket$
is given by
\bea
\label{eq.defA0}
A_0 &\= 
%%%%
\frac 1 2 \Big(
  \log\Big(\frac {-ze^{x\hbar^{1/2}}}{1-ze^{x\hbar^{1/2}}}\Big)
  -\log\Big(\frac {-z}{1-z}\Big)
\Big)
+ \frac 1 \hbar \Big(
\Li_2\Big(\frac 1 {1-ze^{x\hbar^{1/2}}}\Big)
-\Li_2\Big(\frac 1 {1-z}\Big)
\Big)\\
&\qquad+\frac{a^2}{2z} + \frac a {\hbar^{1/2}}\log\Big(\frac{-z}{1-z}\Big)\\
&\=
%%%%
\frac 1 2 a\hsqrt +\frac 1 2 x\hsqrt
+ \frac 1 \hbar \Big(
\Li_2\Big(\frac 1 {1-ze^{x\hbar^{1/2}}}\Big)
-\Li_2\Big(\frac 1 {1-z}\Big)
\Big)
+\frac{a^2}{2z} + \frac a {\hbar^{1/2}}\log\Big(\frac{-z}{1-z}\Big)\,.
\eea
Then Equation~\eqref{eq.a1.eq1} can be written as
\bea
\psi_\hbar(x,z) &\=
 C_\hbar(x,z)\,
 \exp\Big(-\frac \hbar{24}-\frac {\hbar^{1/2}}2 (a+x)+A_0\Big)\\
&\qquad\times\Big\langle\exp\Big(\frac y 2 \hsqrt
-(a(1-z^{-1}+x)y)
\Big)\,
\psi_{\hbar}\Big(y+a,\frac{1}{1-z}\Big)
\Big\rangle_{y,1-{z}^{-1}}.
\eea
We make the change of variables
\bea
w \mapsto w-a+ \frac{xz}{1-z}
\eea
and using Equation~\eqref{fga1} of Lemma~\ref{lem.ff1}, we obtain that
\bea
\psi_\hbar(x,z) &\=
 C_\hbar(x,z)\,
 \exp\Big(-\frac \hbar{24}-\frac {\hbar^{1/2}}2 (a+x)+A_0
 +\frac {a^2}2 - \frac{a^2}{2z}+ax + \frac{x^2}{2(1-z^{-1})}
 -\frac{\hbar^{1/2}} 2 a
 \Big)\\
&\qquad\times\Big\langle\exp\Big(
\frac {\hbar^{1/2}} 2
\Big(y+\frac{xz}{1-z}\Big)
\Big)\,
\psi_{\hbar}\Big(y+\frac{xz}{1-z},\frac{1}{1-z}\Big)
\Big\rangle_{y,1-{z}^{-1}}.
\eea
Hence, it remains to show that
\bea
1\= C_\hbar(x,z)\,
 \exp\Big(-\frac {\hbar^{1/2}}2 (a+x)+A_0
 +\frac {a^2}2 - \frac{a^2}{2z}+ax + \frac{x^2}{2(1-z^{-1})}
 -\frac{\hbar^{1/2}} 2 a
 \Big).
\eea
In other words, using the definitions of $C_\hbar(x,z)$ from
Equation~\eqref{eq.psid} and $A_0$ from Equation~\eqref{eq.defA0}
it suffices to prove that
\bea
0 \= &\frac 1 \hbar \Big(
\Li_2\Big(\frac 1 {1-ze^{x\hbar^{1/2}}}\Big)
-\Li_2\Big(\frac 1 {1-z}\Big)
\Big)
+\frac 1 \hbar \big(
\Li_2(z)
-\Li_2\big(ze^{x\hbar^{1/2}}\big)
\big)\\
&+\frac a {\hbar^{1/2}}\log\Big(\frac{-z}{1-z}\Big)
-\frac x {\hbar^{1/2}} \log(1-z)
+\frac{a^2}2
+ax.
\eea
With the transformation formula of the dilogarithm
\bea
\Li_2\Big(\frac{1}{1-z}\Big) \= 
\Li_2(z) - \frac{\pi^2}3 + \log(z)\log(1-z) - \frac 1 2 \log^2(z-1)
\eea
the right hand side of the previous equation is given by
\bea\label{eq.a1.exprhs}
&\frac 1 \hbar \Big(
\log\big(ze^{x\hbar^{1/2}}\big)\log\big(1-ze^{x\hbar^{1/2}}\big)
-\frac 1 2 \log^2\big(ze^{x\hbar^{1/2}}-1\big)\\
&\qquad -\log(z)\log(1-z)
+\frac 1 2 \log^2(z-1)
\Big)\\
&+\frac a {\hbar^{1/2}}\log\Big(\frac{-z}{1-z}\Big)
-\frac x {\hbar^{1/2}} \log(1-z)
+\frac{a^2}2
+ax.
\eea
With $x\hbar^{1/2} = {\log\big(e^{x\hbar^{1/2}}\big)}$ we compute
\bea[]
&\frac 1 \hbar \log\big(ze^{x\hbar^{1/2}}\big)
  \log\big(1-ze^{x\hbar^{1/2}}\big)
-\frac 1 \hbar\log(z)\log(1-z)
 -\frac x {\hbar^{1/2}} \log(1-z)
 + ax\\
\=&
\frac 1 \hbar \log\big(ze^{x\hbar^{1/2}}\big)
   \log\big(1-ze^{x\hbar^{1/2}}\big)
-\frac 1 \hbar \log(z)\log(1-z)
-\frac 1 \hbar \log(e^{x\hbar^{1/2}}) \log(1-z)
+\frac{a}{\hbar^{1/2}} \log\big(e^{x\hbar^{1/2}}\big)\\
\=& \frac 1 \hbar \log\big(ze^{x\hbar^{1/2}}\big)
\Big(
\log\big(1-ze^{x\hbar^{1/2}}\big) - \log(1-z)
\Big)
+\frac{a}{\hbar^{1/2}} \log\big(e^{x\hbar^{1/2}}\big)\\
\=&- \frac a {\hbar^{1/2}} \log\big(ze^{x\hbar^{1/2}}\big)
+\frac{a}{\hbar^{1/2}} \log\big(e^{x\hbar^{1/2}}\big)\\
\=&-\frac a {\hbar^{1/2}} \log(z)
\eea
so that Equation~\eqref{eq.a1.exprhs} becomes
\bea[]
&-\frac 1 {2\hbar} \log^2\big(ze^{x\hbar^{1/2}}-1\big)
+\frac 1 {2\hbar} \log^2(z-1)
+\frac a {\hbar^{1/2}}\log\Big(\frac{-z}{1-z}\Big)
+\frac{a^2}2
-\frac a {\hbar^{1/2}} \log(z)\\
\=&-\frac 1 {2\hbar} \log^2\big(ze^{x\hbar^{1/2}}-1\big)
+\frac 1 {2\hbar} \log^2(z-1)
-\frac a {\hbar^{1/2}}\log({z-1})
+\frac{a^2}2\\
\=
&-\frac 1 {2\hbar} \log^2\big(ze^{x\hbar^{1/2}}-1\big)
+ \frac 1 2\Big(\frac 1 {\hbar^{1/2}}\log(z-1)-a\Big)^2\\
\=&-\frac 1 {2\hbar} \log^2\big(ze^{x\hbar^{1/2}}-1\big)
+ \frac 1 {2\hbar} \log^2(ze^{x\hbar^{1/2}}-1)\\
\=&0,
\eea
which completes the proof of the last step of Theorem~\ref{thm.fourier}.

%%%%%%%%%%%%%%%%%%%%%%%%%%%%%%%%%%%%%%%%%%%%%%%%%%%%%%%%%%%%%%%%%%%%%%%%%%%%
%%%%%%%%%%%%%%%%%%%%%%%%%%%%%%%%%%%%%%%%%%%%%%%%%%%%%%%%%%%%%%%%%%%%%%%%%%%%

\section{Complements on the pentagon identity}
\label{app.pentagon}

In this appendix, we give the omitted details in the last step of the proof
of Theorem~\ref{thm.pentagon}. They are an affine change of coordinates, followed
by the corresponding computation of the formal Gaussian integration.

Equations~\eqref{pentagon0} and Equation~\eqref{eq.psid} give
\begin{small}
\be
\psi_\hbar(x,z_1)\psi_\hbar(y,z_2)
\=
e^{-\frac{\hbar}{24}}
C_\hbar(x,z_1) C_\hbar(y,\tz_1)\,
\llangle
\psi_\hbar(-w,\tz_1\tz_0^{-1})
\psi_\hbar(w,\tz_0)
\psi_\hbar(-w,\tz_2\tz_0^{-1})
\rrangle_{w,\tdelta}\,,
\ee
\end{small}
where
\be
\label{eq.aB.tdelta}
\begin{aligned}
\tz_1 & \= z_1e^{x\hbar^{1/2}} & 
\tz_2 & \= z_2e^{y\hbar^{1/2}} \\   
\tz_0 & \= \tz_1+\tz_1-\tz_1\tz_2, &  
\tdelta & \= \frac{(\tz_2+\tz_2-\tz_1\tz_2)^2}{\tz_1\tz_2(1-\tz_1)(1-\tz_2)} \,.
\end{aligned}
\ee
Note that $\tz_1$, $\tz_2$, $\tz_0$ and $\tdelta$ are power series in $\hb^{1/2}$
which, when evaluated at $\hb=0$, coincide with $z_1$, $z_2$,
$z_0$ and $\delta$ given in Equations~\eqref{5z} and~\eqref{eq.delta23}. 

We apply Lemma~\ref{lem.ff2} to obtain that
\bea
\psi_\hbar(x,z_1)\psi_\hbar(y,z_2)
\=
&e^{-\frac{\hbar}{24}}
C_\hbar(x,z_1) C_\hbar(y,\tz_1)
%\sqrt{\frac{\tdelta}{\delta}}
\exp\Big(\frac 1 2 (\log \tdelta - \log \delta)\Big)\\
&\times \Big\langle
\exp\Big(\frac {w^2}2(\delta -\tdelta)\Big)
\psi_\hbar(-w,\tz_1\tz_0^{-1})
\psi_\hbar(w,\tz_0)
\psi_\hbar(-w,\tz_2\tz_0^{-1})
\Big\rangle_{w,\delta}\,.
\eea
where
$a=a_\hbar(x,z) \in \BQ(z)[x]\llbracket \hbar^{1/2} \rrbracket $ is given by
\bea
a &\;:=\; \frac 1 {\hbar^{1/2}} \log(z_0\tz_0^{-1}).
\eea
We write similarly to 
Equation~\eqref{eq.psid}
\bea[]
&\psi_\hbar(-w,\tz_1\tz_0^{-1})
\psi_\hbar(w,\tz_0)
\psi_\hbar(-w,\tz_2\tz_0^{-1})\\
\=
&\exp(A_0 
+ w((x+y) + 2a 
+ \Li_0(z_1z_0^{-1})(a+x)
+ \Li_0(z_0)a
+ \Li_0(z_2z_0^{-1})(a+y))
+\frac {w^2}{2} (\tdelta - \delta)
)\\
&\times\psi_\hbar(-w+a+x,\tz_1\tz_0^{-1})
\psi_\hbar(w-a,\tz_0)
\psi_\hbar(-w+a+y,\tz_2\tz_0^{-1})
\eea
where
$A_0=A_{0,\hbar}(x,z) \in  \frac 1 {\hbar}\BQ(z)[x]\llbracket \hbar^{1/2} \rrbracket$
is given by 
\bea
\label{eq.ab.A0def}
A_0 \;:=\;
&\frac 1 \hbar (
\Li_2(\tz_1\tz_0^{-1}) + \Li_2(\tz_0) + \Li_2(\tz_2\tz_0^{-1})
-\Li_2(z_1z_0^{-1}) + \Li_2(z_0) - \Li_2(z_2z_0^{-1})
)\\
&+\frac 1 2(
\log(1-\tz_1\tz_0^{-1}) + \log(1-\tz_0) + \log(1-\tz_2\tz_0^{-1})\\
&\qquad-\log(1-z_1\tz_0^{-1}) + \log(1-z_0) + \log(1-z_2z_0^{-1})
)\\
&+\frac 1 {\hbar^{1/2}}(
\log(1-z_1z_0^{-1})(a+x)
-\log(1-z_0)a
+\log(1-z_2z_0^{-1})(a+y)
)\\
&-\frac 1 2(
\Li_0(z_1z_0^{-1})(a+x)^2 + \Li_0(z_0)a^2 + \Li_0(z_2z_0^{-1})(a+y)^2
)
\eea
Hence, $\psi_\hbar(x,z_1)\psi_\hbar(y,z_2)$ can be written as
\bea[]
&e^{-\frac{\hbar}{24}}
C_\hbar(x,z_1) C_\hbar(y,\tz_1)
\exp\Big(\frac 1 2 (\log (\tdelta) - \log (\delta))+A_0\Big)\\
&\times \Big\langle
\exp\Big(w((u+v) + 2a 
 + \Li_0(z_1z_0^{-1})(a+u)
 + \Li_0(z_0)a
 + \Li_0(z_2z_0^{-1})(a+v))\Big)\\
&\qquad\times
\psi_\hbar(-w+a+x,\tz_1\tz_0^{-1})
\psi_\hbar(w-a,\tz_0)
\psi_\hbar(-w,\tz_2\tz_0^{-1})
\Big\rangle_{w+a+y,\delta}\,.
\eea
With the change of variables
\bea
w \mapsto w+a+x+y - \frac{xz_1+yz_2}{z_0}
\eea
combined with Equation~\eqref{fga1} of Lemma~\ref{lem.ff1}, we obtain that 
\be
\label{eq.aB.fin0}
\begin{tiny}
\begin{aligned}
&e^{-\frac{\hbar}{24}}
C_\hbar(x,z_1) C_\hbar(y,\tz_1)
\exp\Big(\frac 1 2 (\log (\tdelta) - \log (\delta))+A_0\Big)\\
&\exp\Big(
\frac \delta 2\Big(a+x+y - \frac{xz_2+yz_1}{z_0}\Big)^2
 + \Big(a+x+y-\frac{xz_2+yz_1}{z_0}\Big)
 (x+y+\delta a
  +\Li_0(z_1z_0^{-1}) x +\Li_0(z_2z_0^{-1}) y)
\Big)\\
&\times \Big\langle
\psi_\hbar\Big(-w-y+\frac{xz_2+yz_1}{z_0},\tz_1\tz_0^{-1}\Big)
\psi_\hbar\Big(w+x+y+\frac{xz_2+yz_1}{z_0},\tz_0\Big)
\psi_\hbar\Big(-w-x+\frac{xz_2+yz_1}{z_0},\tz_2\tz_0^{-1}\Big)
\Big\rangle_{w,\delta}\,.
\end{aligned}
\end{tiny}
\ee
Hence, in order to prove Equation~\eqref{eq.pentagon} it remains to prove
that the term in front of the bracket simplifies to $e^{-\frac \hbar{24}}$.
For this, we use the definitions of $C_\hbar$ \eqref{Cdef}
and $A_0$ \eqref{eq.ab.A0def} to obtain 
\bea\label{eq.aB.expand}
&\log(C_\hbar(x,z_1))+\log(C_\hbar(y,\tz_1))
+\frac 1 2 (\log (\tdelta) - \log (\delta))+A_0
+\frac \delta 2\Big(a+x+y - \frac{xz_2+yz_1}{z_0}\Big)^2\\
&+ \Big(a+x+y-\frac{xz_2+yz_1}{z_0}\Big)
 (x+y+\delta a
  +\Li_0(z_1z_0^{-1}) x +\Li_0(z_2z_0^{-1}) y))\\
\=&
%%dilog terms (blue)
\frac 1 \hbar (
-\Li_2(\tz_1) + \Li_2(z_1) - \Li_2(\tz_2) + \Li_2(z_2)
+ \Li_2(\tz_1\tz_0^{-1}) - \Li_2(z_1z_0^{-1}) \\
&\qquad+ \Li_2(\tz_0^{-1}) - \Li_2(z_0) 
+ \Li_2(\tz_2\tz_0^{-1}) - \Li_2(z_2z_0^{-1}) 
)\\
%%constant terms (yellow)
&+\frac 1 2(
%from C_h
-\log(1-\tz_1) + \log(1-z_1)
-\log(1-\tz_2) + \log(1-z_2)
%from inttrafo
+\log \tdelta - \log \delta\\
%from A_0
&\qquad+\log(1-\tz_1\tz_0^{-1}) - \log(1-z_1z_0^{-1})
+\log(1-\tz_0) - \log(1-z_0)\\
&\qquad+\log(1-\tz_2\tz_0^{-1}) - \log(1-z_2z_0^{-1}))\\
%%pink
&+\frac 1 {\hbar^{1/2}} (
 -\log(1-z_1)x - \log(1-z_2)y
 +(x+y)(\log(z_0)-\log(\tz_o))\\
 &\qquad+\log(1-z_1z_0^{-1}) (a+x)
 -\log(1-z_0) a
 +\log(1-z_2z_0^{-1}) (a+y)
)\\
%%blue
&+\frac {a^2} 2 (\delta - \Li_0(z_1z_0^{-1})-\Li_0(z_0)-\Li_0(z_2z_0^{-1}))\\
%%orange
&+\Big(x+y-\frac{xz_2+yz_2}{z_0}\Big)
\Big(x+y + \Li_0(z_1z_0^{-1})x+ \Li_0(z_2z_0^{-1})y
- \frac \delta 2 \Big(
x+y - \frac{xz_2+yz_2}{z_0}
\Big)\Big)\\
&\qquad-\Li_0(z_1z_0^{-1})\frac {x^2}2 - \Li_0(z_2z_0^{-1})\frac {y^2}2.
\eea 
Inserting the definitions of $\delta$ \eqref{eq.delta23} and $\tdelta$
\eqref{eq.aB.tdelta} and using the relations
\bea\label{eq.aB.rel}
1-z_1z_0^{-1} &\= z_2(1-z_1)z_0^{-1},\\
1-z_2z_0^{-1} &\= z_1(1-z_2)z_0^{-1},\\
1-z_0 &\= (1-z_1)(1-z_2),
\eea
and similar ones for $\tz_1, \tz_2$ and $\tz_0$ we obtain that the terms
\bea[]
&\frac 1 2(
-\log(1-\tz_1) + \log(1-z_1)
-\log(1-\tz_2) + \log(1-z_2)
+\log \tdelta - \log \delta\\
&\qquad+\log(1-\tz_1\tz_0^{-1}) - \log(1-z_1z_0^{-1})
+\log(1-\tz_0) - \log(1-z_0)\\
&\qquad+\log(1-\tz_2\tz_0^{-1}) - \log(1-z_2z_0^{-1}))\\
\eea
vanish.
Furthermore, we have
%%blue
\bea
\delta - \Li_0(z_1z_0^{-1})-\Li_0(z_0)-\Li_0(z_2z_0^{-1}) \= 2
\eea
as well as
%See mathematica (orange equation)
\bea[]
&\Big(x+y-\frac{xz_2+yz_2}{z_0}\Big)
\Big(x+y + \Li_0(z_1z_0^{-1})x+ \Li_0(z_2z_0^{-1})y - \frac \delta 2
\Big(
x+y - \frac{xz_2+yz_2}{z_0}
\Big)\Big)\\
&-\Li_0(z_1z_0^{-1})\frac {x^2}2 - \Li_0(z_2z_0^{-1})\frac {y^2}2
\= xy.
\eea
Therefore, Equation~\eqref{eq.aB.expand} simplifies to
\bea\label{eq.aB.expand3}
&\frac 1 \hbar (
-\Li_2(\tz_1) + \Li_2(z_1) - \Li_2(\tz_2) + \Li_2(z_2)
+ \Li_2(\tz_1\tz_0^{-1}) - \Li_2(z_1z_0^{-1}) \\
&\qquad+ \Li_2(\tz_0^{-1}) - \Li_2(z_0) 
+ \Li_2(\tz_2\tz_0^{-1}) - \Li_2(z_2z_0^{-1}) 
)\\
%%pink
&+\frac 1 {\hbar^{1/2}} (
 -\log(1-z_1)x -  \log(1-z_2)y
 +(x+y)(\log(z_0)-\log(\tz_o))\\
 &\qquad+\log(1-z_1z_0^{-1}) (a+x)
 -\log(1-z_0) a
 +\log(1-z_2z_0^{-1}) (a+y)
)\\
&+a^2+xy.
\eea
Using the definition of $a$ and the relations from Equation~\eqref{eq.aB.rel}
we compute
\bea[]
&\frac 1 {\hbar^{1/2}} (
 -\log(1-z_1)x -  \log(1-z_2)y
 +(x+y)(\log(z_0)-\log(\tz_o))\\
 &\qquad+\log(1-z_1z_0^{-1}) (a+x)
 -\log(1-z_0) a
 +\log(1-z_2z_0^{-1}) (a+y)
)\\
&+a^2+xy\\
\=&\frac 1 {\hbar^{1/2}}(x(\log(z_2) - \log(\tz_0)) + y(\log(z_1) - \log(\tz_0))\\
&+\frac 1 \hbar
(\log(z_0)(\log(z_1)+\log(z_2))
+\log^2(z_0)
-\log(\tz_0)(\log(z_1)+\log(z_2)))\\
\=&\frac 1\hbar (
\log(z_2)\log(z_0) + \log(z_1)\log(z_0) - \log(z_1)\log(z_2)\\
&\qquad-\log(\tz_2)\log(\tz_0) +\log(\tz_1)\log(\tz_2) - \log(\tz_1)\log(\tz_0)
).
\eea
%Hence, we obtain that the expression inEquation~\eqref{eq.aB.expand} is equal to
Using the 5--term relation of the dilogarithm we obtain that the expression
in Equation~\eqref{eq.aB.expand3}
\be
\begin{tiny}
\begin{aligned}
&
\Li_2(z_1) + \Li_2(z_2) - \Li_2(z_1z_0^{-1}) - \Li_2(z_0) - \Li_2(z_2z_0^{-1})
+\log(z_2)\log(z_0) + \log(z_1)\log(z_0) - \log(z_1)\log(z_2)\\
&-\Li_2(\tz_1) - \Li_2(\tz_2) + \Li_2(\tz_1\tz_0^{-1}) + \Li_2(\tz_0^{-1})
+ \Li_2(\tz_2\tz_0^{-1})
-\log(\tz_2)\log(\tz_0) - \log(\tz_1)\log(\tz_0)+\log(\tz_1)\log(\tz_2).
\end{aligned}
\end{tiny}
\ee
vanishes.
In particular, the terms in front of the bracket in
Equation~\eqref{eq.aB.fin0} simplify to $e^{-\frac\hbar{24}}$
which concludes the proof of the last step of Theorem~\ref{thm.pentagon}.

%%%%%%%%%%%%%%%%%%%%%%%%%%%%%%%%%%%%%%%%%%%%%%%%%%%%%%%%%%%%%%%%%%%%%%%%%%%%
%%%%%%%%%%%%%%%%%%%%%%%%%%%%%%%%%%%%%%%%%%%%%%%%%%%%%%%%%%%%%%%%%%%%%%%%%%%%

%\bibliographystyle{hamsalpha}
\bibliographystyle{plain}
\bibliography{biblio}

\begin{thebibliography}{10}

\bibitem{AK:teichmuller}
J\o rgen~Ellegaard Andersen and Rinat Kashaev.
\newblock The {T}eichm\"{u}ller {TQFT}.
\newblock In {\em Proceedings of the {I}nternational {C}ongress of
  {M}athematicians---{R}io de {J}aneiro 2018. {V}ol. {III}. {I}nvited
  lectures}, pages 2541--2565. World Sci. Publ., Hackensack, NJ, 2018.

\bibitem{AK}
J{\o}rgen~Ellegaard Andersen and Rinat Kashaev.
\newblock A {TQFT} from {Q}uantum {T}eichm\"uller theory.
\newblock {\em Comm. Math. Phys.}, 330(3):887--934, 2014.

\bibitem{AS}
Scott Axelrod and Isadore Singer.
\newblock Chern-{S}imons perturbation theory. {II}.
\newblock {\em J. Differential Geom.}, 39(1):173--213, 1994.

\bibitem{AarhusII}
Dror Bar-Natan, Stavros Garoufalidis, Lev Rozansky, and Dylan Thurston.
\newblock The {Å}rhus integral of rational homology 3-spheres. {II}.
  {I}nvariance and universality.
\newblock {\em Selecta Math. (N.S.)}, 8(3):341--371, 2002.

\bibitem{BIZ}
David Bessis, Claude Itzykson, and Jean-Bernard Zuber.
\newblock Quantum field theory techniques in graphical enumeration.
\newblock {\em Adv. in Appl. Math.}, 1(2):109--157, 1980.

\bibitem{BoydII}
David Boyd, Nathan Dunfield, and Fernando Rodriguez-Villegas.
\newblock Mahler's measure and the dilogarithm {(II)}.
\newblock Preprint 2003,
  \href{https://arxiv.org/abs/math/0308041}{arXiv:0308041}.

\bibitem{regina}
Benjamin Burton.
\newblock {R}egina: Normal surface and 3-manifold topology software.
\newblock \url{http://regina.sourceforge.net}.

\bibitem{snappy}
Marc Culler, Nathan Dunfield, and Jeffrey Weeks.
\newblock Snap{P}y, a computer program for studying the topology of
  $3$-manifolds.
\newblock Available at \url{http://snappy.computop.org} (30/01/2015).

\bibitem{LRS}
Jes{\'u}s~A. De~Loera, J{\"o}rg Rambau, and Francisco Santos.
\newblock {\em Triangulations}, volume~25 of {\em Algorithms and Computation in
  Mathematics}.
\newblock Springer-Verlag, Berlin, 2010.
\newblock Structures for algorithms and applications.

\bibitem{Dimofte:complexCS}
Tudor Dimofte.
\newblock Perturbative and nonperturbative aspects of complex {C}hern-{S}imons
  theory.
\newblock {\em J. Phys. A}, 50(44):443009, 25, 2017.

\bibitem{DG}
Tudor Dimofte and Stavros Garoufalidis.
\newblock The quantum content of the gluing equations.
\newblock {\em Geom. Topol.}, 17(3):1253--1315, 2013.

\bibitem{DGLZ}
Tudor Dimofte, Sergei Gukov, Jonatan Lenells, and Don Zagier.
\newblock Exact results for perturbative {C}hern-{S}imons theory with complex
  gauge group.
\newblock {\em Commun. Number Theory Phys.}, 3(2):363--443, 2009.

\bibitem{EP}
David B.~A. Epstein and Robert~C. Penner.
\newblock Euclidean decompositions of noncompact hyperbolic manifolds.
\newblock {\em J. Differential Geom.}, 27(1):67--80, 1988.

\bibitem{Etingof}
Pavel Etingof, Oleg Golberg, Sebastian Hensel, Tiankai Liu, Alex Schwendner,
  Dmitry Vaintrob, and Elena Yudovina.
\newblock {\em Introduction to representation theory}, volume~59 of {\em
  Student Mathematical Library}.
\newblock American Mathematical Society, Providence, RI, 2011.
\newblock With historical interludes by Slava Gerovitch.

\bibitem{Faddeev}
Ludwig Faddeev.
\newblock Discrete {H}eisenberg-{W}eyl group and modular group.
\newblock {\em Lett. Math. Phys.}, 34(3):249--254, 1995.

\bibitem{FK:QDL}
Ludwig Faddeev and Rinat Kashaev.
\newblock Quantum dilogarithm.
\newblock {\em Modern Phys. Lett. A}, 9(5):427--434, 1994.

\bibitem{Fock-Chekhov}
V.~V. Fok and Leonid Chekhov.
\newblock Quantum {T}eichm\"{u}ller spaces.
\newblock {\em Teoret. Mat. Fiz.}, 120(3):511--528, 1999.

\bibitem{GHRS}
Stavros Garoufalidis, Craig Hodgson, Hyam Rubinstein, and Henry Segerman.
\newblock 1-efficient triangulations and the index of a cusped hyperbolic
  3-manifold.
\newblock {\em Geom. Topol.}, 19(5):2619--2689, 2015.

\bibitem{GY:1loop.torsion}
Stavros Garoufalidis and Seokbeom Yoon.
\newblock 1-loop equals torsion for fibered 3-manifolds.
\newblock Preprint 2023,
  \href{https://arxiv.org/abs/2304.00469}{arXiv:2304.00469}.

\bibitem{GZ:qseries}
Stavros Garoufalidis and Don Zagier.
\newblock Knots and their related $q$-series.
\newblock Preprint 2023,
  \href{https://arxiv.org/abs/2304.09377}{arXiv:2304.09377}.

\bibitem{GZ:kashaev}
Stavros Garoufalidis and Don Zagier.
\newblock Knots, perturbative series and quantum modularity.
\newblock Preprint 2021,
  \href{https://arxiv.org/abs/2111.06645}{arXiv:2111.06645}.

\bibitem{GZ:asymptotics}
Stavros Garoufalidis and Don Zagier.
\newblock Asymptotics of {N}ahm sums at roots of unity.
\newblock {\em Ramanujan J.}, 55(1):219--238, 2021.

\bibitem{Gukov:Apoly}
Sergei Gukov.
\newblock Three-dimensional quantum gravity, {C}hern-{S}imons theory, and the
  {A}-polynomial.
\newblock {\em Comm. Math. Phys.}, 255(3):577--627, 2005.

\bibitem{Hikami}
Kazuhiro Hikami.
\newblock Generalized volume conjecture and the {$A$}-polynomials: the
  {N}eumann-{Z}agier potential function as a classical limit of the partition
  function.
\newblock {\em J. Geom. Phys.}, 57(9):1895--1940, 2007.

\bibitem{K95}
Rinat Kashaev.
\newblock A link invariant from quantum dilogarithm.
\newblock {\em Modern Phys. Lett. A}, 10(19):1409--1418, 1995.

\bibitem{Kashaev:quantization}
Rinat Kashaev.
\newblock Quantization of {T}eichm\"{u}ller spaces and the quantum dilogarithm.
\newblock {\em Lett. Math. Phys.}, 43(2):105--115, 1998.

\bibitem{Kashaev:beta}
Rinat Kashaev.
\newblock Beta pentagon relations.
\newblock {\em Theoret. and Math. Phys.}, 181(1):1194--1205, 2014.
\newblock Russian version appears in Teoret. Mat. Fiz. {{\bf{1}}81}, (2014),
  no. 1, 73--85.

\bibitem{Matveev}
Sergei Matveev.
\newblock Transformations of special spines, and the {Z}eeman conjecture.
\newblock {\em Izv. Akad. Nauk SSSR Ser. Mat.}, 51(5):1104--1116, 1119, 1987.

\bibitem{NZ}
Walter Neumann and Don Zagier.
\newblock Volumes of hyperbolic three-manifolds.
\newblock {\em Topology}, 24(3):307--332, 1985.

\bibitem{Oesterle}
Joseph Oesterl\'{e}.
\newblock Polylogarithmes.
\newblock Number 216, pages Exp. No. 762, 3, 49--67. 1993.
\newblock S\'{e}minaire Bourbaki, Vol. 1992/93.

\bibitem{PWZ}
Marko Petkov{\v{s}}ek, Herbert Wilf, and Doron Zeilberger.
\newblock {\em {$A=B$}}.
\newblock A K Peters Ltd., Wellesley, MA, 1996.
\newblock With a foreword by Donald E. Knuth, With a separately available
  computer disk.

\bibitem{Piergallini}
Riccardo Piergallini.
\newblock Standard moves for standard polyhedra and spines.
\newblock Number~18, pages 391--414. 1988.
\newblock Third National Conference on Topology (Italian) (Trieste, 1986).

\bibitem{RT}
Nicolai Reshetikhin and Vladimir Turaev.
\newblock Invariants of {$3$}-manifolds via link polynomials and quantum
  groups.
\newblock {\em Invent. Math.}, 103(3):547--597, 1991.

\bibitem{Th}
William Thurston.
\newblock {\em The geometry and topology of 3-manifolds}.
\newblock Universitext. Springer-Verlag, Berlin, 1977.
\newblock Lecture notes, Princeton.

\bibitem{Singer:galois-differential}
Marius van~der Put and Michael Singer.
\newblock {\em Galois theory of linear differential equations}, volume 328 of
  {\em Grundlehren der mathematischen Wissenschaften [Fundamental Principles of
  Mathematical Sciences]}.
\newblock Springer-Verlag, Berlin, 2003.

\bibitem{WZ}
Herbert Wilf and Doron Zeilberger.
\newblock An algorithmic proof theory for hypergeometric (ordinary and
  ``{$q$}'') multisum/integral identities.
\newblock {\em Invent. Math.}, 108(3):575--633, 1992.

\bibitem{Witten:Jones}
Edward Witten.
\newblock Quantum field theory and the {J}ones polynomial.
\newblock {\em Comm. Math. Phys.}, 121(3):351--399, 1989.

\bibitem{Witten:complexCS}
Edward Witten.
\newblock Quantization of {C}hern-{S}imons gauge theory with complex gauge
  group.
\newblock {\em Comm. Math. Phys.}, 137(1):29--66, 1991.

\bibitem{Zagier:dilog}
Don Zagier.
\newblock The dilogarithm function.
\newblock In {\em Frontiers in number theory, physics, and geometry. {II}},
  pages 3--65. Springer, Berlin, 2007.

\end{thebibliography}
\end{document}